\documentclass[11pt]{article}
\usepackage[a4paper,margin=29mm]{geometry}
\usepackage{amsmath,amssymb,amsthm,mathtools,mathrsfs}
\usepackage{enumitem,microtype,booktabs,array}
\usepackage{xcolor}
\usepackage[colorlinks=true,linkcolor=blue!55!black,citecolor=blue!55!black,urlcolor=blue!55!black]{hyperref}
\usepackage[nameinlink,capitalise]{cleveref}

\newcommand{\cO}{\mathcal O}
\newcommand{\F}{\mathbb F}
\newcommand{\Cp}{\mathbb C_p}
\newcommand{\Iw}{\mathrm{Iw}}
\newcommand{\ur}{\mathrm{ur}}
\newcommand{\new}{\mathrm{new}}
\newcommand{\midd}{\mathrm{mid}}
\newcommand{\NP}{\operatorname{NP}}
\newcommand{\Dig}{\operatorname{Dig}}
\newcommand{\GL}{\operatorname{GL}}
\newcommand{\val}{v_p}
\newcommand{\eps}{\varepsilon}
\newcommand{\nS}{\mathrm{nS}}
\newtheorem{theorem}{Theorem}[section]
\newtheorem*{introconjecture}{Conjecture}
\newtheorem*{introtheorem}{Theorem}
\newtheorem{proposition}[theorem]{Proposition}
\newtheorem{corollary}[theorem]{Corollary}
\newtheorem{lemma}[theorem]{Lemma}
\newtheorem{lemmanotation}[theorem]{Lemma--Notation}
\crefname{lemmanotation}{Lemma--Notation}{Lemmas--Notation}
\Crefname{lemmanotation}{Lemma--Notation}{Lemmas--Notation}
\theoremstyle{definition}
\newtheorem{definition}[theorem]{Definition}
\newtheorem{notation}[theorem]{Notation}
\crefname{notation}{Notation}{Notations}
\Crefname{notation}{Notation}{Notations}

\theoremstyle{remark}
\newtheorem{remark}[theorem]{Remark}

\title{The Local Ghost Theorem in the Tr\`es Ramifi\'ee Case}
\author{Liyan Wang}
\date{}

\begin{document}
\maketitle

\begin{abstract}
Let \(p\geq11\) and let
\(\bar\rho\cong\left(\begin{smallmatrix}\omega&*\\0&1\end{smallmatrix}\right)\)
be a tr\`es ramifi\'ee nonsplit extension.  For a primitive
projective-augmented module of the corresponding type, We prove the tr\`es ramifi\'ee analogues of the local results of
Liu--Truong--Xiao--Zhao and thereby establish the local ghost conjecture in
this case.  The main difference here is that the Iwahori dimension increases by one between consecutive relevant classical weights and may therefore be odd. Accordingly, the proof
of the near-Steinberg results must treat both integral and half-integral centres.
\end{abstract}

\tableofcontents

\section{Introduction}

The \(U_p\)-slopes of an overconvergent modular form are the \(p\)-adic
valuations of its \(U_p\)-eigenvalues. Describing these slopes as the weight
varies is a basic problem in the study of \(p\)-adic families of modular
forms. The ghost conjecture of Bergdall and Pollack predicts them from an
explicit power series whose coefficients are products of classical weight
parameters \cite{BP}. The essential feature of this prediction is that the
series is constructed from dimensions of classical spaces, rather than
from the matrix of \(U_p\).

Liu, Truong, Xiao, and Zhao formulated a local analogue in which the
automorphic input is replaced by a projective-augmented module
\cite[Definition~2.22 and Conjecture~2.29]{LTXZ1}. We recall the formulation
in some detail, both to explain the combinatorial nature of the ghost
series and to specify the result that we extend here.

For this recollection, let \(p\geq5\), and let \(E/\mathbb Q_p\) be a
finite extension with ring of integers \(\cO\), uniformizer \(\varpi\), and
residue field \(\F\). Put \(G=\GL_2(\mathbb Q_p)\),
\(K=\GL_2(\mathbb Z_p)\), and \(K_1=1+pM_2(\mathbb Z_p)\).
Let \(\Iw\) be the upper Iwahori subgroup and \(\Delta\subset\mathbb Z_p^\times\)
the Teichm\"uller subgroup. We use \(\omega\) for the mod-\(p\) cyclotomic
character and, when evaluating characters of \(\Delta\), for its
Teichm\"uller lift. The reducible, nonsplit, generic inertial types in
\cite{LTXZ1} have the form
\[
 \bar\rho:I_{\mathbb Q_p}\longrightarrow\GL_2(\F),\qquad
 \bar\rho\cong
 \begin{pmatrix}\omega^{a+b+1}&*\neq0\\0&\omega^b\end{pmatrix},
 \qquad 1\leq a\leq p-4,\quad 0\leq b\leq p-2
\]
where \(I_{\mathbb Q_p}\) is the inertia subgroup. The corresponding
Serre weight is \(\sigma_{a,b}=\operatorname{Sym}^a\F^2\otimes\det^b\).
A primitive projective-augmented module of this type is a finite
projective right \(\cO[\![K]\!]\)-module \(\widetilde H\), with a
continuous right \(G\)-action extending its \(K\)-action, for which
\[
 \widetilde H/(\varpi,I_{K_1})\widetilde H
 \cong\operatorname{Proj}_{a,b}
\]
Here \(I_{K_1}\) is the augmentation ideal and
\(\operatorname{Proj}_{a,b}\) is the projective envelope of
\(\sigma_{a,b}\) as an \(\F[\GL_2(\mathbb F_p)]\)-module. The module \(\widetilde H\) gives rise to spaces of abstract overconvergent forms \(\mathrm S_{\widetilde{H}}^{\dagger}\) and its subspaces \(\mathrm S_k^{\Iw}\), \(\mathrm S_k^{\ur}\). One can then define the \(U_p\)-action on these spaces using \(G\)-action (see \cite[Section~2.10 and Lemma~2.14]{LTXZ1}). More importantly, the ghost series 
\[
G^{(\varepsilon)}(w,T)=1+\sum_{n\geq1}g_n^{(\varepsilon)}(w)T^n
\]
is explicitly and combinatorically defined using these local data, and the following local ghost conjecture predicts its connection with the characteristic power series \(C_{\widetilde H}^{(\varepsilon)}(w,T)\) of \(U_p\):

\begin{introconjecture}[Local ghost conjecture, {\cite[Conjecture~2.29]{LTXZ1}}]
Let \(p\geq5\), and let \(\bar\rho\) be a reducible, nonsplit, generic
inertial type as above, with \(1\leq a\leq p-4\) and
\(0\leq b\leq p-2\). Let \(\widetilde H\) be a primitive
projective-augmented module of type \(\bar\rho\), and let
\(\varepsilon\) be any relevant character. Then:
\begin{enumerate}[label=(\arabic*)]
\item the ghost series \(G^{(\varepsilon)}(w,T)\) depends only on
\(\bar\rho\) and \(\varepsilon\), not on the choice of \(\widetilde H\);
\item for every \(w_\star\in\mathfrak m_{\Cp}\),
\[
 \NP\bigl(C_{\widetilde H}^{(\varepsilon)}(w_\star,-)\bigr)
 =\NP\bigl(G^{(\varepsilon)}(w_\star,-)\bigr)
\]
\end{enumerate}
\end{introconjecture}

The first assertion follows from the dimension formulas proved in
\cite[Corollary~4.8]{LTXZ1}. The second substantial statement is proved by Liu, Truong, Xiao, and Zhao in \cite{LTXZ2} under the
stronger genericity condition \(2\leq a\leq p-5\):

\begin{introtheorem}[Liu--Truong--Xiao--Zhao, {\cite[Theorems~1.6 and~2.8]{LTXZ2}}]
Assume \(p\geq11\), and let
\(\sigma_{a,b}=\operatorname{Sym}^a\F^2\otimes\det^b\), with
\(2\leq a\leq p-5\) and \(0\leq b\leq p-2\).
Let \(\widetilde H\) be any primitive projective-augmented module of
type \(\sigma_{a,b}\), and let
\(\varepsilon=\omega^{-s_\varepsilon+b}\times\omega^{a+s_\varepsilon+b}\)
be any relevant character, with \(0\leq s_\varepsilon\leq p-2\).
Then, for every \(w_\star\in\mathfrak m_{\Cp}\),
\[
 \NP\bigl(C_{\widetilde H}^{(\varepsilon)}(w_\star,-)\bigr)
 =\NP\bigl(G^{(\varepsilon)}(w_\star,-)\bigr)
\]
\end{introtheorem}

The equality includes the slope-zero part and is not restricted to
classical weights or to the boundary of weight space. Thus the theorem
reduces the spectral problem for these local modules to the study of an explicit series.

The purpose of the present paper is to establish the corresponding result
for the tr\`es ramifi\'e nonsplit extensions of \(1\) by \(\omega\).
This case lies outside the generic range above. After normalization, we fix
\[
 \bar\rho\cong
 \begin{pmatrix}\omega&*\neq0\\0&1\end{pmatrix}
\]
as a representation of \(G_{\mathbb Q_p}\), with tr\`es ramifi\'e
extension class; the precise definition is recalled in
Section~\ref{sec:setup}. Unlike the generic case, the extension class
must be specified: the tr\`es and peu ramifi\'e classes have different representation-theoretic input and this causes intrinsic differences. 

We formulate the definition of projective-augmented modules of tr\`es ramifi\'e type similar to that of the generic case in Section~2 and the whole local analogue remains valid. Our main result is the following:

\begin{theorem}[Local ghost theorem in the tr\`es ramifi\'e case]\label{thm:intro}
Let \(p\geq11\), let \(\widetilde H\) be a primitive
projective-augmented module of type \(\bar\rho\), and let
\(\varepsilon=\omega^{-s_\varepsilon}\times\omega^{s_\varepsilon}\) be
  relevant.  After the fixed unramified normalization specified in
\cref{prop:Pi}, write
\[
 C_{\widetilde H}^{(\varepsilon)}(w,T)
 =\det(1-TU_p\mid
       \mathrm S_{\widetilde H}^{\dagger,(\varepsilon)})
\]
and let \(G^{(\varepsilon)}(w,T)\) be the corresponding ghost series.  Then, for every \(w_\star\in\mathfrak m_{\Cp}\),
\[
 \NP\bigl(C_{\widetilde H}^{(\varepsilon)}(w_\star,-)\bigr)
 =
 \NP\bigl(G^{(\varepsilon)}(w_\star,-)\bigr)
\]
The equality includes the slope-zero term when \(s_\varepsilon=0\).
\end{theorem}

The main difference from the generic case is that the power-basis degrees
form one arithmetic progression
\[
 s_\varepsilon,\quad s_\varepsilon+p-1,\quad
 s_\varepsilon+2(p-1),\ldots
\]
Thus \(d_k^{\Iw}\) increases by one, rather than by two, when the weight is
increased by \(p-1\).  Consequently, \(d_k^{\Iw}/2\) and
\(d_k^{\new}/2\) may be half-integral.  The arguments of \cite[Section~5]{LTXZ1} must be carried out on the lattice \(\mathbb Z-d_k^{\Iw}/2\). Luckily, this parity issue do not cause any substantial obstruction to the generic approach, e.g. the whole near-Steinberg theory in Section~5 of \cite{LTXZ1} remains valid, which is not true in the peu ramifi\'e case.

\section{Notation and representation-theoretic input}\label{sec:setup}

This section introduces basic notation and collects the representation-theoretic input that we need. 
\begin{enumerate}[label=(\arabic*)]
\item Throughout the paper, $p \geq 11$ is a prime and  $E/\mathbb Q_p$ a finite extension with ring of integers $\cO$, uniformizer $\varpi$, and residue field $\F$. We enlarge \(E\) once and for all so that it contains a
chosen square root \(p^{1/2}\). The \(p\)-adic valuation is always normalized so that $v_p(p)=1$.

\item Put $G=\GL_2(\mathbb Q_p), K=\GL_2(\mathbb Z_p)$ and let $\Iw\subset K$ be the standard upper Iwahori subgroup and let $I_1$ be its pro-$p$ Sylow subgroup. We normalize the central element
$p$ to act trivially, after the unramified twist fixed below.

\item We write $\Delta\cong\F_p^\times\subset\mathbb Z_p^\times$ for the Teichm\"uller subgroup. The natural Teichm\"uller character is denoted by $\omega:\Delta \to \mathbb Z_p^\times$ and we view $\Delta^2\subset K$ as the diagonal torsion subgroup. 

\item Put $\Pi=\left(\begin{smallmatrix}0&1\\p&0\end{smallmatrix}\right)$. It is well known that to give a $G$-action is the same as giving a compatible action of $K$ and $\Pi$. We use the right-action conventions of \cite{LTXZ1,LTXZ2}.

\item Let \(G_{\mathbb Q_p}\) be the absolute Galois group of \(\mathbb Q_p\).  The
nonsplit representation in the introduction determines a nonzero extension class
\[
 c_{\bar\rho}\in
 \operatorname{Ext}^1_{G_{\mathbb Q_p}}(1,\omega)
 =H^1(G_{\mathbb Q_p},\F(\omega))
\]
Local Tate duality gives a perfect pairing
\[
 H^1(G_{\mathbb Q_p},\F(\omega))\times
 H^1(G_{\mathbb Q_p},\F)\longrightarrow
 H^2(G_{\mathbb Q_p},\F(\omega))\cong\F
\]
and we have \(H^1(G_{\mathbb Q_p},\F)\cong
\operatorname{Hom}_{\mathrm{cont}}(\mathbb Q_p^\times,\F)\). Let \(\operatorname{unr}_p\) denote the unramified character (i.e. $\operatorname{unr}_p|_{\mathbb Z_p^\times}=0$) sending $p$ to 1. Then following Serre and Pa\v sk\=unas \cite[Section~6.2]{PasBM}, the representation \(\bar\rho\) is \emph{peu ramifi\'e} if
\(\langle c_{\bar\rho},\operatorname{unr}_p\rangle=0\), and it is \emph{tr\`es ramifi\'e} otherwise.

\item Recall the mod-\(p\) representation
\[
\sigma_{a,b}=\operatorname{Sym}^a\F^2\otimes\operatorname{det}^b
\]
it is irreducible if and only if $0\leq a\leq p-1$ and $0\leq b<p-1$. Our $\sigma_{p-1,0}$ is self-dual in $\F[\GL_2(\F_p)]$.

\item We assume from now on that \(\bar\rho\) is tr\`es ramifi\'e. Following \cite[Definition~2.22]{LTXZ1},A \emph{projective-augmented module of type \(\bar\rho\) with multiplicity \(m\)} is a finite projective right
\(\cO[\![K]\!]\)-module \(\widetilde H\), equipped with a continuous right
\(\GL_2(\mathbb Q_p)\)-action extending its \(K\)-action such that:
\begin{itemize}
\item The quotient 
\[
\overline{H}=\widetilde{H}/{(\varpi,I_{K_1})}
\]
is \(m\) copies of \(\sigma_{p-1,0}\). Here \(I_{K_1}\) denotes the augmentation ideal for \(\cO[\![K_1]\!]\);
\item \(p\) acts by a scalar unit;
\item the pro-\(p\) scalar action is separated by the universal
      determinant twist:
\[
 \widetilde H\cong
 \widetilde H_0\widehat\otimes_{\cO}
 \cO[\![1+p\mathbb Z_p]\!]
\]
where the scalar subgroup acts trivially on \(\widetilde H_0\) and on
the second factor through the normalized determinant map of
\cite[Definition~2.22(3)]{LTXZ1}.
\item the smooth residual dual
\((\widetilde H/\varpi\widetilde H)^\vee\) contains \(\beta^{\oplus m}\) (See \cref{beta} below). In particular, \(\operatorname{soc}_G\beta=\operatorname{St}\).
Here \(\vee\) denotes the continuous \(\F\)-linear dual.
\end{itemize}
The module is \emph{primitive} when \(m=1\).

\item A character relevant to $\bar\rho$ is uniquely written
\[
 \eps=\omega^{-s_\eps}\times\omega^{s_\eps},\qquad 0\leq s_\eps\leq p-2
\]
Put
\[
 \varepsilon_1=\omega^{-s_\varepsilon},\qquad
 \widetilde\varepsilon_1=\varepsilon_1\times\varepsilon_1,
 \qquad
 k_\eps=2+\{2s_\eps\},\qquad
 \delta_\eps=\left\lfloor\frac{2s_\eps}{p-1}\right\rfloor,
 \qquad t_\eps=s_\eps+\delta_\eps+1
\]
where braces denote the representative in $\{0,\ldots,p-2\}$. For $k\equiv k_\eps\pmod{p-1}$ we use the
notation of the reference paper
\[
 k=k_\eps+(p-1)k_\bullet,\quad k_\bullet\in\mathbb Z_{\geq0}
\]
For such a weight,
\(\varepsilon(1\times\omega^{2-k})=\widetilde\varepsilon_1\), so the
unramified classical space with determinant character \(\varepsilon_1\) embeds
in the corresponding Iwahori space.  This is the specialization of the notation
in \cite[Section~2.4(4)--(6)]{LTXZ2} to the present type.

\item For every classical weight $k\equiv k_\eps\pmod{p-1}$, put
\[
 w_k=\exp(p(k-2))-1
\]
Let \(\mathcal W^{(\varepsilon)}\) be the corresponding rigid weight disc with coordinate \(w\), identified with $\mathfrak m_{\mathbb C_p}$.

\item For a character $\psi$ of the diagonal quotient of the Iwahori, put
\[
\begin{aligned}
 \mathrm S_k^{\Iw}(\psi)
 &=\operatorname{Hom}_{\cO[\![\Iw]\!]}\bigl(\widetilde H,
       \cO[z]^{\le k-2}\otimes\psi\bigr),\\
 \mathrm S_k^{\ur}(\psi_1)
 &=\operatorname{Hom}_{\cO[\![K]\!]}\bigl(\widetilde H,
       \operatorname{Sym}^{k-2}\cO^2\otimes(\psi_1\circ\det)\bigr)
\end{aligned}
\]
Their ranks are denoted by $d_k^{\Iw}$ and $d_k^{\ur}$, and
$d_k^{\new}=d_k^{\Iw}-2d_k^{\ur}$.
Once the character is clear from the text, it is omitted from the dimension notation.

\item The ghost series are defined, as in \cite[Definition~2.25]{LTXZ1}, by
\[
 g_n^{(\eps)}(w)=\prod_{k\equiv k_\eps\ (p-1)}(w-w_k)^{m_n(k)},
 \qquad G^{(\eps)}(w,T)=1+\sum_{n\geq1}g_n^{(\eps)}(w)T^n
\]
where $m_n(k)$ is the multiplicity function defined by
\begin{equation}\label{eq:ghost-multiplicity}
 m_n(k)=
 \begin{cases}
 \min\{n-d_k^{\ur},\ d_k^{\Iw}-d_k^{\ur}-n\},
 &d_k^{\ur}<n<d_k^{\Iw}-d_k^{\ur},\\
 0,&\text{otherwise}
 \end{cases}
\end{equation}
The superscript \(\eps\) is omitted once the relevant character is clear.
\end{enumerate}

\begin{remark}\label{beta}
The representation \(\beta\) is the length-three \emph{atome automorphe}
constructed by Colmez in
\cite[Proposition~VII.4.25 and Definition~VII.4.26]{Col10}; its
correspondence with the fixed \(\bar\rho\) is given by
\cite[Proposition~VII.4.24]{Col10}. See also
\cite[Section~6.2, Lemma~6.7 and Theorem~6.10]{PasBM} for the
description and deformation theory of \(\beta\).
Requiring the residual dual to contain \(\beta\) is therefore natural
in view of the mod-\(p\) local Langlands correspondence. For a primitive
module in the tr\`es ramifi\'e case, it would in fact suffice to require
\(\operatorname{St}\hookrightarrow(\widetilde H/\varpi\widetilde H)^\vee\). The residual
inclusion condition is used only in the proof of the following
\cref{prop:Pi}.
\end{remark}

\begin{proposition}\label{prop:Pi}
Let \(\widetilde H\) be a primitive projective-augmented module of type
\(\bar\rho\).  After a finite extension of coefficients and an unramified
twist, \(p\) acts trivially and there is an
\(\cO[\![I_1]\!]\)-basis \(e\), fixed by \(\Delta^2\), such
that
\[
 e\Pi=-e,\qquad (er)\Pi=-e\varphi(r)
 \quad(r\in\cO[\![I_1]\!])
\]
where \(\varphi(r)=\Pi^{-1}r\Pi\).  Changing the normalization by an
unramified unit twist does not change any Newton slope.
\end{proposition}

\begin{proof}
Restriction from \(K\) to \(I_1\) preserves projectivity because
\(\cO[\![K]\!]\) is finite free over \(\cO[\![I_1]\!]\).  The latter
completed group algebra is local, so the restriction of
\(\widetilde H\) is free.  Its rank is the dimension of the reduction of
its \(I_1\)-coinvariants, that is, the upper-unipotent coinvariants of
\(\operatorname{Proj}_{p-1,0}=\sigma_{p-1,0}\), which is 1-dimensional. Hence \(\widetilde H\) is free of
rank one over \(\cO[\![I_1]\!]\).

Let \(J\) be the Jacobson radical of \(\cO[\![I_1]\!]\).
The dual of \(\widetilde H/\widetilde HJ\) is
\(((\widetilde H/\varpi\widetilde H)^\vee)^{I_1}\); the universal central factor disappears
on taking these coinvariants. The residual inclusion in (7) and
\cite[Lemma~6.18]{PasBM} give
\[
 \operatorname{St}^{I_1}\cong\beta^{I_1}
 \lhook\joinrel\longrightarrow((\widetilde H/\varpi\widetilde H)^\vee)^{I_1}
\]
Both sides are one-dimensional, so this is an isomorphism. Thus
\(\Pi\) acts on the residual coinvariants by \(-1\), its scalar
on \(\operatorname{St}^{I_1}\).

If \(p\) acts by \(\xi\), the residual inclusion
implies \(\xi\equiv1\pmod\varpi\). Since \(p\ne2\), there is
a unique square root \(\eta(p)\in1+\varpi\cO\) of \(\xi^{-1}\).
The corresponding unramified determinant twist makes \(p\) act trivially and in particular \(\Pi^2=1\).

Lift a residual generator and average over \(\Delta^2\), obtaining
a fixed lift \(u\). This does not change its reduction, since
the residual torus character is trivial and \(|\Delta^2|\) is a
unit in \(\cO\). Put
\(e=(u-u\Pi)/2\), then \(e\) is fixed by \(\Delta^2\), \(e\Pi=-e\) and \(e\equiv u\pmod{\widetilde HJ}\). Compact Nakayama lemma therefore makes \(e\) a basis.
\end{proof}

For the remainder of the local argument, \(\widetilde H\) denotes an arbitrary
primitive projective-augmented module of type \(\bar\rho\), normalized by the
unramified twist in \cref{prop:Pi}, and \(e\) denotes the basis in
\cref{prop:Pi}.  The construction and all estimates commute with finite
extension of the coefficient field.

\section{Dimensions and the ghost series}

This section reproduces, for the present tr\`es ramifi\'e type, the dimension and coefficient
calculations of \cite[Section~4]{LTXZ1}.  The definitions of the classical
spaces and ghost multiplicities are those of
\cite[Sections~2.4 and~2.6]{LTXZ1}.  What changes is the list of power-basis
degrees: in the present case it is the single progression
\[
 s_\eps,\ s_\eps+p-1,\ s_\eps+2(p-1),\ldots 
\]
and $d_k^{\Iw}$ is not always even. Luckily, this parity issue causes no substantial obstruction to the arguments of \cite[Section~4]{LTXZ1}.

\begin{proposition}\label{prop:iw-dimension}
For every integer \(k\geq2\),
\begin{equation}\label{eq:iw-general}
 d_k^{\Iw}\bigl(\eps(1\times\omega^{2-k})\bigr)
 =\left\lfloor\frac{k-2-s_\eps}{p-1}\right\rfloor+1
\end{equation}
For \(k=k_\eps+(p-1)k_\bullet\), this becomes
\[
 d_k^{\Iw}=k_\bullet+1-\delta_\eps
\]
In particular \(d_k^{\Iw}\) increases by one when \(k\) is replaced by
\(k+p-1\), and it can be odd.
\end{proposition}

\begin{proof}
The generator \(e\) has trivial diagonal-torsion character.  Hence the
power basis on the \(\eps\)-disc consists of
\(e^*z^{s_\eps+i(p-1)}\), \(i\geq0\).  Such a vector is classical
exactly when \(s_\eps+i(p-1)\leq k-2\), which gives
\eqref{eq:iw-general}.  Substituting
\(k-2=\{2s_\eps\}+(p-1)k_\bullet\) gives the second formula.
\end{proof}

\begin{proposition}\label{prop:ur-dimension}
For \(k=k_\eps+(p-1)k_\bullet\), recall that
\(t_\eps=s_\eps+\delta_\eps+1\).  Then
\begin{equation}\label{eq:ur-dimension}
 d_k^{\ur}(\eps_1)
 =\left\lfloor\frac{k_\bullet-t_\eps}{p+1}\right\rfloor+1
\end{equation}
The right-hand side is nonnegative.  Moreover,
\[
 d_{k+p^2-1}^{\ur}=d_k^{\ur}+1,\qquad
 d_k^{\ur}=\frac{k}{p^2-1}+O(1),\qquad
 d_k^{\new}=\frac{k}{p+1}+O(1)
\]
\end{proposition}

\begin{proof}
Put \(m=k-2=\{2s_\eps\}+(p-1)k_\bullet\).  Twisting by
\(\det^{s_\eps}\), and using
\(\operatorname{Proj}_{p-1,s_\eps}=\sigma_{p-1,s_\eps}\), gives
\[
\begin{aligned}
 d_k^{\ur}(\eps_1)
 &=\dim_{\F}\operatorname{Hom}_{\F[\overline G]}
   (\operatorname{Proj}_{p-1,0},\sigma_{m,-s_\eps})\\
 &=\dim_{\F}\operatorname{Hom}_{\F[\overline G]}
   (\operatorname{Proj}_{p-1,s_\eps},\sigma_{m,0})\\
 &=\operatorname{Mult}_{\sigma_{p-1,s_\eps}}(\sigma_{m,0})
\end{aligned}
\]
Here $\operatorname{Mult}_{\sigma_{p-1,s_\eps}}(\sigma_{m,0})$ denotes the multiplicity of $\sigma_{p-1,s_\eps}$ appearing in the Jordan--H\"older factors of $\sigma_{m,0}$. We compute this multiplicity using the same exact sequence as
in \cite[Appendix~A and proof of Proposition~4.7]{LTXZ1}:
\[
 0\longrightarrow\sigma_{r-(p+1),b+1}
 \longrightarrow\sigma_{r,b}
 \longrightarrow
 \operatorname{Ind}_{\overline B}^{\overline G}(\eta^r)\otimes{\det}^b
 \longrightarrow 0
\]
and
\[
 [\operatorname{Ind}_{\overline B}^{\overline G}(\eta^r)\otimes{\det}^b]
 =[\sigma_{\{r\},b}]
  +[\sigma_{p-1-\{r\},\{r\}+b}]
\]
with \(\overline G=\GL_2(\mathbb F_p)\), and
\(\overline B\) its upper-triangular subgroup, $\eta\!\begin{pmatrix}\alpha&*\\0&\delta\end{pmatrix}=\alpha$.
Taking \(r=m-j(p+1)\) and \(b=j\) therefore gives
\begin{equation}\label{eq:tres-dickson-row}
\begin{split}
 [\sigma_{m-j(p+1),j}]-[\sigma_{m-(j+1)(p+1),j+1}]
 ={}&[\sigma_{\{2s_\eps-2j\},j}]\\
 &+[\sigma_{p-1-\{2s_\eps-2j\},\{2s_\eps\}-j}]
\end{split}
\end{equation}

If \(m\geq p^2-1\), summing \eqref{eq:tres-dickson-row} for
\(0\leq j\leq p-2\) yields
\begin{align}\label{eq:tres-full-cycle}
 [\sigma_{m,0}]-[\sigma_{m-(p^2-1),0}]
 =\sum_{j=0}^{p-2}\bigl(&[\sigma_{\{2s_\eps-2j\},j}]\nonumber\\
 &+[\sigma_{p-1-\{2s_\eps-2j\},\{2s_\eps\}-j}]\bigr)
\end{align}
The first summand in every row has first index at most \(p-2\).  The
second equals \(\sigma_{p-1,s_\eps}\) precisely when
\[
 \{2s_\eps-2j\}=0,
 \qquad \{2s_\eps\}-j\equiv s_\eps\pmod{p-1}
\]
or equivalently when \(j=s_\eps\).  Hence the target occurs exactly
once in \eqref{eq:tres-full-cycle}, and
\begin{equation}\label{eq:tres-periodicity}
 \operatorname{Mult}_{\sigma_{p-1,s_\eps}}(\sigma_{m,0})
 -\operatorname{Mult}_{\sigma_{p-1,s_\eps}}
       (\sigma_{m-(p^2-1),0})=1
\end{equation}

It remains to compute one cycle.  Suppose \(m<p^2-1\) and write
\[
 m=\ell(p+1)+r,
 \qquad 0\leq\ell\leq p-2,\quad0\leq r\leq p
\]
Summing the first \(\ell\) rows of \eqref{eq:tres-dickson-row} gives
\begin{align}\label{eq:tres-initial-cycle}
 [\sigma_{m,0}]-[\sigma_{r,\ell}]
 =\sum_{j=0}^{\ell-1}\bigl(&[\sigma_{\{2s_\eps-2j\},j}]\nonumber\\
 &+[\sigma_{p-1-\{2s_\eps-2j\},\{2s_\eps\}-j}]\bigr)
\end{align}
The sum on the right contributes the target exactly when
\(\ell\geq s_\eps+1\).  The remainder contributes it exactly when
\[
 \ell=s_\eps,\qquad r=p-1
\]
Indeed, this is immediate for \(r\leq p-1\), while at the remaining
endpoint
\[
 [\sigma_{p,\ell}]=[\sigma_{1,\ell}]
                    +[\sigma_{p-2,\ell+1}]
\]
so no Steinberg factor occurs.  The remainder condition is equivalent
to \(k_\bullet=t_\eps\); moreover,
\[
 \ell\geq s_\eps+1
 \quad\Longleftrightarrow\quad
 k_\bullet\geq
 \left\lceil
 \frac{(s_\eps+1)(p+1)-\{2s_\eps\}}{p-1}
 \right\rceil=t_\eps+1
\]
Thus the multiplicity in the initial cycle is zero for
\(k_\bullet<t_\eps\) and one for \(k_\bullet\geq t_\eps\).

Finally write \(k_\bullet=h(p+1)+b\), with \(0\leq b\leq p\).
Equation \eqref{eq:tres-periodicity} and the initial-cycle calculation
give
\[
 d_k^{\ur}=h+
 \begin{cases}0,& b<t_\eps\\1,&b\geq t_\eps\end{cases}
\]
which is \eqref{eq:ur-dimension}.  The periodicity and asymptotic
formulae follow from this identity and \eqref{eq:iw-general}.
\end{proof}

For \(n\geq0\), retain the endpoint notation of \cite[Section~4]{LTXZ1}:
\begin{align}\label{eq:endpoints}
 k_{\max\bullet}(n)&=n(p+1)+s_\eps+\delta_\eps,\nonumber\\
 k_{\midd\bullet}(n)&=2n+\delta_\eps,\\
 k_{\min\bullet}(n)&=n+\delta_\eps+
             \left\lceil\frac{n-s_\eps}{p}\right\rceil.\nonumber
\end{align}
Put \(k_*(n)=k_\eps+(p-1)k_{*\bullet}(n)\) for
\(*\in\{\min,\midd,\max\}\).

\begin{lemmanotation}\label{lem:extremal}
For \(k=k_\eps+(p-1)k_\bullet\),
\[
\begin{array}{rclcrcl}
d_k^{\ur}\leq n
&\Longleftrightarrow&
k_\bullet\leq k_{\max\bullet}(n),
&&
d_k^{\Iw}-d_k^{\ur}>n
&\Longleftrightarrow&
k_\bullet\geq k_{\min\bullet}(n),\\
d_k^{\Iw}\geq2n+2
&\Longleftrightarrow&
k_\bullet>k_{\midd\bullet}(n),
&&
d_k^{\Iw}=2n+1
&\Longleftrightarrow&
k_\bullet=k_{\midd\bullet}(n)
\end{array}
\]
Consequently
\begin{equation}\label{eq:m-increment}
m_{n+1}(k)-m_n(k)=
\begin{cases}
+1,&k_{\midd\bullet}(n)<k_\bullet\leq k_{\max\bullet}(n)\\
-1,&k_{\min\bullet}(n)\leq k_\bullet<k_{\midd\bullet}(n)\\
0,&\text{otherwise}
\end{cases}
\end{equation}
In particular, the odd-dimensional central weight \(k_\midd(n)\) is
absent from both nonzero intervals.
\end{lemmanotation}

\begin{proof}
The first equivalence follows without approximation:
\[
\begin{aligned}
 d_k^{\ur}\leq n
 &\Longleftrightarrow
 \left\lfloor\frac{k_\bullet-t_\eps}{p+1}\right\rfloor\leq n-1\\
 &\Longleftrightarrow k_\bullet-t_\eps<n(p+1)
 \Longleftrightarrow k_\bullet\leq k_{\max\bullet}(n)
\end{aligned}
\]
For the second, write
\(k_\bullet-t_\eps=(p+1)h+j\), with \(0\leq j\leq p\).  Then
\[
 d_k^{\Iw}-d_k^{\ur}=s_\eps+1+ph+j
\]
If \(n-s_\eps=ph_0+j_0\), \(0\leq j_0<p\), the least value of
\(k_\bullet-t_\eps\) for which the last expression is greater than
\(n\) is
\[
 n-s_\eps-1+\left\lceil\frac{n-s_\eps}{p}\right\rceil
\]
Adding \(t_\eps\) gives exactly \(k_{\min\bullet}(n)\).  The remaining
two equivalences follow from
\(d_k^{\Iw}=k_\bullet+1-\delta_\eps\).

It remains to inspect the endpoints in
\eqref{eq:ghost-multiplicity}.  At \(d_k^{\ur}=n\), the increment is
\(+1\) precisely when \(d_k^{\Iw}\geq2n+2\).  At
\(d_k^{\Iw}-d_k^{\ur}=n+1\), it is \(-1\) precisely when
\(d_k^{\Iw}\leq2n\).  Between these endpoints the two affine terms in
the minimum change by \(+1\) and \(-1\), respectively.  When
\(d_k^{\Iw}=2n+1\), they have the same value at \(n\) and \(n+1\), so
the increment is zero.  This proves \eqref{eq:m-increment}, including
all boundary equalities.
\end{proof}

\begin{proposition}\label{prop:degree-increment}
For \(\mathbf e_{n+1}=e^*z^{s_\eps+n(p-1)}\), we have
\[
 \deg g_{n+1}-\deg g_n
 =n(p-2)+s_\eps+\left\lceil\frac{n-s_\eps}{p}\right\rceil
 =\deg\mathbf e_{n+1}
  -\left\lfloor\frac{\deg\mathbf e_{n+1}}p\right\rfloor
\]
for all $n\geq 0$. These increments are strictly increasing.
\end{proposition}

\begin{proof}
Sum \eqref{eq:m-increment} over \(k_\bullet\).  The result is
\[
 \bigl(k_{\max\bullet}(n)-k_{\midd\bullet}(n)\bigr)
 -\bigl(k_{\midd\bullet}(n)-k_{\min\bullet}(n)\bigr)
\]
which gives the displayed formula after inserting \eqref{eq:endpoints}.
The difference of two consecutive values is \(p-2\) or \(p-1\).
\end{proof}

Put $g_{n,\widehat k}(w)=g_n(w)/(w-w_k)^{m_n(k)}$. For a finite set \(\mathbf k\), remove all corresponding factors and
write \(g_{n,\widehat{\mathbf k}}\).

The coefficient comparisons in \cite[Proposition~4.18]{LTXZ1} begin
with the following identity.  We state it because all later changes
of endpoints are obtained by taking its indicated finite differences.

\begin{lemma}\label{lem:valuation-increment}
For \(k_0=k_\eps+(p-1)k_{0\bullet}\),
\begin{equation}\label{eq:valuation-increment}
\begin{split}
&\val(g_{n+1,\widehat{k_0}}(w_{k_0}))-
 \val(g_{n,\widehat{k_0}}(w_{k_0}))\\
={}&\sum_{\substack{k_{\midd\bullet}(n)<k_\bullet
                    \leq k_{\max\bullet}(n)\\
                    k_\bullet\ne k_{0\bullet}}}
  \bigl(1+\val(k_\bullet-k_{0\bullet})\bigr)\\
&-\sum_{\substack{k_{\min\bullet}(n)\leq k_\bullet
                    <k_{\midd\bullet}(n)\\
                    k_\bullet\ne k_{0\bullet}}}
  \bigl(1+\val(k_\bullet-k_{0\bullet})\bigr)
\end{split}
\end{equation}
For \(n\geq1\), its second difference is
\begin{align}\label{eq:valuation-second}
&\val(g_{n+1,\widehat{k_0}}(w_{k_0}))
-2\val(g_{n,\widehat{k_0}}(w_{k_0}))
+\val(g_{n-1,\widehat{k_0}}(w_{k_0}))\nonumber\\
={}&
\sum_{\substack{k_{\max\bullet}(n-1)<k_\bullet
                         \leq k_{\max\bullet}(n)\\
                         k_\bullet\ne k_{0\bullet}}}
 \bigl(1+\val(k_\bullet-k_{0\bullet})\bigr)
+\sum_{\substack{k_{\min\bullet}(n-1)\leq k_\bullet
                         <k_{\min\bullet}(n)\\
                         k_\bullet\ne k_{0\bullet}}}
 \bigl(1+\val(k_\bullet-k_{0\bullet})\bigr)\nonumber\\
&-\sum_{\substack{k_{\midd\bullet}(n-1)\leq k_\bullet
                         <k_{\midd\bullet}(n)\\
                         k_\bullet\ne k_{0\bullet}}}
 \bigl(1+\val(k_\bullet-k_{0\bullet})\bigr)
-\sum_{\substack{k_{\midd\bullet}(n-1)<k_\bullet
                         \leq k_{\midd\bullet}(n)\\
                         k_\bullet\ne k_{0\bullet}}}
 \bigl(1+\val(k_\bullet-k_{0\bullet})\bigr)
\end{align}
The overlap in the last two sums gives coefficient \(-2\) at the
unique integer between consecutive middle endpoints.
\end{lemma}

\begin{proof}
By the definition of the ghost coefficients,
\[
\begin{split}
&\val(g_{n+1,\widehat{k_0}}(w_{k_0}))-
 \val(g_{n,\widehat{k_0}}(w_{k_0}))\\
&\qquad=\sum_{\substack{k\equiv k_\eps\ (p-1)\\k\ne k_0}}
 \bigl(m_{n+1}(k)-m_n(k)\bigr)\val(w_k-w_{k_0})
\end{split}
\]
For \(k=k_\eps+(p-1)k_\bullet\), one has
\[
 \val(w_k-w_{k_0})=1+\val(k_\bullet-k_{0\bullet})
\]
Equation~\eqref{eq:m-increment} says that
\[
 m_{n+1}(k)-m_n(k)=
 \begin{cases}
  1,&k_\bullet\in(k_{\midd\bullet}(n),k_{\max\bullet}(n)]\\
  -1,&k_\bullet\in[k_{\min\bullet}(n),k_{\midd\bullet}(n))\\
  0,&\text{otherwise}
 \end{cases}
\]
Substitution gives \eqref{eq:valuation-increment}, including the exclusion
of \(k_{0\bullet}\) from both sums.

To compute the second difference, subtract the same identity with \(n-1\)
in place of \(n\).  Since all three endpoint functions are increasing, the
positive intervals satisfy
\[
\begin{split}
&\boldsymbol 1_{(k_{\midd\bullet}(n),k_{\max\bullet}(n)]}
-\boldsymbol 1_{(k_{\midd\bullet}(n-1),k_{\max\bullet}(n-1)]}\\
&\quad=
 \boldsymbol 1_{(k_{\max\bullet}(n-1),k_{\max\bullet}(n)]}
-\boldsymbol 1_{(k_{\midd\bullet}(n-1),k_{\midd\bullet}(n)]}
\end{split}
\]
whereas the negative intervals satisfy
\[
\begin{split}
&-\boldsymbol 1_{[k_{\min\bullet}(n),k_{\midd\bullet}(n))}
+\boldsymbol 1_{[k_{\min\bullet}(n-1),k_{\midd\bullet}(n-1))}\\
&\quad=
 \boldsymbol 1_{[k_{\min\bullet}(n-1),k_{\min\bullet}(n))}
-\boldsymbol 1_{[k_{\midd\bullet}(n-1),k_{\midd\bullet}(n))}
\end{split}
\]
Multiplying these identities by
\(1+\val(k_\bullet-k_{0\bullet})\) and summing over
\(k_\bullet\ne k_{0\bullet}\) gives \eqref{eq:valuation-second}.  Finally,
\(k_{\midd\bullet}(n)-k_{\midd\bullet}(n-1)=2\).  The two intervals removed
at the middle endpoint therefore overlap only at
\(k_{\midd\bullet}(n-1)+1\), which occurs with coefficient \(-2\).
\end{proof}

For \(m\geq0\), let \(\Dig(m)\) denote the sum of the \(p\)-adic digits of \(m\).
For a coefficient index \(n\), we use throughout
\[
 M=k_{\max\bullet}(n)-k_{\midd\bullet}(n)=(p-1)n+s_\eps,\qquad q=k_{\midd\bullet}(n)-k_{\min\bullet}(n)
 =\left\lfloor\frac Mp\right\rfloor
\]
to simplify notation.

\begin{proposition}\label{prop:old-bound}
For \(k_0\equiv k_\eps\pmod{p-1}\), the first
\(d_{k_0}^{\ur}\) slopes satisfy
\[
 \text{\rm slope}\leq
 \left\lfloor\frac{k_0-2}{p+1}\right\rfloor
 <\frac{k_0-2}{2}
\]
when \(d_{k_0}^{\ur}>0\).
\end{proposition}

\begin{proof}
We shall prove that, for every \(1\leq h\leq d_{k_0}^{\ur}\),
\[
\begin{split}
&\sum_{n=d_{k_0}^{\ur}-h}^{d_{k_0}^{\ur}-1}
 \bigl(v_p(g_{n+1}(w_{k_0}))-v_p(g_n(w_{k_0}))\bigr)\\
&\qquad=
 v_p(g_{d_{k_0}^{\ur}}(w_{k_0}))
 -v_p(g_{d_{k_0}^{\ur}-h}(w_{k_0}))
 \leq h\left\lfloor\frac{k_0-2}{p+1}\right\rfloor
\end{split}
\]

In the range \(0\leq n<d_{k_0}^{\ur}\),
\(m_n(k_0)=m_{n+1}(k_0)=0\) and
\(k_{0\bullet}>k_{\max\bullet}(n)\), so the deleted factors in
\eqref{eq:valuation-increment} are absent. The positive and negative
summation intervals have respective lengths $M$ and $q$.

In the positive sum, the integers \(k_{0\bullet}-k_\bullet\) run
from \(k_{0\bullet}-k_{\max\bullet}(n)\) to
\(k_{0\bullet}-k_{\midd\bullet}(n)-1\); in the negative sum,
they run from \(k_{0\bullet}-k_{\midd\bullet}(n)+1\) to
\(k_{0\bullet}-k_{\min\bullet}(n)\). These together give
\begin{align}\label{eq:old-increment-factorials}
&\val(g_{n+1}(w_{k_0}))-\val(g_n(w_{k_0}))\nonumber\\
={}&(p-1)n+s_\eps-
 \left\lfloor\frac{(p-1)n+s_\eps}{p}\right\rfloor\nonumber\\
&+\val\!\left(
 \frac{(k_{0\bullet}-k_{\midd\bullet}(n)-1)!}
      {(k_{0\bullet}-k_{\max\bullet}(n)-1)!}\right)\nonumber\\
&-\val\!\left(
 \frac{(k_{0\bullet}-k_{\min\bullet}(n))!}
      {(k_{0\bullet}-k_{\midd\bullet}(n))!}\right)
\end{align}
Put \(N=k_{0\bullet}-k_{\max\bullet}(n)\).
The inequality \(n<d_{k_0}^{\ur}\) and
\cref{lem:extremal} imply \(N\geq1\).  Since
\[
 k_0-2=(p+1)M+(p-1)N
\]
we have
\[
 \left\lfloor\frac{k_0-2}{p+1}\right\rfloor
 =M+\left\lfloor\frac{(p-1)N}{p+1}\right\rfloor
\]
The two factorial quotients in
\eqref{eq:old-increment-factorials} become
\[
 \frac{(N+M-1)!}{(N-1)!},\qquad
 \frac{(N+M+q)!}{(N+M)!}
\]
Since \(M=pq+\{M\}_p\), Legendre's formula gives
\(v_p(M!)-v_p(q!)=q\).  Hence
\begin{align*}
 \val\!\left(\frac{g_{n+1}(w_{k_0})}{g_n(w_{k_0})}\right)
 &=M-q+v_p\!\left(\frac{(N+M-1)!}{(N-1)!}\right)
       -v_p\!\left(\frac{(N+M+q)!}{(N+M)!}\right)\\
 &=M+v_p\binom{N+M-1}{M}
       -v_p\binom{N+M+q}{q}
\end{align*}
For use in a sum, Legendre's formula rewrites this identity as
\begin{align*}
 (p-1)\left(\val\!\left(\frac{g_{n+1}(w_{k_0})}{g_n(w_{k_0})}\right)-M\right)
 ={}&\{M\}_p+\Dig(N-1)-\Dig(N+M-1)\\
 &+\Dig(N+M+q)-\Dig(N+M)\\
 ={}&\Dig(N-1)+\Dig((p+1)(N+M)-N)\\
 &-\Dig(N+M-1)-\Dig(N+M)
\end{align*}
The last equality follows by taking digit sums in
\[
 (p+1)(N+M)-N=p(N+M+q)+\{M\}_p
\]
We now estimate the sum specified at the beginning of the proof.
Fix \(1\leq h\leq d_{k_0}^{\ur}\), and put
\[
 A=k_{0\bullet}-k_{\max\bullet}(d_{k_0}^{\ur}-1),\qquad
 B=k_{0\bullet}-k_{\midd\bullet}(d_{k_0}^{\ur}-1)
\]
Recall that \(\deg\mathbf e_{d_{k_0}^{\ur}}
=(p-1)(d_{k_0}^{\ur}-1)+s_\eps\).
The formula for \(d_{k_0}^{\ur}\) gives \(1\leq A\leq p+1\), and
\(B=A+\deg\mathbf e_{d_{k_0}^{\ur}}\geq1\).
Write \(n=d_{k_0}^{\ur}-1-j\), with \(0\leq j<h\), to enumerate
the increments in this sum in reverse order. The quantities
\(M,q,N\) vary with \(j\), whereas \(A,B\) are fixed, and
\[
 N=A+(p+1)j,\qquad M=\deg\mathbf e_{d_{k_0}^{\ur}}-(p-1)j,\qquad N+M=B+2j
\]
In particular the two positive digit-sum arguments become
\[
 N-1=A-1+(p+1)j,\qquad
 (p+1)(N+M)-N=(p+1)(B+j-1)+(p+1-A)
\]
Digit sums satisfy \(\Dig(x+y)\leq\Dig(x)+\Dig(y)\),
\(\Dig(px)=\Dig(x)\), and \(\Dig(x)\leq x\) for nonnegative
integers.  If \(A\leq p\), these inequalities give
\begin{align*}
 \Dig(N-1)&\leq A-1+2j\\
 \Dig((p+1)(N+M)-N)&\leq
 \Dig((p+1)(B+j-1))+p+1-A
\end{align*}
Their sum is at most \(p+2j+\Dig((p+1)(B+j-1))\).
If \(A=p+1\), the corresponding bound is
\(1+2j+\Dig((p+1)(B+j-1))\), since \(\Dig(A-1)=\Dig(p)=1\).

We use the precise digit-sum inequality
\cite[Lemma~4.29]{LTXZ1}: for nonnegative integers \(x,m\),
\[
 \Dig((p+1)x)\leq\Dig(x)+\Dig(x+m)+(p-2)m
\]
Apply it with \(x=B+j-1\) and \(m=h\), using the same fixed
length \(h\) for every \(0\leq j<h\), and sum over these \(j\).
The arguments \(B+j-1\) and
\(B+h+j-1\) together run through all integers from \(B-1\) to
\(B+2h-2\), as do \(B+2j-1\) and \(B+2j\).  Thus their digit
sums cancel exactly against the two negative sums.  We obtain
\begin{align*}
 (p-1)\left(
 \val\!\left(\frac{g_{d_{k_0}^{\ur}}(w_{k_0})}
 {g_{d_{k_0}^{\ur}-h}(w_{k_0})}\right)
 -h\deg\mathbf e_{d_{k_0}^{\ur}}+\frac{p-1}{2}h(h-1)\right)
 &\leq ph+h(h-1)+(p-2)h^2\\
 &=(p-1)(h^2+h)
\end{align*}
For \(A=p+1\), replace \(ph\) by \(h\); the resulting bound is
\((p-1)h^2\).

Suppose first that \(h\geq2\).  As \(p\geq11\),
\[
 \frac{p-1}{2}h(h-1)-(h^2+h)
 \geq5h(h-1)-h^2-h=2h(2h-3)>0
\]
Consequently
\[
 \val\!\left(\frac{g_{d_{k_0}^{\ur}}(w_{k_0})}
 {g_{d_{k_0}^{\ur}-h}(w_{k_0})}\right)
 \leq h\deg\mathbf e_{d_{k_0}^{\ur}} -\frac{p-1}{2}h(h-1)+h^2+h
 \leq h\deg\mathbf e_{d_{k_0}^{\ur}}
 \leq h\left\lfloor\frac{k_0-2}{p+1}\right\rfloor
\]
For \(h=1\), the digit-sum bound gives \(\val\!\left(\frac{g_{d_{k_0}^{\ur}}(w_{k_0})}
 {g_{d_{k_0}^{\ur}-1}(w_{k_0})}\right)\leq \deg\mathbf e_{d_{k_0}^{\ur}}+2\)
when \(A\leq p\), and \(\val\!\left(\frac{g_{d_{k_0}^{\ur}}(w_{k_0})}
 {g_{d_{k_0}^{\ur}-1}(w_{k_0})}\right)\leq \deg\mathbf e_{d_{k_0}^{\ur}}+1\) when \(A=p+1\).
These suffice for \(3\leq A\leq p+1\), because
\[
 \left\lfloor\frac{k_0-2}{p+1}\right\rfloor
 =\deg\mathbf e_{d_{k_0}^{\ur}}+\left\lfloor\frac{(p-1)A}{p+1}\right\rfloor
\]
and the last integer is at least two for \(3\leq A\leq p\),
and equals \(p-1\) for \(A=p+1\).
If \(A=1\), the first binomial coefficient in the formula for
\(\val\!\left(\frac{g_{d_{k_0}^{\ur}}(w_{k_0})}
 {g_{d_{k_0}^{\ur}-1}(w_{k_0})}\right)-\deg\mathbf e_{d_{k_0}^{\ur}}\) is one, so that difference is nonpositive.
If \(A=2\), that binomial coefficient equals \(B-1\).
When \(p\nmid B-1\), the difference is again nonpositive.  Otherwise
write \(B-1=p^\nu t\) with \(\nu\geq1\) and \(p\nmid t\).  Then
\[
 q=\left\lfloor\frac{B-2}{p}\right\rfloor=p^{\nu-1}t-1,
 \qquad B+q=(p+1)p^{\nu-1}t
\]
Using
\[
 \binom{B+q}{q}=\frac{B+q}{B}\binom{B+q-1}{q}
\]
and \(v_p(B)=0\), its valuation is at least \(\nu-1\).  Hence
\(\val\!\left(\frac{g_{d_{k_0}^{\ur}}(w_{k_0})}
 {g_{d_{k_0}^{\ur}-1}(w_{k_0})}\right)-\deg\mathbf e_{d_{k_0}^{\ur}}\leq \nu-(\nu-1)=1\), which is precisely the required
bound for \(A=2\).
\end{proof}

The following Proposition recovers the main result Proposition 4.18 of section 4 of \cite{LTXZ1}

\begin{proposition}\label{prop:compatibility}
Fix \(k_0\geq2\), and put
\[
 d=d_{k_0}^{\Iw}\bigl(\eps(1\times\omega^{2-k_0})\bigr)
\]
The theta and opposite characters are, respectively,
\[
\begin{aligned}
 \eps'&=\eps(\omega^{k_0-1}\times\omega^{1-k_0})
       =\omega^{-s_{\eps'}}\times\omega^{s_{\eps'}},
 &s_{\eps'}&=\{s_\eps+1-k_0\},\\
 \eps''&=\omega^{-s_{\eps''}}\times\omega^{s_{\eps''}},
 &s_{\eps''}&=\{k_0-2-s_\eps\}
\end{aligned}
\]
\begin{enumerate}[label=(\arabic*)]
\item For every \(\ell\geq1\), the \((d+\ell)\)-th slope of
\(\NP(G^{(\eps)}(w_{k_0},-))\) is \(k_0-1\) plus the \(\ell\)-th
slope of \(\NP(G^{(\eps')}(w_{2-k_0},-))\).
\item If \(k_0\not\equiv k_\eps\pmod{p-1}\), then, for
\(1\leq\ell\leq d\), the sum of the \(\ell\)-th slope of
\(\NP(G^{(\eps)}(w_{k_0},-))\) and the \((d-\ell+1)\)-st slope of
\(\NP(G^{(\eps'')}(w_{k_0},-))\) is \(k_0-1\).
\item If \(k_0\equiv k_\eps\pmod{p-1}\), then \(\eps''=\eps\), and
the \(\ell\)-th and \((d-\ell+1)\)-st slopes of
\(\NP(G^{(\eps)}(w_{k_0},-))\) add to \(k_0-1\) for
\(1\leq\ell\leq d_{k_0}^{\ur}\).
\item If \(k_0\equiv k_\eps\pmod{p-1}\), then, for
\(0\leq j\leq d_{k_0}^{\new}\),
\begin{equation}\label{eq:ghost-duality}
\begin{split}
&\val\bigl(g_{d_{k_0}^{\Iw}-d_{k_0}^{\ur}-j,\widehat{k_0}}
                    (w_{k_0})\bigr)
-\val\bigl(g_{d_{k_0}^{\ur}+j,\widehat{k_0}}(w_{k_0})\bigr)\\
&\hspace{23mm}=(k_0-2)
\left(\frac{d_{k_0}^{\new}}2-j\right)
\end{split}
\end{equation}
Thus the \(d_{k_0}^{\new}\) slopes in the new interval are
\((k_0-2)/2\).
\end{enumerate}
\end{proposition}

\begin{proof}
We first record the coefficient comparisons from which the slope
statements follow.  The endpoint lengths are
\[
\begin{split}
 k_{\max\bullet}(n)-k_{\midd\bullet}(n)
   &=(p-1)n+s_\eps,\\
 k_{\midd\bullet}(n)-k_{\min\bullet}(n)
   &=\left\lfloor\frac{(p-1)n+s_\eps}{p}\right\rfloor 
\end{split}
\]
Direct division in \eqref{eq:iw-general} gives
\[
 (p-1)d=k_0-1-s_\eps+s_{\eps'},
 \qquad
 (p-1)(d-1)=k_0-2-s_\eps-s_{\eps''}
\]
We give the interval cancellations explicitly.  For an evaluation
weight \(k_0\), use the degree
\(m=\deg\mathbf e_{n+1}^{(\eps)}=s_\eps+(p-1)n\) as the
summation endpoint.  Translating the two intervals of
\eqref{eq:valuation-increment} by \(2n+\delta_\eps\) gives
\begin{equation}\label{eq:compatibility-increment-sums}
\begin{split}
&\val(g_{n+1,\widehat{k_0}}^{(\eps)}(w_{k_0}))
 -\val(g_{n,\widehat{k_0}}^{(\eps)}(w_{k_0}))\\
&=\sideset{}{'}\sum_{v=1}^{m}
 \bigl(1+v_p(2m+2-k_0+(p-1)v)\bigr)\\
&\quad-\sideset{}{'}\sum_{v=-\lfloor m/p\rfloor}^{-1}
 \bigl(1+v_p(2m+2-k_0+(p-1)v)\bigr)
\end{split}
\end{equation}
Here and below a prime excludes the term for which
\(2m+2-k_0+(p-1)v=0\).  This is exactly the removal of the
evaluation weight.
This formula also applies at a noncongruent evaluation weight:
\(2m+2-k_0+(p-1)v\) is the actual integer difference of the two weights,
and no factor vanishes.  The only arithmetic fact needed to cancel
these sums is
\[
\begin{gathered}
 p\mid 2m+2-k_0+(p-1)v
 \quad\Longleftrightarrow\quad v=2m+2-k_0+pt,\\
 \frac{2m+2-k_0+(p-1)v}{p}=2m+2-k_0+(p-1)t
\end{gathered}
\]

For the theta comparison, take
\(m=\deg\mathbf e_{d+\ell+1}^{(\eps)}\).  The first identity for
\(d\) gives
\[
 \deg\mathbf e_{\ell+1}^{(\eps')}=m-k_0+1,\qquad
 2(m-k_0+1)+2-(2-k_0)=2m+2-k_0
\]
Thus the two increments have the same summand in
\eqref{eq:compatibility-increment-sums}.
Subtracting the two sums in \eqref{eq:compatibility-increment-sums}, the positive interval
is \([m-k_0+2,m]\), of length \(k_0-1\), and the negative interval is
\([-\lfloor m/p\rfloor,-\lfloor (m-k_0+1)/p\rfloor-1]\).
The integers \(v\) in the first interval divisible in the above
sense correspond exactly to integers \(t\) in the second, since
\[
 \left\lceil\frac{m-k_0+2-(2m+2-k_0)}{p}\right\rceil
 =-\left\lfloor\frac mp\right\rfloor,
 \qquad
 \left\lfloor\frac{k_0-2-m}{p}\right\rfloor
 =-\left\lfloor\frac{(m-k_0+1)}p\right\rfloor-1
\]
None of these factors is zero.  Dividing each divisible positive
factor by \(p\) therefore cancels its full valuation against
\(1+v_p(2m+2-k_0+(p-1)t)\).  Only the \(k_0-1\) constant terms remain.
Thus, for every \(\ell\geq0\), we obtain
\begin{align}\label{eq:theta-coefficient-comparison}
 &\val(g^{(\eps)}_{d+\ell+1}(w_{k_0}))
  -\val(g^{(\eps)}_{d+\ell}(w_{k_0}))\nonumber\\
 &\quad=\val(g^{(\eps')}_{\ell+1}(w_{2-k_0}))
  -\val(g^{(\eps')}_{\ell}(w_{2-k_0}))+k_0-1
\end{align}
For the opposite-character comparison, take
\(m=\deg\mathbf e_{d-\ell+1}^{(\eps)}\).  The second identity for
\(d\) gives \(\deg\mathbf e_\ell^{(\eps'')}=k_0-2-m\).
Replacing \(v\) by \(-v\) in the second increment uses the identity
\[
 2(k_0-2-m)+2-k_0-(p-1)v
 =-\bigl(2m+2-k_0+(p-1)v\bigr)
\]
Since changing a sign does not change a valuation, we obtain
\begin{align*}
&\val(g_{d-\ell+1,\widehat{k_0}}^{(\eps)}(w_{k_0}))
 -\val(g_{d-\ell,\widehat{k_0}}^{(\eps)}(w_{k_0}))\\
&\quad+\val(g_{\ell,\widehat{k_0}}^{(\eps'')}(w_{k_0}))
 -\val(g_{\ell-1,\widehat{k_0}}^{(\eps'')}(w_{k_0}))\\
&=\sideset{}{'}\sum_{v=-(k_0-2-m)}^{m}
 \bigl(1+v_p(2m+2-k_0+(p-1)v)\bigr)\\
&\quad-\sideset{}{'}\sum_{v=-\lfloor m/p\rfloor}^{\lfloor(k_0-2-m)/p\rfloor}
 \bigl(1+v_p(2m+2-k_0+(p-1)v)\bigr)
\end{align*}
The terms at \(v=0\) have been inserted into both sums and cancel.
The divisible factors in the first interval map to the second
interval, because
\[
 \left\lceil\frac{-(k_0-2-m)-(2m+2-k_0)}{p}\right\rceil
 =-\left\lfloor\frac mp\right\rfloor,
 \qquad
 \left\lfloor\frac{m-(2m+2-k_0)}{p}\right\rfloor
 =\left\lfloor\frac{k_0-2-m}{p}\right\rfloor
\]
The same division argument shows that the difference of sums is
the number of nonzero factors in the first interval, namely
\(k_0-1\) minus one if a zero factor occurs there.
In the noncongruent case no zero factor occurs, giving
\[
\begin{split}
 &\val(g^{(\eps)}_{d+1-\ell}(w_{k_0}))
  -\val(g^{(\eps)}_{d-\ell}(w_{k_0}))\\
 &\quad+\val(g^{(\eps'')}_{\ell}(w_{k_0}))
  -\val(g^{(\eps'')}_{\ell-1}(w_{k_0}))=k_0-1
\end{split}
\]

The congruent case contains the new parity issue, so we give its
calculation.  For \(1\leq\ell\leq d\), apply the preceding calculation with
\(m=s_\eps+(p-1)(\ell-1)\) and
\(k_0-2-m=s_\eps+(p-1)(d-\ell)\).
The zero factor has index \(v=(k_0-2-2m)/(p-1)\), which is an integer
in the congruent case.  It belongs to \([-(k_0-2-m),m]\) exactly when
\[
 k_0-2\leq(p+1)m,\qquad k_0-2\leq(p+1)(k_0-2-m)
\]
Indeed, these are the two inequalities obtained by multiplying
the endpoint conditions by \(p-1\).  For the first inequality,
using \(k_0-2=(p-1)(k_{0\bullet}-\delta_\eps)+2s_\eps\) gives
\[
 (p+1)(\ell-1)\geq k_{0\bullet}-\delta_\eps-s_\eps
 \quad\Longleftrightarrow\quad \ell-1\geq d_{k_0}^{\ur}
\]
The last equivalence is the formula for \(d_{k_0}^{\ur}\), since
\(\lceil(k_{0\bullet}-\delta_\eps-s_\eps)/(p+1)\rceil
=d_{k_0}^{\ur}\).
The second inequality similarly gives \(d-\ell\geq d_{k_0}^{\ur}\).
Counting the nonzero factors therefore proves
\begin{equation}\label{eq:congruent-coefficient-comparison}
\begin{split}
 &\val(g_{d-\ell+1,\widehat{k_0}}(w_{k_0}))
 -\val(g_{d-\ell,\widehat{k_0}}(w_{k_0}))\\
 &\quad+\val(g_{\ell,\widehat{k_0}}(w_{k_0}))
 -\val(g_{\ell-1,\widehat{k_0}}(w_{k_0}))\\
 &\quad=
 \begin{cases}
 k_0-1,&1\leq\ell\leq d_{k_0}^{\ur}
       \text{ or }d-d_{k_0}^{\ur}<\ell\leq d,\\
 k_0-2,&d_{k_0}^{\ur}<\ell\leq d-d_{k_0}^{\ur}
 \end{cases}
\end{split}
\end{equation}
When \(d\) is odd and \(\ell=(d+1)/2\), one has \(2m+2-k_0=0\), so the
omitted zero is \(v=0\) in both complete intervals.  The cancellation
above still subtracts exactly one from \(k_0-1\), giving \(k_0-2\).
This verifies the fixed central index without evaluating a zero factor.

 Summing \eqref{eq:theta-coefficient-comparison} from \(0\) to
 \(\ell-1\) gives
 \[
  \val(g^{(\eps)}_{d+\ell}(w_{k_0}))
  -\val(g^{(\eps)}_d(w_{k_0}))
  =\val(g^{(\eps')}_\ell(w_{2-k_0}))+\ell(k_0-1)
  \geq\ell(k_0-1)
 \]
 In the noncongruent case, summing the opposite-character comparison
 gives
 \[
  \val(g^{(\eps)}_d(w_{k_0}))
  -\val(g^{(\eps)}_{d-\ell}(w_{k_0}))
  +\val(g^{(\eps'')}_\ell(w_{k_0}))
  =\ell(k_0-1)
 \]
 Hence every secant from the \(d\)-th coefficient to a coefficient on
 its left has slope at most \(k_0-1\), whereas every secant to a
 coefficient on its right has slope at least \(k_0-1\).  The \(d\)-th
 coefficient therefore lies on the Newton polygon, and the two
 coefficient comparisons give parts~(1) and~(2).  The first line of
 \eqref{eq:congruent-coefficient-comparison} gives part (3).  Summing
its second line across the new interval gives
\eqref{eq:ghost-duality}.  Finally, \cref{prop:old-bound} makes the two
endpoints of that interval vertices; hence its intervening lower hull
has slope \((k_0-2)/2\), proving part (4).
\end{proof}

\section{Vertices of the Newton polygon}

This section reconstruct the near-Steinberg theory in \cite[Section~5]{LTXZ1} where \(d_k^{\Iw}/2\) is always an intger in the generic case. Here it may be half-integral, and this does change the detailed computations but makes no substantial modifications: The key facts Proposition 5.16 and Theorem 5.19 of \cite{LTXZ1} remain true (see \cref{prop:shifting} and \ref{thm:vertices}). Thus one can transfer from the local polygon $\underline\Delta_k$ to Newton polygon of the ghost series at any \(w_\star\).

\begin{notation}\label{not:delta-lattice}
Fix \(k=k_\eps+(p-1)k_\bullet\).  For
\[
-\frac{d_k^{\new}}2\leq\ell\leq\frac{d_k^{\new}}2,
\qquad \frac{d_k^{\Iw}}2+\ell\in\mathbb Z
\]
put
\begin{equation}\label{eq:Delta-prime}
 \Delta'_{k,\ell}
 =\val\left(g_{\frac12d_k^{\Iw}+\ell,\widehat k}(w_k)\right)
  -\frac{k-2}{2}\ell
\end{equation}
Let \(\underline\Delta_k\) be the lower convex hull of these points
and let \(\Delta_{k,\ell}\) be its ordinate.  By
the ghost duality \eqref{eq:ghost-duality}, one has \(\Delta'_{k,\ell}=\Delta'_{k,-\ell}\).
The allowed \(\ell\)'s lie in \(\mathbb Z\) if \(d_k^{\Iw}\) is even
and in \(\mathbb Z+\tfrac12\) if it is odd. We would call such \(\ell\) an admissible index for simplicity.
In the calculations below, we frequently use the following notation
\[
 n=\frac12d_k^{\Iw}-\ell,\qquad a=2\ell-1 
\]
to simplify expressions and
\[
M=(p-1)n+s_\eps,\qquad q=\lfloor M/p\rfloor
\]
as in section~3.
\end{notation}

The following lemma gives the order-one difference estimate.

\begin{lemma}\label{lem:radial-gap}
For all \(\ell\geq 1\) an admissible
positive-side index, we have
\addtocounter{equation}{1}
\begin{equation}\label{eq:radial-gap}
 \Delta'_{k,\ell}-\Delta'_{k,\ell-1}
 \geq\frac{p+1}{2}+(p-1)(\ell-1)
\end{equation}
\end{lemma}

\begin{proof}
The ghost duality changes the left-hand side into
\[
 \frac{k-2}{2}-
 \bigl(\val(g_{n+1,\widehat k}(w_k))
       -\val(g_{n,\widehat k}(w_k))\bigr)
\]
Substitution of \(k_\bullet=d_k^{\Iw}-1+\delta_\eps\) and
\(n=d_k^{\Iw}/2-\ell\) in \eqref{eq:endpoints} gives
\[
\begin{split}
 k_\bullet-k_{\midd\bullet}(n)&=a,\\
 k_{\max\bullet}(n)-k_\bullet&=M-a,\\
 k_\bullet-k_{\min\bullet}(n)
 &=q+a
\end{split}
\]
The second expression is nonnegative even at the boundary
\(n=d_k^{\ur}\).
After inserting the three lengths in \eqref{eq:valuation-increment},
the positive and negative factorial contributions are respectively
\[
 M-1+\val((a-1)!)+\val((M-a)!)
\quad\text{and}\quad
 q+\val((q+a)!)
 -\val(a!)
\]
The identity \(k-2=2M+(p-1)a\) and the two preceding expressions give
\begin{align*}
 \Delta'_{k,\ell}-\Delta'_{k,\ell-1}
 ={}&1+\frac{p-1}{2}a+q+v_p((q+a)!)-v_p(a!)\\
 &-v_p((a-1)!)-v_p((M-a)!)
\end{align*}
Since \(v_p(M!)=q+v_p(q!)\), the last four valuation terms
together with \(q\) equal
\[
 v_p\left(\frac{M!(q+a)!}{q!a!(a-1)!(M-a)!}\right)
 =v_p\left(\binom M a\,a\binom{q+a} a\right)
\]
Thus we have the following identity
\addtocounter{equation}{-2}
\begin{equation}\label{eq:radial-gap-exact}
\Delta'_{k,\ell}-\Delta'_{k,\ell-1}
=\frac{p+1}{2}+(p-1)(\ell-1)+\val\!\left[
 \binom Ma\,a\binom{q+a}a\right]
\end{equation}
\addtocounter{equation}{1}
The last valuation is nonnegative and
\eqref{eq:radial-gap} follows.  The calculation uses only
that \(2\ell-1\) is a positive integer, so it includes the
half-integral lattice and excludes precisely \(\ell=1/2\).
\end{proof}

\begin{remark}
Equation \eqref{eq:radial-gap-exact} supplies the consecutive-difference
estimate used below.  The comparisons with other weights and the
strict vertex inequalities also require the present formulas
and are verified separately. At a half-integral center, the radius \(1/2\) contains no integer in its
open symmetric interval, so the first radius relevant to a
near-Steinberg range is \(3/2\).
\end{remark}

We next record only the separation estimates from
\cite[Corollaries~5.3, 5.9 and 5.10]{LTXZ1} that will be used in
Section~5.

\begin{corollary}\label{cor:radial-gap-weight}
Let \(\ell\geq1\) be an admissible index and \(k'\ne k\), and suppose that one of
\[
 d_{k'}^{\ur},\qquad d_{k'}^{\Iw}-d_{k'}^{\ur},\qquad
 \frac12d_{k'}^{\Iw}
\]
lies respectively in the open, open, or closed interval of radius
\(\ell\) about \(d_k^{\Iw}/2\).  If \(d_k^{\Iw}\) is even and
\(\ell=1\), then
\begin{equation}\label{eq:strong-gap-minimal}
 \Delta'_{k,1}-\Delta'_{k,0}-\val(w_k-w_{k'})\geq\frac12
\end{equation}
For every other admissible \(\ell\),
\begin{equation}\label{eq:strong-gap}
 \Delta'_{k,\ell}-\Delta'_{k,\ell-1}-\val(w_k-w_{k'})
 \geq\frac12+(p-1)(\ell-1)
 -\left\lfloor\log_p(2(p+1)\ell)\right\rfloor>0
\end{equation}
\end{corollary}

\begin{proof}
The endpoint equivalences in
\cref{lem:extremal} turn the first two open-interval conditions into
\[
\begin{aligned}
 d_{k'}^{\ur}:\quad
 &k_{\max\bullet}(n)<k'_\bullet
       \leq k_{\max\bullet}(n+a),\\
 d_{k'}^{\Iw}-d_{k'}^{\ur}:\quad
 &k_{\min\bullet}(n)\leq k'_\bullet
       <k_{\min\bullet}(n+a)
\end{aligned}
\]
Since \(k_\bullet=2n+a+\delta_\eps\), the first line is
equivalently
\[
 M-a+1
 \leq k'_\bullet-k_\bullet
 \leq M+pa
\]
The closed center condition is, directly,
\[
 -2\ell\leq k'_\bullet-k_\bullet\leq2\ell
\]
Since
\(k_{\min\bullet}(n)=2n+\delta_\eps-q\), the second endpoint
condition becomes
\[
 \left\lfloor\frac{M-a}{p}\right\rfloor+1
 \leq k_\bullet-k'_\bullet\leq q+a
\]
In particular all the integers in this interval are positive.

For this positive integer \(a\), let \(b\geq a\) and
\(z\in[b-a+1,b]\).  Counting multiples of \(p^e\) in this interval proves
\[
 v_p\binom ba\geq v_p(z)-\lfloor\log_p a\rfloor
\]
Indeed, for every \(e\) the count is at least
\(\lfloor a/p^e\rfloor\), the contribution from the denominator
\(a!\).  For \(\lfloor\log_p a\rfloor<e\leq v_p(z)\), the
denominator contributes zero and the numerator contributes at least
one.  Summing these nonnegative differences proves the inequality;
if this range of \(e\) is empty, its right side is nonpositive.

Put \(V=v_p\!\left(\binom M a\,a\binom{q+a}a\right)\).  In the first endpoint
case, \(M-a+1\leq k'_\bullet-k_\bullet\leq M+pa\).
If \(k'_\bullet-k_\bullet\leq M\), apply the preceding estimate to \(\binom Ma\).
If \(k'_\bullet-k_\bullet>M\) and \(p\mid(k'_\bullet-k_\bullet)\), then
\[
 q+1\leq\frac{k'_\bullet-k_\bullet}{p}\leq q+a
\]
so apply the estimate to \(\binom{q+a}a\), using
\(v_p(k'_\bullet-k_\bullet)=1+v_p((k'_\bullet-k_\bullet)/p)\).
If \(p\nmid(k'_\bullet-k_\bullet)\), its valuation is zero.  These three cases give respectively
\[
 v_p(k'_\bullet-k_\bullet)\leq V+\lfloor\log_p a\rfloor,\qquad
 v_p(k'_\bullet-k_\bullet)\leq V+1+\lfloor\log_p a\rfloor,\qquad v_p(k'_\bullet-k_\bullet)=0
\]
In the second endpoint case,
for \(k_\bullet-k'_\bullet\leq q\) one has \(M-a+1\leq p(k_\bullet-k'_\bullet)\leq M\); for \(k_\bullet-k'_\bullet>q\)
one has \(q+1\leq k_\bullet-k'_\bullet\leq q+a\).  The same two binomial estimates give
\[
 v_p(k_\bullet-k'_\bullet)\leq V+\lfloor\log_p a\rfloor-1
 \quad\hbox{or}\quad
 v_p(k_\bullet-k'_\bullet)\leq V+\lfloor\log_p a\rfloor
\]
Finally the center condition gives
\(v_p(k_\bullet-k'_\bullet)\leq\lfloor\log_p(a+1)\rfloor\).
Since \((p+1)(a+1)>pa\), all cases imply
\[
 v_p(k_\bullet-k'_\bullet)
 \leq V+\lfloor\log_p((p+1)(a+1))\rfloor
\]
Subtracting \(1+v_p(k'_\bullet-k_\bullet)\) from
\eqref{eq:radial-gap-exact} therefore gives the stronger bound
\[
 \Delta'_{k,\ell}-\Delta'_{k,\ell-1}-v_p(w_k-w_{k'})
 \geq\frac{p-1}{2}+(p-1)(\ell-1)
 -\lfloor\log_p(2(p+1)\ell)\rfloor
\]
For \(\ell\geq3/2\), the right side of \eqref{eq:strong-gap}
is positive: replacing the integer part by the logarithm gives
at \(\ell=3/2\) a value greater than
\(1/2+(p-1)/2-2=(p-4)/2>0\), and its derivative is
\(p-1-1/(\ell\log p)>0\).  This proves \eqref{eq:strong-gap}.

At an integral center and \(\ell=1\), one has \(a=1\).
The two endpoint conditions above place \(k'_\bullet-k_\bullet\) in
\[
 [M,M+p]
 \quad\text{or}\quad
 \left[-q-1,-\left\lfloor\frac{M+p-1}{p}\right\rfloor\right]
\]
whereas the center condition gives \(0<|k'_\bullet-k_\bullet|\leq2\).  In either endpoint
interval, division of \(M\) by \(p\) gives the elementary bound
\[
 v_p(k'_\bullet-k_\bullet)\leq1+v_p\!\left(M(q+1)\right)
\]
Write \(M=pq+r\), \(0\leq r<p\).  In the first interval a
multiple of \(p\) is \((q+1)p\) when \(r>0\), and is either \(pq=M\) or
\((q+1)p\) when \(r=0\).  Its valuation is therefore at most
\(1+v_p(M(q+1))\); a nonmultiple of \(p\) has valuation zero.
The second interval consists of \(-(q+1)\) when \(r>0\), and of
\(-q-1,-q\) when \(r=0\).  For the latter value,
\(v_p(q)=v_p(M)-1\), since \(M\geq a=1\).
These values give the same bound.  Formula
\eqref{eq:radial-gap-exact} at \(2\ell-1=1\), together with
\(v_p(w_k-w_{k'})=1+v_p(k'_\bullet-k_\bullet)\), now gives
\[
 \Delta'_{k,1}-\Delta'_{k,0}-v_p(w_k-w_{k'})
 \geq\frac{p-3}{2}\geq\frac12
\]
For the center interval the stronger lower bound \((p-1)/2\) follows
from \(v_p(k'_\bullet-k_\bullet)=0\).  This proves \eqref{eq:strong-gap-minimal}.  At a
half-integral center the first nonempty radius is \(3/2\), so no further
initial case occurs.
\end{proof}

\begin{lemma}\label{lem:convex-hull-difference}
For \(p\geq11\) and every admissible \(\ell>0\),
\begin{equation}\label{eq:convex-hull-difference}
 0\leq\Delta'_{k,\ell}-\Delta_{k,\ell}
 \leq3\bigl(\log_p(2\ell)\bigr)^2
\end{equation}
Moreover, \(\Delta'_{k,\ell}=\Delta_{k,\ell}\) when
\(0<2\ell<2p\) and \(2\ell\ne p\); at \(2\ell=p\) one has
\(\Delta'_{k,\ell}-\Delta_{k,\ell}\leq1\).
\end{lemma}

\begin{proof}
We first compute the second-order difference.  In this
calculation write \(M=pq+r\), \(0\leq r<p\), with the fixed notation above. The identity
\[
 \binom Ma\,a\binom{q+a}a
 =\frac{M!(q+a)!}{(M-a)!(a-1)!a!q!}
\]
and \(v_p(M!)=q+v_p(q!)\) give
\[
 v_p\left(\binom Ma\,a\binom{q+a}a\right)
 =q+v_p((q+a)!)-v_p((M-a)!)-v_p((a-1)!)-v_p(a!)
\]
At radius \(\ell+1\), the quantities \(n,a,M,q\) become respectively
\(n-1,a+2,M-p+1,q-1+e\), where \(e=1\) if \(r=p-1\) and
\(e=0\) otherwise.  Subtracting therefore gives
\begin{align*}
&\quad\Delta'_{k,\ell+1}-2\Delta'_{k,\ell}+\Delta'_{k,\ell-1}\\
 &=p-2+e
   +v_p\left(\prod_{h=1}^{1+e}(q+a+h)\right)
   +v_p\left(\prod_{h=0}^{p}(M-a-h)\right)\\
 &-v_p(a)-2v_p(a+1)-v_p(a+2)
\end{align*}
Since \(\ell+1\leq d_k^{\new}/2\), one computes that \(M-a\geq p+1\). Among the \(p+1\) consecutive positive integers \(M-a-p,\ldots,M-a\),
at least one is divisible by \(p\).  All the other terms on the first
line of valuations are nonnegative.  Consequently
\begin{equation}\label{eq:convex-second-corrected}
 \Delta'_{k,\ell+1}-2\Delta'_{k,\ell}+\Delta'_{k,\ell-1}
 \geq p-1-v_p(2\ell-1)-2v_p(2\ell)-v_p(2\ell+1)
\end{equation}

Suppose that \(\Delta'_{k,\ell}>\Delta_{k,\ell}\), since otherwise there is
nothing to prove.  Choose admissible indices \(0\leq u<\ell<v\) such that
\((u,\Delta'_{k,u})\) and \((v,\Delta'_{k,v})\) lie on
\(\underline\Delta_k\), whereas no point
\((x,\Delta'_{k,x})\) with \(u<x<v\) lies on \(\underline\Delta_k\).
Set \(D=v-u\), and subtract their common chord by putting
\[
 h_i=\Delta'_{k,u+i}-\frac{D-i}{D}\Delta'_{k,u}-\frac{i}{D}\Delta'_{k,v}
 \qquad(0\leq i\leq D)
\]
Then \(D\) is an integer at least two, \(h_0=h_D=0\), and
\(h_i>0\) for \(1\leq i<D\).  In particular, telescoping gives
\[
 \sum_{i=1}^{D-1}(h_{i+1}-2h_i+h_{i-1})=-h_1-h_{D-1}<0
\]

The integers in all the valuation terms in
\eqref{eq:convex-second-corrected}, for these interior indices, lie
between \(2u+1\) and \(2v-1\).  Choose an integer \(z_0\) in this
interval for which \(\nu =v_p(z_0)\) is maximal.  If \(z\ne z_0\) is
another integer in the interval, then
\[
 v_p(z)\leq v_p(z-z_0)
\]
Indeed, if \(v_p(z-z_0)<\nu \), the valuation of
\(z=z_0+(z-z_0)\) equals \(v_p(z-z_0)\); if it is greater than
\(\nu \), that valuation equals \(\nu \); and if it equals \(\nu \),
the maximality defining \(\nu \) excludes a valuation greater than
\(\nu \).  Summing separately to the left and right of \(z_0\), and
using \(v_p(m!)=\sum_{b\geq1}\lfloor m/p^b\rfloor\leq m/(p-1)\),
we obtain
\begin{align*}
 \sum_{z=2u+1}^{2v-1}v_p(z)
 &\leq \nu +v_p((z_0-2u-1)!)+v_p((2v-1-z_0)!)\\
 &\leq \nu +\frac{2D-2}{p-1}
\end{align*}
Each integer has total coefficient at most two in the sum of the
three valuation terms in \eqref{eq:convex-second-corrected}.  It follows
that
\begin{align*}
 &\sum_{i=1}^{D-1}
 \bigl(v_p(2u+2i-1)+2v_p(2u+2i)+v_p(2u+2i+1)\bigr)\\
 &\qquad\leq2\nu +\frac{4(D-1)}{p-1}
\end{align*}
Combining this with the strictly negative telescoping sum gives
\[
 \left(p-1-\frac4{p-1}\right)(D-1)<2\nu 
\]
Since \(p\geq11\), the coefficient on the left is at least
\(48/5>8\).  Hence \(D-1<\nu /4\), \(\nu \geq5\), and \(D\leq \nu /2\).

To bound the height, fix \(1\leq j<D\) and consider the function
\[
 i\longmapsto\frac{\min(i,j)(D-\max(i,j))}{D}
 \qquad(0\leq i\leq D)
\]
It is linear on each side of \(j\), is zero
at \(i=0,D\), and has second difference \(-1\) at \(i=j\).
For the latter assertion the three numerators are
\((j-1)(D-j),j(D-j),j(D-j-1)\), whose second difference is \(-D\).
These properties, and the fact that a sequence of zero second
differences with zero boundary values is zero, prove
\[
 h_i=-\sum_{j=1}^{D-1}
 \frac{\min(i,j)(D-\max(i,j))}{D}
 (h_{j+1}-2h_j+h_{j-1})
\]
Also
\[
 0\leq\frac{\min(i,j)(D-\max(i,j))}{D}\leq\frac D4
\]
since the two factors in the numerator have sum at most \(D\).  Applying
\eqref{eq:convex-second-corrected} and dropping its positive constant
term now yields
\[
 h_i\leq\frac D4\left(2\nu +\frac{4(D-1)}{p-1}\right)
 \leq\frac D4\frac{21\nu }{10}\leq\frac{21\nu ^2}{80}
\]
Finally, \(p^\nu \leq z_0\leq2v-1\) and \(v-\ell\leq D-1\) give
\[
 p^\nu \leq2\ell+2D-3<2\ell+\frac \nu 2-1
\]
For \(\nu \geq5\), write \(x=p^{\nu /2}\).  Bernoulli's inequality gives
\(x=(\sqrt p)^\nu \geq1+\nu (\sqrt p-1)>\nu /2\), and \(x\geq2\), so
\(p^\nu -p^{\nu /2}=x(x-1)\geq x>\nu /2\).  The preceding bound consequently
implies \(2\ell>p^{\nu /2}\).  Therefore
\[
 \Delta'_{k,\ell}-\Delta_{k,\ell}
 \leq\frac{21\nu ^2}{80}
 \leq\frac{21}{20}\bigl(\log_p(2\ell)\bigr)^2
 \leq3\bigl(\log_p(2\ell)\bigr)^2
\]

In fact \((\ell,\Delta'_{k,\ell})\) lies on \(\underline\Delta_k\)
for every admissible \(\ell\) with \(0<2\ell<2p\).
Otherwise the preceding argument would give \(\nu \geq5\) and
\(p^\nu <2p+\nu /2-1\).  This is impossible: at \(\nu =5\) the reverse
strict inequality follows from \(p^5>2p+3/2\), and increasing \(\nu \)
by one increases \(p^\nu -2p-\nu /2+1\) by
\((p-1)p^\nu -1/2>0\).  This proves, in particular, both small-index
assertions in the statement, including the half-integral case.
\end{proof}

\begin{remark}
The valuation estimate in this proof requires maximality on the entire
integer interval \([2u+1,2v-1]\), not just on the integers of one parity.
The terms \(v_p(2\ell-1)\) and \(v_p(2\ell+1)\) in
\eqref{eq:convex-second-corrected} cannot be discarded.  The argument
applies to both integral and half-integral radii without changing this
interval of integers.
\end{remark}

\begin{corollary}\label{cor:convex-gap}
Under the hypotheses of \cref{cor:radial-gap-weight},
\[
 \Delta_{k,\ell}-\Delta'_{k,\ell-1}-\val(w_k-w_{k'})
 \geq\ell-\frac12
\]
If \(\ell\geq2\), the left-hand side is at least
\(2\ell-\tfrac12\), and
\[
 \Delta_{k,\ell}-\Delta'_{k,\ell-1}\geq2\ell+\frac12
\]
\end{corollary}

\begin{proof}
For \(2\leq\ell<p\), the proof of
\cref{lem:convex-hull-difference} gives
\(\Delta_{k,\ell}=\Delta'_{k,\ell}\).  Moreover
\(2(p+1)\ell<2p(p+1)<p^3\), so \eqref{eq:strong-gap} gives
\[
 \Delta_{k,\ell}-\Delta'_{k,\ell-1}-v_p(w_k-w_{k'})
 -(2\ell-\tfrac12)
 \geq(p-3)\ell-p\geq p-6>0
\]
For \(\ell\geq p\), subtract \eqref{eq:convex-hull-difference} from
\eqref{eq:strong-gap}, and bound the integer part of the logarithm
by the logarithm itself.  It suffices to prove \(F(\ell)\geq0\), where
\[
 F(x)=(p-3)x-(p-2)-\log_p(2(p+1)x)-3(\log_p(2x))^2
\]
At \(x=p\), the inequalities \(2p(p+1)<p^3\) and \(2p<p^2\)
give \(F(p)>p^2-4p-13>0\).  Differentiating gives
\[
 F'(x)=p-3-\frac{1+6\log_p(2x)}{x\log p}
\]
Both \(1/x\) and \(\log_p(2x)/x\) decrease for \(x\geq p\), the
derivative of the latter being \((1-\log(2x))/(x^2\log p)<0\).
Thus \(F'(x)>p-3-13/(p\log p)>p-4>0\), proving the required bound.
For \(\ell=1\), use \eqref{eq:strong-gap-minimal} and the vanishing
convex-hull difference.  For \(\ell=3/2\), that difference also
vanishes, and \(3(p+1)<p^2\) gives from \eqref{eq:strong-gap}
a lower bound \(1/2+(p-1)/2-1=(p-2)/2\geq1\).
These are all initial cases.  Finally add
\(v_p(w_k-w_{k'})\geq1\) to the second bound to obtain the third.
\end{proof}

\begin{corollary}\label{cor:multi-gap}
Let \(k'\ne k\), and let \(\ell\leq\ell'\leq\ell''\) be nonnegative admissible indices with
\(\ell''>\ell\), and suppose that
\(d_{k'}^{\ur}\) or \(d_{k'}^{\Iw}-d_{k'}^{\ur}\) lies in the closed
interval of radius \(\ell'\) about \(d_k^{\Iw}/2\).  Then
\begin{equation}\label{eq:multi-gap}
 \Delta_{k,\ell''}-\Delta'_{k,\ell}
 -(\ell''-\ell')\val(w_k-w_{k'})
 \geq\frac12\bigl((\ell'')^2-\ell^2\bigr)
 +(\ell'-\ell)
 \left\lfloor\log_p(2(p+1)\ell'')+2\right\rfloor
\end{equation}
Independently of the existence of such a weight \(k'\), every pair of
admissible radii \(0\leq\ell<\ell''\) satisfies
\begin{equation}\label{eq:quadratic-gap}
 \Delta_{k,\ell''}-\Delta'_{k,\ell}
 \geq\frac12\bigl((\ell'')^2-\ell^2\bigr)
\end{equation}
At most two such weights satisfy
\[
 \val(w_k-w_{k'})\geq
 \left\lfloor\log_p(2(p+1)\ell'')+3\right\rfloor
\]
There is at most one endpoint of each type among these weights.  If
there are two distinct weights, their indicated endpoints have sum
\(d_k^{\Iw}\); in particular they lie on opposite sides of the center
unless both endpoints equal the center.
\end{corollary}

\begin{proof}
We first record two bounds, including their numerical verification:
for every positive-side radius \(r\),
\begin{align*}
 \Delta_{k,r}-\Delta'_{k,r-1}&\geq4r+1\\
 \Delta_{k,r}-\Delta'_{k,r-1}-v_p(w_k-w_{k'})&\geq4r-1
\end{align*}
The second bound requires the endpoint or center hypothesis of
\cref{cor:radial-gap-weight} at radius \(r\).  For \(r<p\), the
convex-hull difference is zero by the proof of
\cref{lem:convex-hull-difference}.  The first bound then follows from
\eqref{eq:radial-gap}, since
\[
 \frac{p+1}{2}+(p-1)(r-1)-(4r+1)
 =(p-5)r-\frac{p-1}{2}\geq\frac{p-9}{2}>0
\]
For the second bound at \(r=1\), the proof of
\eqref{eq:strong-gap-minimal} gives \((p-3)/2\geq4>3\).
For \(3/2\leq r<p\), use the stronger inequality proved immediately
before that initial calculation, and
\(\lfloor\log_p(2(p+1)r)\rfloor\leq2\).  The difference from
\(4r-1\) is at least
\[
 (p-5)r-\frac{p-3}{2}-2\geq p-8>0
\]
For \(r\geq p\), both bounds follow from the positivity of
\[
 (p-5)r-\frac{p-3}{2}-\log_p(2(p+1)r)
                -3(\log_p(2r))^2
\]
At \(r=p\) this is greater than \(p^2-6p-15>0\).  Its derivative
is \(p-5-(1+6\log_p(2r))/(r\log p)>p-6>0\), by the decreasing
function calculation in the proof of \cref{cor:convex-gap}.
For the first bound, its corresponding lower bound is larger by
\(\log_p(2(p+1)r)-1>0\).

For \(\ell<r\leq\ell'\) use the first bound.  For
\(\ell'<r\leq\ell''\), the prescribed endpoint lies strictly inside
the interval of radius \(r\), so use the second.  Since
\(\Delta'_{k,r}\geq\Delta_{k,r}\), telescoping these inequalities gives
\begin{align*}
 &\Delta_{k,\ell''}-\Delta'_{k,\ell}
       -(\ell''-\ell')v_p(w_k-w_{k'})\\
 &\quad\geq
 \sum_{r=\ell+1}^{\ell'}(4r+1)
 +\sum_{r=\ell'+1}^{\ell''}(4r-1)\\
 &\quad=2((\ell'')^2-\ell^2)+(\ell''-\ell)+2(\ell'-\ell)
\end{align*}
All sums have step one, including for half-integral indices.
Put \(L=\lfloor\log_p(2(p+1)\ell'')\rfloor\).  For
\(x=\ell''\geq1\), we have \(2(p+1)<p^2\) and
\(\log_p x\leq x-1\), the latter following by differentiation from
equality at \(x=1\).  Hence \(L\leq\ell''+1\), and
\begin{align*}
 &\frac32((\ell'')^2-\ell^2)+(\ell''-\ell)-(\ell'-\ell)L\\
 &\quad\geq(\ell''-\ell)
       \left(\frac32(\ell''+\ell)+1-L\right)\geq0
\end{align*}
This proves \eqref{eq:multi-gap}.  Taking \(\ell'=\ell''\) in this
calculation uses only the first, unconditional bound, and proves
\eqref{eq:quadratic-gap} as well as the version of
\eqref{eq:multi-gap} with no valuation term and no endpoint hypothesis.

For the last assertion, the valuation hypothesis implies
\(k'_\bullet\equiv k_\bullet\pmod{p^{L+2}}\), and
\[
 p^{L+2}>2p(p+1)\ell''>2(p+1)\ell''+p
\]
Write each candidate as
\(k'_\bullet=(p+1)h+s_\eps+\delta_\eps+1+j\), with
\(0\leq j\leq p\).  If its indicated endpoint is \(u'=d_{k'}^{\ur}\),
then
\[
 k'_\bullet=(p+1)u'+s_\eps+\delta_\eps-p+j
\]
Two such candidates with \(|u'-d_k^{\Iw}/2|\leq\ell'\) differ by at most
\(2(p+1)\ell'+p<p^{L+2}\), so they cannot be distinct and congruent
modulo \(p^{L+2}\).  If its endpoint is
\(v'=d_{k'}^{\Iw}-d_{k'}^{\ur}\), then
\[
 pk'_\bullet=(p+1)v'+p\delta_\eps-s_\eps-1-j
\]
The same bound applies to \(p\) times the difference of two
candidates, proving uniqueness for this type too.
For one candidate of each type, adding their two identities and
subtracting \((p+1)k_\bullet\) gives
\[
 k'_{\bullet,1}+pk'_{\bullet,2}-(p+1)k_\bullet
 =(p+1)(u'+v'-d_k^{\Iw})+j_1-j_2
\]
The left side is divisible by \(p^{L+2}\); the absolute value of the right
side is less than \(p^{L+2}\), so it is zero.  Since \(u'+v'-d_k^{\Iw}\) is an
integer and \(|j_1-j_2|\leq p\), it follows that
\(u'+v'=d_k^{\Iw}\) and \(j_1=j_2\).  This proves the assertion, including
the possibility of two endpoints at the center.
\end{proof}

\begin{definition}\label{def:near-steinberg}
Let \(w_\star\in\mathfrak m_{\Cp}\).  If \(d_k^{\Iw}\) is even, let
\(L_{w_\star,k}\) be the largest member of
\[
 \left\{1,2,\ldots,\frac12d_k^{\new}\right\}
\]
satisfying
\begin{equation}\label{eq:near-condition}
 \val(w_\star-w_k)\geq
 \Delta_{k,L_{w_\star,k}}-\Delta_{k,L_{w_\star,k}-1}
\end{equation}
If \(d_k^{\Iw}\) is odd, use instead
\[
 \left\{\frac{3}{2},\frac{5}{2},\ldots,\frac{1}{2}d_k^{\new}\right\}
\]
Either set is understood to be empty when its upper endpoint is smaller
than its first member.
When such a radius exists, put
\[
 \nS_{w_\star,k}=
 \left(\frac{1}{2}d_k^{\Iw}-L_{w_\star,k},
       \frac{1}{2}d_k^{\Iw}+L_{w_\star,k}\right)
\]
and write \(\overline{\nS}_{w_\star,k}\) for its closure; otherwise
put \(\nS_{w_\star,k}=\varnothing\).
\end{definition}

Equivalently, an integer \(n\) is in \(\nS_{w_\star,k}\) precisely
when
\[
 d_k^{\ur}<n<d_k^{\Iw}-d_k^{\ur}
\]
and
\[
 \val(w_\star-w_k)\geq
 \Delta_{k,|n-d_k^{\Iw}/2|+1}
 -\Delta_{k,|n-d_k^{\Iw}/2|}
\]
Indeed, the successive slopes of \(\underline\Delta_k\) are
nondecreasing, and the admissible radii differ by one.  Thus the
defining inequality at \(L_{w_\star,k}\) holds at every smaller
positive radius, while membership in the open interval is equivalent
to \(|n-d_k^{\Iw}/2|+1\leq L_{w_\star,k}\).

If the defining radius is not outermost, convexity shows that maximality
forces
\(\Delta'_{k,L_{w_\star,k}}=\Delta_{k,L_{w_\star,k}}\): otherwise the
same successive-difference inequality would hold at the next admissible
radius.  At the outermost radius this equality holds because the point
is an endpoint of the finite point set defining \(\underline\Delta_k\).
Thus the defining radius is always a vertex of
\(\underline\Delta_k\).

The omission of \(1/2\) is forced: at an odd-dimensional center that
radius contains no integer coefficient index.  The first nonempty
radius is \(3/2\), and at the outer radius \(d_k^{\new}/2\) the open
interval still excludes \(d_k^{\ur}\) and
\(d_k^{\Iw}-d_k^{\ur}\).

\begin{lemma}\label{lem:straight-local}
If \eqref{eq:near-condition} holds and \(w_\star\ne w_k\), the lower
convex hull of
\[
 \left(n,\val(g_{n,\widehat k}(w_k))
       +m_n(k)\val(w_\star-w_k)\right),
 \qquad n\in\overline{\nS}_{w_\star,k}\cap\mathbb Z
\]
is a line of slope \((k-2)/2\).  If \(w_\star=w_k\), the actual ghost
coefficients over the full new interval have a straight lower hull of
this slope; its interior coefficients are zero.
\end{lemma}

\begin{proof}
Write \(n=d_k^{\Iw}/2+\ell\).  After subtracting the affine function
of slope \((k-2)/2\), the ordinate is
\[
 \Delta'_{k,\ell}+
 \left(\frac12d_k^{\new}-|\ell|\right)\val(w_\star-w_k)
\]
Put \(L=L_{w_\star,k}\) and \(v=v_p(w_\star-w_k)\).
The point \((L,\Delta'_{k,L})\) lies on \(\underline\Delta_k\), so
\(\Delta'_{k,L}=\Delta_{k,L}\).  For \(0\leq |\ell|\leq L\),
convexity and \eqref{eq:near-condition} give
\[
 \Delta_{k,L}-\Delta_{k,|\ell|}
 =\sum_{r=|\ell|+1}^{L}(\Delta_{k,r}-\Delta_{k,r-1})
 \leq(L-|\ell|)v
\]
Using symmetry and \(\Delta'_{k,\ell}\geq\Delta_{k,|\ell|}\),
the difference between the displayed ordinate and its endpoint
ordinate is at least
\[
 \Delta_{k,|\ell|}-\Delta_{k,L}+(L-|\ell|)v\geq0
\]
Both endpoints have the same ordinate, proving the assertion.
At \(w_\star=w_k\), every strictly interior coefficient of the full
new interval vanishes because \(m_n(k)>0\).  Only its two endpoints
remain finite, and their secant slope is \((k-2)/2\) by
\eqref{eq:ghost-duality}.  One must not remove the vanishing factor
from the interior coefficients in this last assertion.
\end{proof}

\begin{proposition}\label{prop:shifting}
Let \(\nS_{w_\star,k}\) be a near-Steinberg range.
\begin{enumerate}[label=(\arabic*)]
\item If \(k'\ne k\) and
\[
 \val(w_{k'}-w_k)\geq
 \Delta_{k,L_{w_\star,k}}-\Delta_{k,L_{w_\star,k}-1}
\]
then \(d_{k'}^{\Iw}/2\) is outside the closed
range and \(d_{k'}^{\ur}\) and
\(d_{k'}^{\Iw}-d_{k'}^{\ur}\) are outside the open range.
\item For a weight \(k'\) satisfying the hypothesis of (1), the
function \(m_n(k')\) is affine in \(n\) on the closed range.
\item If \(\mathbf k\) contains \(k\) and all factors that vanish at
either evaluation point on the interval, then
\[
 \val(g_{n,\widehat{\mathbf k}}(w_\star))
 -\val(g_{n,\widehat{\mathbf k}}(w_k))
\]
is affine in \(n\) on the closed range.
\end{enumerate}
\end{proposition}

\begin{proof}
Part (1) follows directly from \cref{cor:convex-gap}, and therefore the closed near-Steinberg range lies inside one affine piece of every multiplicity function \(m_n(k')\) whenever \(k'\) satisfying the hypothesis of (1),
which proves (2), including a half-integral middle break point.  For
(3), after the prescribed factors are
removed, subtract the two evaluations:
\begin{equation}\label{eq:shift-difference}
 \sum_{k''\notin\mathbf k}m_n(k'')
 \bigl(\val(w_\star-w_{k''})
       -\val(w_k-w_{k''})\bigr)
\end{equation}
The parenthesis is zero when
\(\val(w_k-w_{k''})<\val(w_\star-w_k)\); in every remaining term,
\(m_n(k'')\) is affine by (2).  Hence \eqref{eq:shift-difference} is
affine on the range.
\end{proof}

\begin{corollary}\label{cor:straight-global}
Let \(\nS_{w_\star,k}\ne\varnothing\).
\begin{enumerate}[label=(\arabic*)]
\item Suppose \(g_n(w_\star)\ne0\) at every integer in the closed
range.  The lower convex hull of the coefficient points on this
closed range is a line.  Its slope belongs to
\(\mathbb Z+\mathbb Z\alpha\), where
\[
 \alpha=\max\bigl(\{0\}\cup
 \{\val(w_\star-w_{k'}):m_n(k')>0
       \text{ for some }n\in\nS_{w_\star,k}\}\bigr)
\]
At an integer evaluation weight its slope is integral.
\item Suppose \(w_\star=w_{k_0}\) with \(k_0\ne k\).  The lower
convex hull of
\[
 \bigl(n,v_p(g_{n,\widehat{k_0}}(w_{k_0}))\bigr),
 \qquad n\in\overline{\nS}_{w_{k_0},k}\cap\mathbb Z
\]
is a line of integral slope.  If \(m_n(k_0)>0\) at an interior
integer, the closed range is contained in
\([d_{k_0}^{\ur},d_{k_0}^{\Iw}-d_{k_0}^{\ur}]\), and
\(m_n(k_0)\) is positive at every interior integer.
\item If \(w_\star=w_k\), the range is the full open new interval.
Its two endpoint coefficients are nonzero and have secant slope
\((k-2)/2\); every interior coefficient vanishes.  This assertion
concerns the actual coefficients, not the coefficients with the
\(k\)-factor removed.
\end{enumerate}
\end{corollary}

\begin{proof}
In (1), apply \cref{prop:shifting}(3) with the factor at \(k\)
removed, and then restore it.  The resulting ordinates differ by an
affine function from those in \cref{lem:straight-local}.  The latter
have a straight lower hull of slope \((k-2)/2\), an integer since
\(k\) is even.  Each affine multiplicity occurring in the difference
has integral slope.  Among the distances from \(w_\star\) to the
finitely many roots in question, every value smaller than their
maximum is the distance between two classical roots, hence an
integer.  Thus the affine correction has slope in
\(\mathbb Z+\mathbb Z\alpha\).  At an integer evaluation weight all
the nonzero root distances are integral, proving the last assertion.

For (2), put \(v=v_p(w_{k_0}-w_k)\).  The defining inequality and
\cref{prop:shifting}(1)--(2) show that \(m_n(k_0)\) is affine on
the closed range and that neither endpoint of its new interval is
in the open range.  Consequently positivity at one interior integer
implies positivity throughout the open range and the stated
containment.  Remove both factors \(k,k_0\) when applying
\cref{prop:shifting}(3).  The resulting comparison is with
\[
 v_p(g_{n,\widehat k}(w_k))+m_n(k)v-m_n(k_0)v
\]
The last term is affine of integral slope.  All remaining valuation
differences are integral as well.  Now
\cref{lem:straight-local} proves the line and its integral slope.
This reasoning also covers the case in which \(m_n(k_0)\) is
identically zero on the range.  Part (3) is precisely the classical
case of \cref{lem:straight-local}.
\end{proof}

\begin{theorem}\label{thm:vertices}
For \(w_\star\in\mathfrak m_{\Cp}\):
\begin{enumerate}[label=(\arabic*)]
\item two near-Steinberg ranges are disjoint or one contains the other;
\item the integer abscissae of the vertices of
\(\NP(G^{(\eps)}(w_\star,-))\) are precisely the integers outside all
near-Steinberg ranges;
\item the slope on a maximal range is the slope described in
\cref{cor:straight-global}.
\end{enumerate}
\end{theorem}

\begin{proof}
We prove the interval and vertex assertions directly.  The argument has
the same order as \cite[Theorem~5.19]{LTXZ1}, but the calculation at an
integer abscissa below includes all three middle weights.

\emph{Comparison and containment of two ranges.}
Write \(L_i=L_{w_\star,k_i}\), and suppose their closures meet, with
\(L_1\geq L_2\) and \(k_1\ne k_2\).  Since
\(d_k^{\Iw}=k_\bullet+1-\delta_\eps\), their centers give
\[
 0<|k_{1\bullet}-k_{2\bullet}|\leq2(L_1+L_2)\leq4L_1
\]
For \(i=1,2\), the point \((L_i,\Delta'_{k_i,L_i})\) lies on
\(\underline\Delta_{k_i}\).
Consequently \eqref{eq:radial-gap} implies, for \(i=1,2\),
\[
 v_p(w_\star-w_{k_i})\geq
 \Delta_{k_i,L_i}-\Delta_{k_i,L_i-1}
 \geq\frac{p+1}{2}+(p-1)(L_i-1)
\]
For every \(L\geq1\),
\[
 p^{(p-1)/2+(p-1)(L-1)}>4L
\]
Indeed this holds at \(L=1\), and the logarithmic derivative of
the left side divided by \(L\) is
\((p-1)\log p-1/L>0\).  The ultrametric inequality therefore
excludes \(L_1=L_2\).  It also proves
\begin{equation}\label{eq:range-comparison-strict}
 \Delta_{k_1,L_1}-\Delta_{k_1,L_1-1}
   >v_p(w_\star-w_{k_2})
\end{equation}
For otherwise both distances would be at least the left side,
forcing \(|k_{1\bullet}-k_{2\bullet}|\) to exceed \(4L_1\).
In particular,
\[
 v_p(w_{k_1}-w_{k_2})=v_p(w_\star-w_{k_2}),\qquad
 \nS_{w_{k_1},k_2}=\nS_{w_\star,k_2}
\]

An endpoint of the larger range cannot lie in the smaller open
range.  If that endpoint is a new endpoint for \(k_1\), this follows
from \cref{prop:shifting}(1), applied to the smaller range and
\(k'=k_1\).  Otherwise that endpoint is strictly inside the new
interval for \(k_1\).  The same proposition excludes the center and
the new endpoints for \(k_1\) from the smaller range, so its closure
lies on one side of that center, inside the closed new interval.
By \cref{cor:straight-global}(2), the lower hull of
\(v_p(g_{j,\widehat{k_1}}(w_{k_1}))\) on this smaller interval is a
line.  Subtracting \((k_1-2)(j-d_{k_1}^{\Iw}/2)/2\) preserves that
property.  This contradicts the strict vertex at radius \(L_1\):
if this radius is not outermost, maximality in the defining
inequality makes the next hull slope strictly larger.  This proves
containment and hence (1).

For each fixed integer \(n\), only finitely many ranges can contain
\(n\), since \(d_k^{\ur}<n\) bounds \(k_\bullet\).  Thus every range
is contained in a maximal one.  We first prove (2) at a point
\(w_\star\) which is not a relevant classical root; put
\(V(j)=v_p(g_j(w_\star))\).  By \cref{cor:straight-global}(1),
each maximal range can be replaced by its endpoint chord, and
every intervening coefficient point is on or above that chord.
It remains to check strict increase of the slopes at every
integer outside these open ranges.

\emph{An integer not in the closure of a range.}
Fix such an \(n\geq1\).  The two positive sums in the second
difference \eqref{eq:valuation-second}, now evaluated at \(w_\star\),
have root indices
\[
 \bigl(k_{\max\bullet}(n-1),k_{\max\bullet}(n)\bigr],
 \qquad
 \bigl[k_{\min\bullet}(n-1),k_{\min\bullet}(n)\bigr)
\]
The first contains \(p+1\) integers and the second contains one or
two.  The negative indices are
\[
 2n+\delta_\eps-2,\quad 2n+\delta_\eps-1,\quad
 2n+\delta_\eps
\]
with multiplicities \(1,2,1\), respectively.  Cancel an index
occurring on both sides with its multiplicity.  A middle weight
whose multiplicity \(m_n\) is zero disappears in this cancellation:
this follows directly by evaluating its tent at \(n-1,n,n+1\).
The remaining positive multiset has at least \(p-3\) members and
the negative multiset has total multiplicity at most four.

If \(v_p(w_\star)<1\), every root distance equals \(v_p(w_\star)\);
strict positivity of the second difference follows from
\cref{prop:degree-increment}.  Suppose \(v_p(w_\star)\geq1\).
Distinct middle roots have mutual distance one, since their bullet
indices differ by one or two.  If all remaining negative distances
equal one, the positive multiset just counted makes the second
difference strictly positive.

Otherwise there is a unique remaining middle weight \(k\) with
\(v=v_p(w_\star-w_k)>1\).  Write \(c=2\) for the integral-center
weight and \(c=1\) for either half-integral-center weight.  The
other negative terms have total valuation at most \(4-c\).
Here \(m_n(k)>0\).  Since \(n\notin\nS_{w_\star,k}\), and since
the first two nonnegative coefficient points lie on \(\underline\Delta_k\),
symmetry and the defining inequality give
\begin{equation}\label{eq:unit-vertex-reference}
 \sum_{\text{remaining positive }k'}v_p(w_k-w_{k'})
   >cv+(4-c-\text{cancelled negative multiplicity})
\end{equation}
More explicitly, the root-deleted second difference at \(n\)
equals twice \(\Delta_{k,1}-\Delta_{k,0}\) when \(c=2\), and equals
\(\Delta_{k,3/2}-\Delta_{k,1/2}\) when \(c=1\).  Subtracting the
other, unit-distance negative terms gives exactly
\eqref{eq:unit-vertex-reference}.  Thus the required strict
inequality follows immediately if none of the positive root
distances decreases when \(w_k\) is replaced by \(w_\star\).

Suppose a positive distance decreases.  Such a root has
\(v_p(w_k-w_{k'})>v\) and \(v_p(w_\star-w_{k'})=v\).
If \(c=1\), this one term, together with at least \(p-4\) other
positive terms of valuation at least one, exceeds \(v+3\).
This proves the assertion in the half-integral case.

If \(c=2\), put \(b=2n+\delta_\eps-1=k_\bullet\).
The endpoint formulas give the following bijection:
\begin{equation}\label{eq:unit-vertex-endpoint-bijection}
 r=(p+1)b-ph,\qquad
 \begin{cases}
 k_{\min\bullet}(n-1)\leq h<k_{\min\bullet}(n),\\
 k_{\max\bullet}(n-1)<r\leq k_{\max\bullet}(n),\
                         r\equiv b\pmod p
 \end{cases}
\end{equation}
Indeed, substituting \(r=(p+1)(2n+\delta_\eps-1)-ph\) and
\eqref{eq:endpoints}, the inequalities for \(r\) are equivalent to
\[
 n+\delta_\eps-1+\frac{n-s_\eps-1}{p}
 \leq h<
 n+\delta_\eps+\frac{n-s_\eps}{p}
\]
Since \(h\) is integral, taking ceilings at both endpoints gives
\[
 n-1+\delta_\eps+
       \left\lceil\frac{n-1-s_\eps}{p}\right\rceil
 \leq h<
 n+\delta_\eps+\left\lceil\frac{n-s_\eps}{p}\right\rceil
\]
These are precisely the asserted inequalities for \(h\).
Conversely \(r\equiv b\pmod p\) makes \(h=((p+1)b-r)/p\)
integral, so the correspondence is a bijection.
Moreover \(r-b=p(b-h)\).  If \(v<2\),
one decreasing positive term contributes \(v\) and the other
positive terms contribute at least \(p-4>v+2\); this already gives
the required bound \(2v+2\).  We may therefore assume \(v\geq2\).
Choose a decreasing term of maximal distance from \(w_k\).
By \eqref{eq:unit-vertex-endpoint-bijection} it must be an \(r\)
in the upper interval; otherwise its partner there has larger
distance.  Its partner \(h\) has distance one less, hence greater
than \(v-1\geq1\), so was not cancelled.  At \(w_\star\) these two
terms have valuations \(v\) and strictly greater than \(v-1\).
There remain at least \(p-5\) unit contributions.  Their sum is
strictly greater than
\(2v-1+p-5>2v+2\), since \(p\geq11\).
This proves
\begin{equation}\label{eq:unit-vertex-strict}
 V(n+1)-V(n)>V(n)-V(n-1)
\end{equation}

\emph{An endpoint of one or two maximal ranges.}
If a maximal range ends at \(n\), denote its radius by \(L_-\);
otherwise put \(2L_-=1\), so the left comparison uses the adjacent
integer \(n-1\).  Define \(L_+\) in the same manner on the right.
At least one of these is a genuine radius.  We first suppose the
larger radius is on the right and write
\[
 k=k_+,\quad L=L_+,\quad n=d_k^{\Iw}/2-L,\quad
 v=v_p(w_\star-w_k)
\]
If both ranges are genuine, \eqref{eq:range-comparison-strict}
shows that \(L_+>L_-\), \(v>v_p(w_\star-w_{k_-})\), and
\(v_p(w_k-w_{k_-})=v_p(w_\star-w_{k_-})\).

Among weights \(k'\) satisfying \(v_p(w_k-w_{k'})\geq v\),
there is at most one with a new endpoint at \(n\), counting also
\(k\) itself if \(d_k^{\ur}=n\).  Here are the required integer
bounds.  Put \(Q=p^{\lceil v-1\rceil}\).  The radial bound gives
\[
 Q\geq p^{(p-1)/2+(p-1)(L-1)}>2(p+1)L+p
\]
The last inequality holds at \(L=1\) and persists as \(L\)
increases, by the same derivative comparison used above.
An unramified endpoint at \(n\) has
\[
 k'_\bullet=(p+1)n+s_\eps+\delta_\eps-p+j,\qquad 0\leq j\leq p
\]
and a reflected endpoint at \(n\) has
\[
 pk'_\bullet=(p+1)n+p\delta_\eps-s_\eps-1-j,\qquad0\leq j\leq p
\]
The congruence \(k'_\bullet\equiv k_\bullet\pmod Q\) allows at
most one of each type.  If both types occurred, subtracting
\((p+1)k_\bullet\) from the sum of these identities would give
a multiple of \(Q\) equal to
\[
 -2(p+1)L+j_1-j_2
\]
Its absolute value is positive and less than \(Q\), a
contradiction.  Denote the possible exceptional weight by \(k_1\).

For every other weight \(k'\ne k\) with \(v_p(w_k-w_{k'})\geq v\),
the multiplicity \(m_j(k')\) is affine on
\([n-2L_-,n+2L]\).  On the right this is
\cref{prop:shifting}.  On a genuine left range, the ultrametric
equality just proved makes its distance from \(k_-\) equal to
\(v_p(w_\star-w_{k_-})\), so the same proposition applies there.
No new endpoint remains at \(n\), by the choice of \(k_1\).
If the left interval has length one, there is no integer endpoint
in its interior; its only possible half-integral center would
have bullet distance \(2L+1<Q\) from \(k\), and so cannot be a
a weight satisfying this valuation inequality.  The center at \(n\) is excluded by
\cref{prop:shifting}(1).  These observations check affineness on
the entire interval, including the case of different parities
of the two genuine centers.

Subtracting an affine function, we may therefore compute the
two secants using the ordinates
\begin{equation}\label{eq:three-endpoint-ordinates}
 \begin{gathered}
 v_p(g_{j,\widehat k}(w_k))+m_j(k)v
 -m_j(k_1)\bigl(v_p(w_k-w_{k_1})-v_p(w_\star-w_{k_1})\bigr),
 \\
 j=n-2L_-,n,n+2L
 \end{gathered}
\end{equation}
The last term is included only when \(k_1\) exists and differs
from \(k\).  To justify the subtraction, expand each coefficient
as a sum over roots.  Terms with distance from \(w_k\) less than
\(v\) have unchanged valuations.  Every remaining multiplicity
outside \(\{k,k_1\}\) is affine by the preceding paragraph.
This proves \eqref{eq:three-endpoint-ordinates} without any
unstated comparison of Newton polygons.

If \(d_k^{\ur}=n\), then \(k_1=k\), and \(m_j(k)=0\) at all three
displayed indices.  The right secant has slope \((k-2)/2\), by
ghost duality.  By \cref{prop:old-bound}, the left secant has slope
at most \(\lfloor(k-2)/(p+1)\rfloor\): its right endpoint is the
old-range endpoint on the polygon and its left coefficient is
on or above that polygon.  Their difference is positive.
Here \(k>2\), since weight two has no nonempty near-Steinberg
range in this case.

Otherwise the left comparison interval lies inside
\([d_k^{\ur},d_k^{\Iw}/2]\).  For a genuine left range its interior
cannot contain \(d_k^{\ur}\), by \cref{prop:shifting}; for a unit
interval this follows from integral endpoints and \(n>d_k^{\ur}\).
Thus \(L+2L_-\leq d_k^{\new}/2\).
If there is no \(k_1\), right secant minus left secant equals
\[
 \frac{\Delta'_{k,L+2L_-}-\Delta'_{k,L}}{2L_-}-v>0
\]
Indeed \(\Delta'_{k,L}=\Delta_{k,L}\), and convexity bounds this
quotient below by \(\Delta_{k,L+1}-\Delta_{k,L}>v\), the last
inequality being maximality of \(L\).

If \(k_1\ne k\) exists, its multiplicity has slope jump \(+1\)
at \(n\), regardless of which new endpoint it has there.
The difference of the two secants is consequently at least
\[
 \frac{\Delta'_{k,L+2L_-}-\Delta'_{k,L}}{2L_-}
       -v_p(w_k-w_{k_1})
\]
Here we used \(v_p(w_\star-w_{k_1})\geq v\).
For each admissible \(r=L+1,\ldots,L+2L_-\), its indicated
endpoint lies strictly inside the interval of radius \(r\).
By \cref{cor:convex-gap},
\[
 \Delta'_{k,r}-\Delta'_{k,r-1}
       -v_p(w_k-w_{k_1})
 \geq\Delta_{k,r}-\Delta'_{k,r-1}
       -v_p(w_k-w_{k_1})\geq r-\tfrac12>0
\]
Summing proves the required strict inequality.

If the larger range is on the left, reflect the three abscissae
about \(n\).  The two secant slopes change signs and exchange
places, so their difference has the same expression and the
same positive lower bounds.  Symmetry of \(\Delta'_k\) supplies
the radial calculation on that side.  In the outermost case
\(n=d_k^{\Iw}-d_k^{\ur}\), the right secant is at least
\(k-1-\lfloor(k-2)/(p+1)\rfloor\), by the reflected old-range
bound in \cref{prop:compatibility}; subtracting the left secant
\((k-2)/2\) is again strictly positive.  This accounts for the
only nonsymmetric endpoint estimate.

We have proved, in all endpoint cases,
\begin{equation}\label{eq:adjacent-range-secants}
 \frac{V(n)-V(n-2L_-)}{2L_-}
 <
 \frac{V(n+2L_+)-V(n)}{2L_+}
\end{equation}
Together with \eqref{eq:unit-vertex-strict}, this says that the
polygon obtained by retaining precisely the integers outside
the open ranges has strictly increasing successive slopes.
Its chords lie below every removed point, by
\cref{cor:straight-global}(1), so it is exactly the lower convex
hull.  This proves (2) at nonclassical points.

At \(w_\star=w_{k_0}\), the same conclusion follows with all
finite endpoint ordinates unchanged by a sufficiently small
nonclassical perturbation.  To be precise, for a fixed \(n\)
only finitely many weights can have a range meeting \(n\) or an
adjacent comparison interval: their unramified endpoints bound
their bullet indices.  For these weights other than \(k_0\),
all distances are constant on a small enough punctured disc.
The radius at \(k_0\) becomes its full new radius there.
Thus the relevant ranges and their finite endpoint ordinates
are unchanged.  Their strict secant inequalities persist.
An index where \(g_n(w_{k_0})=0\) is inside the full
\(k_0\)-range and cannot be a vertex.  This proves (2) also at
the classical point, without taking a hull of an empty set of
finite points on a smaller range.

Finally a maximal range at a classical point either is the full
range belonging to that point, or has no vanishing coefficient
on its closure.  Otherwise (1) and the full classical range
would contradict maximality; equality of distinct ranges was
already excluded.  Part (3) therefore follows from
\cref{cor:straight-global}(3) or (1), respectively.
\end{proof}

\begin{remark}
The preceding proof treats the three middle weights in
\eqref{eq:valuation-second} separately.  Their negative multiplicities
are \(1,2,1\); the middle multiplicity requires the explicit endpoint
bijection \eqref{eq:unit-vertex-endpoint-bijection}.  In the comparison
of adjacent ranges, the stronger radial bound gives uniqueness of an
endpoint weight satisfying the indicated valuation inequality even
at radius one.  These calculations
replace the corresponding case distinctions in
\cite[proof of Theorem~5.19]{LTXZ1}.
\end{remark}

\begin{corollary}\label{cor:integral-slopes}
For every integer \(k\), all slopes of
\(\NP(G^{(\eps)}(w_k,-))\) are integers; the evaluation parameter
\(w_k=\exp(p(k-2))-1\) need not be one of the relevant classical
roots occurring in the coefficients of \(G^{(\varepsilon)}\).
\end{corollary}

\begin{proof}
For every relevant root \(w_{k'}\),
\(\val(w_k-w_{k'})=1+\val(k-k')\) is an integer, with the usual
interpretation \(+\infty\) if the two parameters agree.  Hence a segment of
horizontal length one has integral slope because its endpoint valuations are
integers.  Every longer segment is a maximal near-Steinberg range; its slope
belongs to \(\mathbb Z+\mathbb Z\alpha\) by
\cref{thm:vertices} and part~(1) of \cref{cor:straight-global} when its endpoint and
interior coefficients are finite, and the number \(\alpha\) is one
of these integral distances.  If it is the full range at the evaluation
weight itself, use instead \cref{cor:straight-global}(3), whose slope
is \((k-2)/2\).  A smaller range contained in that vanishing interval
is not a maximal range and is not used in this argument.
Thus all slopes are integral, independently of the parity of
\(d_k^{\Iw}\).
\end{proof}

\begin{proposition}\label{prop:delta-vertices}
For an admissible \(0\leq\ell<d_{k_0}^{\new}/2\), the following are
equivalent:
\begin{enumerate}[label=(\arabic*)]
\item \((\ell,\Delta'_{k_0,\ell})\) is not a vertex of
\(\underline\Delta_{k_0}\);
\item \(d_{k_0}^{\Iw}/2+\ell\in\nS_{w_{k_0},k_1}\) for some \(k_1>k_0\);
\item \(d_{k_0}^{\Iw}/2-\ell\in\nS_{w_{k_0},k_2}\) for some \(k_2<k_0\).
\end{enumerate}
Every slope of \(\underline\Delta_{k_0}\) is integral, and every edge
of horizontal length greater than one comes from a symmetric
near-Steinberg range.
\end{proposition}

\begin{proof}
Consider the finite coefficient ordinates
\[
 v_p(g_{n,\widehat{k_0}}(w_{k_0})),\qquad
 d_{k_0}^{\ur}\leq n\leq d_{k_0}^{\Iw}-d_{k_0}^{\ur}
\]
Subtracting \((k_0-2)(n-d_{k_0}^{\Iw}/2)/2\) gives the ordinates
\(\Delta'_{k_0,n-d_{k_0}^{\Iw}/2}\).  We check the convex-hull argument for
these finite ordinates; the actual coefficients, which vanish
inside the new interval, are not used.

Suppose \(n=d_{k_0}^{\Iw}/2+\ell\), \(\ell\geq0\), lies in a range
\(\nS_{w_{k_0},k_1}\) with \(k_1\ne k_0\).  Its defining
inequality allows \cref{prop:shifting}(1) with \(k'=k_0\).
Thus its closure does not contain \(d_{k_0}^{\Iw}/2\), and its open interval
contains neither new endpoint for \(k_0\).  Since \(n\) is
strictly inside the new interval, this proves
\begin{equation}\label{eq:deleted-range-containment}
 k_1>k_0,\qquad
 \overline{\nS}_{w_{k_0},k_1}
 \subseteq[d_{k_0}^{\Iw}/2,d_{k_0}^{\Iw}-d_{k_0}^{\ur}]
\end{equation}
Here the integer left endpoint is strictly greater than \(d_{k_0}^{\Iw}/2\)
when \(d_{k_0}^{\Iw}\) is even.  By \cref{cor:straight-global}(2),
the ordinates \(v_p(g_{n,\widehat{k_0}}(w_{k_0}))\) on that closed range have a straight
lower hull of integral slope.  Its endpoint chord shows that
an interior index cannot be a strict vertex, either before or
after subtracting the displayed affine function.  This proves
(2) implies (1).  The identical containment argument on the
left half proves (3) implies (1).

We prove the converse by checking the retained coefficient points
on the right half.  Retain those integers \(n\geq d_{k_0}^{\Iw}/2\) outside
every open range belonging to \(k_1>k_0\); include the outer
endpoint.  The relevant ranges are nested by \cref{thm:vertices}(1)
and lie in this half by \eqref{eq:deleted-range-containment}.
Only finitely many occur.  Replace each maximal one by its
endpoint chord, as justified in the preceding paragraph.

At the smallest nonnegative radius, the point is a strict vertex
of the full symmetric hull.  Indeed all positive-side increments
in \eqref{eq:radial-gap} are positive, so every chord starting at
the first nonnegative point has positive slope.  At an integral
center symmetry makes the incoming slopes negative; at a
half-integral center the two central points are joined by a
horizontal segment of length one and the first outgoing slope
is positive.  No foreign range even has a closure containing
either central integer: if it did, with radius \(L\), the
distance of its bullet index from \(k_{0\bullet}\) would be at
most \(2L+1\), whereas its defining inequality would force
that nonzero distance to be at least
\(p^{(p-1)/2+(p-1)(L-1)}>2L+1\).

Let \(n>d_{k_0}^{\Iw}/2\) now be a retained interior integer.  If no foreign
range closure contains \(n\), use the unit comparison proved in
\eqref{eq:unit-vertex-strict}, with \(v_p(g_{n,\widehat{k_0}}(w_{k_0}))\)
in place of \(V(n)\).
Its hypotheses and its root calculation remain valid for the
following explicit reasons.  The three abscissae \(n-1,n,n+1\)
are in the right half of the closed new interval.  On them
\(m_j(k_0)=d_{k_0}^{\Iw}-d_{k_0}^{\ur}-j\) is affine, so removing
this factor contributes zero to the second difference.
Furthermore \(k_0\) is not an endpoint weight at this interior
\(n\), and is not one of the three middle weights there.
The positive multisets, the possible negative multiplicities
\(1,2,1\), and the bijection
\eqref{eq:unit-vertex-endpoint-bijection} are therefore unchanged.
For any remaining middle weight \(k\), the failure of
\(n\in\nS_{w_{k_0},k}\) gives the same strict first-increment
inequality.  All the distances used are finite.  The two cases
according as a positive distance decreases, including the
paired endpoints when the middle multiplicity is two, give
\[
\begin{aligned}
 &v_p(g_{n+1,\widehat{k_0}}(w_{k_0}))
   -v_p(g_{n,\widehat{k_0}}(w_{k_0}))\\
 &\qquad>v_p(g_{n,\widehat{k_0}}(w_{k_0}))
   -v_p(g_{n-1,\widehat{k_0}}(w_{k_0})).
\end{aligned}
\]

If \(n\) is an endpoint of one or two maximal foreign ranges,
use their endpoint abscissae, or the adjacent integer when
there is no range on that side, exactly as in
\eqref{eq:adjacent-range-secants}.  By
\eqref{eq:deleted-range-containment}, the whole comparison
interval is in the right half of the new interval for \(k_0\).
In the shift to the weight \(k\) of its larger range, remove
\(k_0\) as well as \(k\).  Relative to
\eqref{eq:three-endpoint-ordinates}, this subtracts the term
\[
 m_j(k_0)v_p(w_k-w_{k_0})
\]
from all three reference ordinates.  It is finite and affine
in \(j\), so it changes both secant slopes by the same amount.
The exceptional endpoint weight \(k_1\), if present, cannot
be \(k_0\), because \(n\) is inside its new interval.
Thus the uniqueness calculation and all three cases of the
endpoint comparison apply: at a new endpoint for \(k\) use
the old-range bound and its reflection; in the absence of an
exceptional weight use maximality of its radius; otherwise
sum \cref{cor:convex-gap} over the indicated consecutive radii.
Their positive differences are unaltered by the affine term.
This proves the strict secant comparison for the ordinates
\(v_p(g_{n,\widehat{k_0}}(w_{k_0}))\).

The retained points consequently form a strictly convex polygon
below all omitted points.  Subtracting the affine function
does not change its vertices.  Therefore every nonvertex on
the right belongs to one of the foreign ranges, proving
(1) implies (2).  Symmetry
\(\Delta'_{k_0,\ell}=\Delta'_{k_0,-\ell}\) and the same argument
on the left prove (1) implies (3).

Finally each edge of length greater than one is a whole maximal
foreign range, and hence is symmetric about its own center.
Its slope before the affine subtraction is integral by
\cref{cor:straight-global}(2).  The subtracted slope
\((k_0-2)/2\) is integral.  Edges of length one also have
integral slope, since the difference of their root-deleted
coefficient valuations is integral; the half-integral central
edge has slope zero.  This proves both remaining assertions.
\end{proof}

The following proposition is the analogue of
\cite[Proposition~4.7]{LTXZ2}.  It is the precise estimate needed when the
truncated Taylor terms in the Lagrange interpolation of a Fredholm
coefficient are evaluated at an arbitrary point of weight space.  The proof
of the cited proposition uses the near-Steinberg results of
\cite[Section~5]{LTXZ1} with integral centre; here those steps are replaced by
\cref{thm:vertices,prop:delta-vertices,cor:multi-gap}.

\begin{proposition}\label{prop:lagrange-term-bound}
Assume \(p\geq11\).  Fix \(n\), a ghost zero \(w_k\) of \(g_n\), and
\(0\leq i<m_n(k)\).  Suppose \(A\in\mathbb C_p\) satisfies
\begin{equation}\label{eq:separation-hypothesis}
 v_p(A)\geq
 \Delta_{k,\frac12d_k^{\new}-i}
 -\Delta'_{k,\frac12d_k^{\new}-m_n(k)}
\end{equation}
Then, for every \(w_\star\in\mathfrak m_{\mathbb C_p}\),
\[
 \left(n,
 v_p\bigl(A(w_\star-w_k)^i g_{n,\widehat k}(w_\star)\bigr)\right)
\]
lies on or above \(\NP(G^{(\varepsilon)}(w_\star,-))\).  It lies strictly
above when \((n,v_p(g_n(w_\star)))\) is a vertex.

If \(k_0\ne k\) is another ghost weight for \(g_n\), the same assertion holds
at \(w_{k_0}\) after removing the \(k_0\)-factor on both sides; equivalently,
with \(n=d_{k_0}^{\Iw}/2+\ell\),
\begin{equation}\label{eq:separation-two-centres}
 v_p(A)+(i-m_n(k))v_p(w_{k_0}-w_k)+\Delta'_{k_0,\ell}
 \geq\Delta_{k_0,\ell}
\end{equation}
\end{proposition}

\begin{proof}
We follow the proof of both assertions of
\cite[Proposition~4.7]{LTXZ2}. The shifting argument is supplied by
\cref{prop:shifting}, and the required assertions concerning
near-Steinberg ranges and vertices follow from
\cref{thm:vertices,prop:delta-vertices}. The estimates from
\cite[Proposition~2.22]{LTXZ2} are replaced by
\cref{cor:multi-gap} and the bounds established in its proof.
We explain the numerical adjustment required in Cases B2 and B3;
the remaining arguments are unchanged.

Put
\[
 x=\left|n-\frac12d_k^{\Iw}\right|
  =\frac12d_k^{\new}-m_n(k)
\]
As in the cited proof, suppose that
\(n\notin\nS_{w_\star,k}\) and that
\((x,\Delta'_{k,x})\) is not a vertex of
\(\underline\Delta_k\). Choose a maximal range
\(\nS_{w_k,k'}\) containing \(n\), and write
\[
 L'=L_{w_k,k'},
 \qquad \gamma=v_p(k_\bullet-k'_\bullet)
\]
The estimate established in the proof of \cref{cor:multi-gap} gives
\[
 1+\gamma
 \geq\Delta_{k',L'}-\Delta_{k',L'-1}
 \geq\Delta_{k',L'}-\Delta'_{k',L'-1}
 \geq4L'+1
\]
Hence \(L'\leq\gamma/4\). In the present setting, the distance
between the two centres is
\[
 \left|\frac12d_k^{\Iw}-\frac12d_{k'}^{\Iw}\right|
 =\frac12|k_\bullet-k'_\bullet|
\]
Consequently,
\[
 x>\frac12p^\gamma-L'
 \geq\frac12p^\gamma-\frac{\gamma}{4}
 >\gamma+1
\]
where the last inequality holds for \(p\geq11\) and \(\gamma\geq1\).

When \(v_p(w_\star-w_k)\leq1+\gamma\), the hypothesis and \cref{cor:multi-gap} yield
\[
 v_p(A)
 \geq\Delta_{k,x+m_n(k)-i}-\Delta'_{k,x}
 \geq\frac12\bigl((x+m_n(k)-i)^2-x^2\bigr)
 >(m_n(k)-i)v_p(w_\star-w_k)
\]
This supplies the strict inequality used in Case B2 and the
corresponding part of Case B3. When \(v_p(w_\star-w_k)>1+\gamma\), the remainder
of Case B3 uses only shifting and convexity and applies unchanged.
Thus the proof of \cite[Proposition~4.7]{LTXZ2} establishes both
assertions.
\end{proof}

\section{The \texorpdfstring{$U_p$}{U-p}-operator and the local ghost theorem}\label{sec:operator}

We now compare the ghost series with the characteristic power series of the
$U_p$-operator.  We retain the construction and notation of
\cite[Sections~2 and 3]{LTXZ2}.  Thus
\(\mathrm S^{\dagger,(\varepsilon)}_{\widetilde H}\) is the space of abstract
overconvergent forms on the $\varepsilon$-disc and
\[
 C_{\widetilde H}^{(\varepsilon)}(w,T)
 =1+\sum_{n\geq1}c_n^{(\varepsilon)}(w)T^n
\]
is the characteristic power series of $U_p$. We use \(T\) for the power series variable, reserving \(t\) for Taylor degree
in the coefficient estimates below.

Put, for the rest of this section,
\[
 a_i=s_\varepsilon+(i-1)(p-1),\qquad
 \mathbf e_i^{(\varepsilon)}=e^*z^{a_i}\quad(i\geq1)
\]
These vectors form the power basis of
\(\mathrm S^{\dagger,(\varepsilon)}_{\widetilde H}\).  We write
\(U^{\dagger,(\varepsilon)}\) for the matrix of \(U_p\) in this basis, with
the convention that its \((r,c)\)-entry is the coefficient of
\(\mathbf e_r^{(\varepsilon)}\) in
\(U_p\mathbf e_c^{(\varepsilon)}\).  Thus
\[
 C_{\widetilde H}^{(\varepsilon)}(w,T)
 =\det(1-TU^{\dagger,(\varepsilon)})
\]
This convention is needed below because row and column degrees play different
roles in the minor normalizations.

\begin{lemma}\label{lem:weak-hodge-block}
The matrix \(U^{\dagger,(\varepsilon)}\) belongs to
\(\mathrm M_\infty(\cO\langle w/p\rangle)\), and its row indexed by
\(\mathbf e_r^{(\varepsilon)}\) belongs to
\(p^{a_r/2}\cO\langle w/p\rangle\).  At an integral weight \(w_k\),
that row belongs to \(p^{a_r}\cO\).

If \(d=d_k^{\Iw}(\varepsilon(1\times\omega^{2-k}))\), then
\(U^{\dagger,(\varepsilon)}(w_k)\) is block upper triangular: its
upper-left \(d\times d\) block is the classical Iwahori matrix and its entries
in rows \(>d\) and columns \(\leq d\) vanish.
\end{lemma}

\begin{proof}
The two row estimates are proved exactly as in
\cite[Proposition~3.2]{LTXZ2}. The double-coset action on a monomial
belongs to \(\cO\langle w/p\rangle\langle p^{1/2}z\rangle\), and its
specialization at \(w_k\) belongs to \(\cO[\![pz]\!]\). These calculations
do not depend on the projective type; in the present basis the row degree
is \(a_r=s_\varepsilon+(r-1)(p-1)\).
The theta exact sequence identifies the span of the first \(d\) power
vectors with the \(U_p\)-stable classical subspace, as in
\cite[Proposition~2.12(1)]{LTXZ2}. This gives the asserted zero entries.
\end{proof}

The proof follows the three steps of \cite[Sections~4--6]{LTXZ2}.  First,
the classical identity expressing $U_p$ as an Atkin--Lehner term plus an
operator factoring through the unramified space gives the required corank at
every classical weight.  Second, the modified Mahler basis gives a uniform
valuation estimate for every finite minor.  Third, Lagrange interpolation and
cofactor expansion convert these two estimates into the coefficient bound that
implies equality of Newton polygons.  We give the formulas for each step and
isolate the two changes from the generic case: the basis degrees form one
progression, and the reflection of an odd-dimensional classical matrix has one
fixed index.

\subsection{The \texorpdfstring{$p$}{p}-stabilization of classical forms}

We first specialize the \(U_p\)-operator at a classical weight.  The explicit
\(\Pi\)-action from \cref{prop:Pi} gives the Atkin--Lehner anti-diagonal
matrix, while \(p\)-stabilization identifies the complementary term with a
map through the unramified space.  The fixed central index when
\(d_k^{\Iw}\) is odd is included in this calculation.

For a fixed relevant classical weight \(k\), put temporarily
\[
 d=d_k^{\Iw},\qquad
 a_i=s_\varepsilon+(i-1)(p-1)
\]
Thus \((\mathbf e_i^{(\varepsilon)})_{i\geq1}\) is the power basis, its first \(d\)
elements form a basis of \(\mathrm S_k^{\Iw}\), and
\[
 a_{d+1-i}=k-2-a_i\qquad(1\leq i\leq d)
\]

\begin{proposition}\label{prop:classical-matrix}
With respect to the power basis,
\begin{equation}\label{eq:AL-action}
 \operatorname{AL}_{(k,\widetilde\varepsilon_1)}
  (\mathbf e_i^{(\varepsilon)})
  =-p^{a_{d+1-i}}\mathbf e_{d+1-i}^{(\varepsilon)}
  \qquad(1\leq i\leq d)
\end{equation}
Moreover,
\begin{equation}\label{eq:pstabilization}
 U_p=\iota_2\circ\operatorname{proj}_1-
       \operatorname{AL}_{(k,\widetilde\varepsilon_1)}
\end{equation}
where $\operatorname{proj}_1$ takes values in $\mathrm S_k^{\ur}$.  Hence the
matrix of $U_p$ is the sum of the negative of the anti-diagonal matrix in
\eqref{eq:AL-action} and a matrix of rank at most $d_k^{\ur}$.
\end{proposition}

\begin{proof}
By definition, for $h\in\widetilde H$ and $P\in\cO[z]_{\leq k-2}$,
\[
 (\operatorname{AL}_{(k,\widetilde\varepsilon_1)}f)(h)
 =f(h\Pi^{-1})|_\Pi,
 \qquad (P|_\Pi)(z)=(pz)^{k-2}P\bigl((pz)^{-1}\bigr).
\]
By \cref{prop:Pi}, $e\Pi^{-1}=-e$.  Since
$\mathbf e_i^{(\varepsilon)}(e)=z^{a_i}$, we obtain
\[
 (\operatorname{AL}_{(k,\widetilde\varepsilon_1)}
   \mathbf e_i^{(\varepsilon)})(e)
 =-p^{k-2-a_i}z^{k-2-a_i}
 =-p^{a_{d+1-i}}\mathbf e_{d+1-i}^{(\varepsilon)}(e).
\]
Evaluation at $e$ determines the form, proving \eqref{eq:AL-action}.
Formula \eqref{eq:pstabilization} is
the double-coset identity \cite[Proposition~3.5]{LTXZ2}, whose proof uses no
genericity parameter.  The rank assertion follows from its factorisation through
$\mathrm S_k^{\ur}$.

If \(d\) is odd, the index \((d+1)/2\) is fixed by reflection.  Its
degree is $(k-2)/2$, so its Atkin--Lehner entry is $-p^{(k-2)/2}$.  This entry is
counted once in every rank calculation below.
\end{proof}

For finite subsets $\underline\zeta,\underline\xi\subset\mathbb Z_{\geq1}$ of
cardinality $n$, in increasing order, put as in
\cite[Definition--Proposition~3.23]{LTXZ2}
\[
\begin{aligned}
 r_{\underline\zeta\times\underline\xi}(k)
 &=\#\{i\leq d_k^{\Iw}:i\in\underline\xi,
          \ d_k^{\Iw}+1-i\in\underline\zeta\},\\
 s_{\underline\xi}(k)
 &=\#\{i\in\underline\xi:i>d_k^{\Iw}\}
\end{aligned}
\]

\begin{proposition}\label{prop:corank}
The corank at $w=w_k$ of
$U^{\dagger,(\varepsilon)}(\underline\zeta\times\underline\xi)$ is at least
\[
 n-d_k^{\ur}-r_{\underline\zeta\times\underline\xi}(k)
   -s_{\underline\xi}(k)
\]
For $\underline\zeta=\underline\xi=\{1,\ldots,n\}$ and \(m_n(k)>0\), this
lower bound is \(m_n(k)\).  Consequently
\[
 p^{-\deg g_n^{(\varepsilon)}}g_n^{(\varepsilon)}(w)
 \mid\det U^{\dagger,(\varepsilon)}(1,\ldots,n)
 \quad\text{in }\cO\langle w/p\rangle
\]
\end{proposition}

\begin{proof}
Apply the rank argument of
\cite[Definition--Proposition~3.23]{LTXZ2}, using
\cref{lem:weak-hodge-block,prop:classical-matrix}. The nonclassical
columns contribute at most \(s_{\underline\xi}(k)\), and the classical
columns contribute at most
\(d_k^{\ur}+r_{\underline\zeta\times\underline\xi}(k)\).

The assertions for the initial principal minor are the analogue of
\cite[Corollary~3.11]{LTXZ2}, with the same proof.
For the initial principal minor with \(m_n(k)>0\), one has
\(n<d_k^{\Iw}\), and the anti-diagonal part has rank
\(\max\{0,2n-d_k^{\Iw}\}\). Thus
\[
 n-d_k^{\ur}-\max\{0,2n-d_k^{\Iw}\}=m_n(k).
\]
This calculation applies to both parities of \(d_k^{\Iw}\); a fixed
anti-diagonal index contributes one to the rank.
The determinant-corank argument in \cite[Corollary~3.11]{LTXZ2} gives the required
order of vanishing at every root of \(g_n\).
Finally, \(p^{-\deg g_n}g_n(pz)\) is monic over \(\cO\);
division in \(\cO\langle z\rangle\), with \(z=w/p\), therefore gives
the asserted integral divisibility.
\end{proof}

\begin{lemma}\label{lem:first-coefficient}
The least relevant classical weight for $\varepsilon$ at which the Iwahori space is nonzero is
$k=2+2s_\varepsilon$.  At that weight
\[
 d_k^{\Iw}=1,\qquad d_k^{\ur}=0,
 \qquad U_p\mathbf e_1^{(\varepsilon)}=p^{s_\varepsilon}
       \mathbf e_1^{(\varepsilon)}
\]
In particular, $c_1^{(1\times1)}(0)$ is a unit and $g_1^{(1\times1)}=1$.
\end{lemma}

\begin{proof}
Take $k_\bullet=\delta_\varepsilon$.  Then
$k=k_\varepsilon+(p-1)k_\bullet=2+2s_\varepsilon$, and the two dimension
formulae give $d_k^{\Iw}=1$ and $d_k^{\ur}=0$.  Thus
\eqref{eq:pstabilization} gives $U_p=-\operatorname{AL}$, while
\eqref{eq:AL-action} gives
$\operatorname{AL}(\mathbf e_1)=-p^{s_\varepsilon}\mathbf e_1$.
For \(s_\varepsilon=0\), one has \(w_2=0\) and the first diagonal
entry of \(U^\dagger(0)\) is one.  By
\cref{lem:weak-hodge-block}, every other diagonal entry has valuation
at least \((i-1)(p-1)\) at index \(i\geq2\).  The convergent trace
therefore satisfies
\[
 c_1(0)=-\operatorname{tr}U^\dagger(0)
       =-1-\sum_{i\geq2}U^\dagger_{i,i}(0)\equiv-1\pmod p
\]
Hence \(c_1(w)\) has unit constant term.
Finally, \cref{prop:degree-increment} gives $\deg g_1=0$.
\end{proof}

The identity $U\operatorname{AL}U=p^{k-1}\operatorname{AL}$ used for two distinct
Iwahori characters in \cite[Proposition~3.12(3)]{LTXZ1} does not apply to the self-dual
line at $s_\varepsilon=0$.  Applying it there would incorrectly give slope $1/2$ at
$k=2$.

\subsection{The modified Mahler estimate}

We next establish the entrywise estimate for the infinite \(U_p\)-matrix
required in \cite[Section~5]{LTXZ2}.  The modified Mahler basis is unchanged;
the simplification is that only the single sequence of degrees
\(s_\varepsilon+(i-1)(p-1)\) has to be retained.

Define
\[
 f_1(z)=\frac{z^p-z}{p},\qquad f_{j+1}(z)=\frac{f_j(z)^p-f_j(z)}p
\]
If $m=\sum_{j\geq0}m_jp^j$, $0\leq m_j<p$, put
\[
 \mathbf m_m(z)=z^{m_0}f_1(z)^{m_1}f_2(z)^{m_2}\cdots,
 \qquad
 \mathbf f_i^{(\varepsilon)}
 =e^*\mathbf m_{\deg\mathbf e_i^{(\varepsilon)}}(z)
\]
These are the modified Mahler functions and associated basis of
\cite[Section~3.14]{LTXZ2}.  For the base-$p$ digits $m_j,n_j$, put
$D(m,n)=\#\{j\geq0:n_{j+1}>m_j\}$.
Write $\mathbf C^{(\varepsilon)}=(\mathbf f_i^{(\varepsilon)})_{i\geq1}$
and let $U_{\mathbf C}^{(\varepsilon)}$ denote the matrix of $U_p$ in this basis.

\begin{proposition}\label{prop:mahler}
The functions $\mathbf m_m$ form an orthonormal basis of
$\mathcal C^0(\mathbb Z_p;\mathbb Z_p)$.  The change from the usual Mahler basis is
upper triangular over $\mathbb Z_p$ with unit diagonal.  For the matrices $P_{m,n}$
and $Q_{m,n}$ of \cite[Section~3.24]{LTXZ2},
\begin{equation}\label{eq:refined-halo}
 P_{m,n},Q_{m,n}\in
 p^{D(m,n)}p^{m-\lfloor n/p\rfloor}\cO\langle w/p\rangle
\end{equation}
In particular, with $a_i=s_\varepsilon+(i-1)(p-1)$,
\[
 (U_{\mathbf C}^{(\varepsilon)})_{i,j}
 \in p^{D(a_i,a_j)+a_i-\lfloor a_j/p\rfloor}\cO\langle w/p\rangle.
\]
\end{proposition}

\begin{proof}
The first assertions are \cite[Lemma~3.15]{LTXZ2}, and
\eqref{eq:refined-halo} is \cite[Proposition~3.26]{LTXZ2}.  They are calculations
on $\mathcal C^0(\mathbb Z_p)$ and use no genericity parameter.

By \cref{prop:Pi}, evaluation at $e$ identifies the space of forms with
the functions satisfying $h(\alpha z)=\alpha^{s_\varepsilon}h(z)$ for
$\alpha\in\Delta$.  By \cite[Lemma~3.15(1)]{LTXZ2},
\[
 \mathbf m_n(\alpha z)=\alpha^n\mathbf m_n(z)
 \qquad(\alpha\in\Delta).
\]
Thus, as in \cite[Notation~3.16]{LTXZ2}, the modified Mahler basis of this
space consists precisely of $e^*\mathbf m_{a_i}$, where
$a_i=s_\varepsilon+(i-1)(p-1)$.

By \cite[equation~(2.9.1)]{LTXZ1}, the $U_p$-action under evaluation at $e$
is a limit of finite $\cO$-linear combinations of slash actions $|_\gamma$, where
\[
 \gamma\in
 \begin{pmatrix}p\mathbb Z_p&\mathbb Z_p\\
                 p\mathbb Z_p&\mathbb Z_p^\times\end{pmatrix},
 \qquad \det\gamma\in p\mathbb Z_p^\times.
\]
For each such action, $P_{m,n}$ is its coefficient from $\mathbf m_n$ to
$\mathbf m_m$.  Estimate \eqref{eq:refined-halo} is preserved under these
finite sums and coefficientwise limits, as in the proof of
\cite[Proposition~3.19]{LTXZ2}.  Taking $m=a_i$ and $n=a_j$ therefore gives
the asserted estimate for
$(U_{\mathbf C}^{(\varepsilon)})_{i,j}$.
\end{proof}

Let $Y^{(\varepsilon)}$ be the change from $\mathbf C^{(\varepsilon)}$
to the normalized power basis.  The estimates for $Y^{(\varepsilon)}$ and its
inverse in \cite[Lemma~3.17]{LTXZ2} depend only on the degrees and apply unchanged.

\begin{proposition}\label{prop:hodge}
For every $n\geq1$,
\begin{equation}\label{eq:hodge-equality}
 \deg g_n^{(\varepsilon)}
 =\sum_{i=1}^n\left(
 \deg\mathbf e_i^{(\varepsilon)}-
 \left\lfloor\frac{\deg\mathbf e_i^{(\varepsilon)}}p\right\rfloor\right)
\end{equation}
Consequently the error $\boldsymbol\delta$ in
\cite[equation~(5.10.4), in Proposition~5.10]{LTXZ2} is zero.
\end{proposition}

\begin{proof}
For $\deg\mathbf e_{n+1}=s_\varepsilon+n(p-1)$,
\[
 \deg\mathbf e_{n+1}-\left\lfloor\frac{\deg\mathbf e_{n+1}}p\right\rfloor
 =n(p-2)+s_\varepsilon+\left\lceil\frac{n-s_\varepsilon}{p}\right\rceil
\]
which is $\deg g_{n+1}-\deg g_n$ by \cref{prop:degree-increment}.  Summation
proves \eqref{eq:hodge-equality}.
\end{proof}

For two subsets
$\underline\lambda=\{\lambda_1<\cdots<\lambda_n\}$ and
$\underline\eta=\{\eta_1<\cdots<\eta_n\}$ of
$\mathbb Z_{\geq1}$, write
\[
 \deg(\underline\lambda)=\sum_{i=1}^n
   \deg\mathbf e_{\lambda_i}^{(\varepsilon)},\qquad
 \deg(\underline\eta)=\sum_{i=1}^n
   \deg\mathbf e_{\eta_i}^{(\varepsilon)}
\]
and write
\[
 \deg\mathbf e_{\lambda_i}^{(\varepsilon)}
   =\sum_{j\geq0}\lambda_{i,j}p^j,
 \qquad
 \deg\mathbf e_{\eta_i}^{(\varepsilon)}
   =\sum_{j\geq0}\eta_{i,j}p^j
\]
Following \cite[Notation~3.27]{LTXZ2}, put
\[
\begin{split}
D_{\leq\alpha}^{(\varepsilon)}(\underline\lambda,j)
 &=\#\{i:\lambda_{i,j}\leq\alpha\},\\
D_{=0}^{(\varepsilon)}(\underline\lambda,j)
 &=D_{\leq0}^{(\varepsilon)}(\underline\lambda,j),\\
D^{(\varepsilon)}(\underline\lambda,\underline\eta)
 &=\sum_{j\geq0}\max\{D_{=0}^{(\varepsilon)}(\underline\lambda,j)
       -D_{=0}^{(\varepsilon)}(\underline\eta,j+1),0\},\\
\mathbb D^{(\varepsilon)}(\underline\lambda,\underline\eta)
 &=\sum_{j\geq0}\max_{0\leq\alpha\leq p-2}
   \{D_{\leq\alpha}^{(\varepsilon)}(\underline\lambda,j)
       -D_{\leq\alpha}^{(\varepsilon)}(\underline\eta,j+1),0\}
\end{split}
\]
The digits in these definitions are the digits of the degrees, not
of the indices.

\begin{proposition}\label{prop:det-mahler}
Assume \(p\geq11\).  For all $\underline\lambda$ and
$\underline\eta$ of cardinality $n$,
\begin{equation}\label{eq:det-mahler}
\begin{split}
v_p\!\left(\det U_{\mathbf C}^{(\varepsilon)}
  (\underline\lambda\times\underline\eta)\right)
\geq{}&\deg g_n^{(\varepsilon)}
 +\frac{\deg(\underline\lambda)-\deg(\underline\eta)}2\\
&+\sum_{i=1}^n v_p\!\left(
 \frac{(\deg\mathbf e_{\lambda_i}^{(\varepsilon)})!}
      {(\deg\mathbf e_{\eta_i}^{(\varepsilon)})!}\right)
\end{split}
\end{equation}
\end{proposition}

\begin{proof}
We follow the determinant estimate in the proof of
\cite[Proposition~5.5]{LTXZ2}. By the determinant expansion in
\cite[Corollary~3.29]{LTXZ2}, \eqref{eq:refined-halo} gives
\begin{equation}\label{eq:tuple-halo}
\begin{split}
v_p\!\left(\det U_{\mathbf C}^{(\varepsilon)}
  (\underline\lambda\times\underline\eta)\right)
\geq{}&\mathbb D^{(\varepsilon)}(\underline\lambda,\underline\eta)\\
&+\sum_{i=1}^n\left(\deg\mathbf e_{\lambda_i}^{(\varepsilon)}
 -\left\lfloor\frac{\deg\mathbf e_{\eta_i}^{(\varepsilon)}}p
 \right\rfloor\right).
\end{split}
\end{equation}
This argument does not depend on the set of retained degrees.
We verify the changes in the subsequent numerical estimate.

Here the degrees are precisely the nonnegative integers congruent to
\(s_\varepsilon\) modulo \(p-1\). The digit interchanges in
\cite[Lemma~A.8]{LTXZ2} preserve this single residue class, because
\(p^j\equiv1\pmod{p-1}\). Moving a digit in position \(j+1\) to an
empty position \(j\) does not increase the integer, so it also
preserves the truncation at \(\deg\mathbf e_n\). Consequently, with
\(\underline n=\{1,\ldots,n\}\),
\[
 D_{=0}(\underline n,j)\leq D_{=0}(\underline n,j+1),
 \qquad D(\underline n,\underline n)=0.
\]
Writing \(\deg\mathbf e_n=\sum b_jp^j\), the inverse interchange
preserves the truncation when
\(b_{j+1}=p-1,\ b_j>0\), or when \(b_j=b_{j+1}=0\).
The corresponding consecutive zero counts are then equal.
These are the equality cases needed below. Unlike in the reference,
\(\boldsymbol\delta=0\) by \cref{prop:hodge}; hence the initial
row and column sets already satisfy the desired estimate.

We use \(D\leq\mathbb D\) for the rest of the proof. As in
\cite[equation~(5.11.2)]{LTXZ2}, Legendre's formula reduces the
claim to
\[
\begin{split}
D(\underline\lambda,\underline\eta)
+\sum_{i=1}^n\left(
 \frac{\deg\mathbf e_{\lambda_i}+\deg\mathbf e_{\eta_i}}2
 +v_p\!\left(\left\lfloor\frac{\deg\mathbf e_{\eta_i}}p
                  \right\rfloor!\right)\right)
\geq \deg g_n+\sum_{i=1}^n v_p((\deg\mathbf e_{\lambda_i})!).
\end{split}
\]
Increasing a column index by one increases its degree by \(p-1\);
thus the quotient of that degree by \(p\), rounded down, increases
by zero or one. The carry argument of
\cite[equation~(5.13.2)]{LTXZ2} therefore applies unchanged and
reduces the assertion to \(\underline\eta=\underline n\).

For \(\underline\lambda=\{1,\ldots,n-1,n+1\}\), let \(\gamma\)
be the largest valuation of an integer in
\((\deg\mathbf e_n,\deg\mathbf e_{n+1}]\).
This interval has length \(p-1\), and its factorial quotient has
valuation \(\gamma\). If \(\gamma\leq1\), the degree contribution
\((p-1)/2\) suffices. If \(\gamma\geq2\), crossing the unique
multiple of \(p^\gamma\) changes the digits in positions
\(1,\ldots,\gamma-1\) from \(p-1\) to zero. The equality cases
above give one contribution to \(D(\underline\lambda,\underline n)\)
from each position \(1,\ldots,\gamma-2\). Hence
\[
 D(\underline\lambda,\underline n)+\frac{p-1}{2}
 \geq\gamma-2+\frac{p-1}{2}\geq\gamma.
\]
This proves the required exceptional row case with \(D\), not merely
with \(\mathbb D\).

For every other noninitial row set, follow the replacement in
\cite[Lemma~5.13 and its proof]{LTXZ2}: let \(n_-\) be its
least omitted index, replace \(\lambda_n\) by \(n_-\), and denote
the resulting set by \(\underline\lambda'\). Put
\[
 H=(\lambda_n-n_-)(p-1),\qquad
 \gamma=\max_{\deg\mathbf e_{n_-}<m\leq\deg\mathbf e_{\lambda_n}}v_p(m),
 \qquad (p-1)p^{\delta-1}<H\leq(p-1)p^\delta.
\]
Here \(\lambda_n-n_-\geq2\) and \(\delta\geq1\).
The factorial estimate used in the cited proof gives
\[
 v_p\!\left(\frac{(\deg\mathbf e_{\lambda_n})!}
                  {(\deg\mathbf e_{n_-})!}\right)
 \leq\gamma+\left\lfloor\frac{H-2}{p-1}\right\rfloor.
\]
For completeness, the digit bounds needed for this degree sequence are
\[
 D(\underline\lambda,\underline n)-D(\underline\lambda',\underline n)
 \geq
 \begin{cases}
 -\gamma-1,&\gamma\leq\delta,\\
 \gamma-2\delta-2,&\gamma>\delta.
 \end{cases}
\]
In the first case, only positions \(0,\ldots,\gamma\) can change,
each decreasing \(D\) by at most one. In the second, the interval
of length \(H<p^\gamma\) crosses exactly one multiple of
\(p^\gamma\). The digits of its upper and lower endpoints satisfy
\[
 \alpha_j=\beta_j\ (j>\gamma),\quad
 \alpha_\gamma=\beta_\gamma+1\ne0,\quad
 \alpha_j=0,\ \beta_j=p-1\ (\delta<j<\gamma),\quad
 \beta_\delta\ne0.
\]
The counting in \cite[equations~(5.13.10)--(5.13.11)]{LTXZ2}
now applies with the equality cases verified above: positions
\(\delta+1,\ldots,\gamma-2\) contribute one each, position
\(\gamma\) contributes zero, and positions \(\delta,\gamma-1\)
contribute nonnegatively. The remaining \(\delta\) positions
contribute at least \(-\delta\). This proves the second bound.

It remains to check the two sufficient inequalities
\[
 \frac H2-\left\lfloor\frac{H-2}{p-1}\right\rfloor
 \geq
 \begin{cases}
 2\gamma+1,&\gamma\leq\delta,\\
 2\delta+2,&\gamma>\delta.
 \end{cases}
\]
Writing \(H=d(p-1)\), their left side is \(d(p-3)/2+1\), with
\(d\geq2\). It is at least \(p-2\), which handles
\(\gamma\leq1\) in the first case and \(\delta=1\) in the second.
For the remaining cases it is greater than
\[
 \frac{p-3}{2}p^{\delta-1}\geq4\cdot11^{\delta-1}
 >2\delta+2\qquad(\delta\geq2),
\]
which also exceeds \(2\gamma+1\) when \(\gamma\leq\delta\).
Every replacement decreases the sum of the indices, so induction
proves the numerical estimate for all row sets. The column reduction
and \(\mathbb D\geq D\) then give \eqref{eq:det-mahler} from
\eqref{eq:tuple-halo}.
\end{proof}

\begin{corollary}\label{cor:power-minor-bound}
For all row and column sets \(\underline\zeta,\underline\xi\) of cardinality
\(n\),
\begin{equation}\label{eq:power-minor-bound}
 p^{\frac12(\deg(\underline\xi)-\deg(\underline\zeta))}
 \det U^{\dagger,(\varepsilon)}
       (\underline\zeta\times\underline\xi)
 \in p^{\deg g_n^{(\varepsilon)}}\cO\langle w/p\rangle
\end{equation}
This is the global Tate-algebra bound used in the interpolation
argument below. Together with the definition of the interpolation
coefficients, it gives the degree-dependent bound for the remainder, as explained
in \cref{prop:taylor-coefficients}.
\end{corollary}

\begin{proof}
Use \(U^\dagger=YU_{\mathbf C}Y^{-1}\) and the calculation in the proof of \cite[Proposition~5.10, equation~(5.10.3)]{LTXZ2}.
The estimates for \(Y\) and \(Y^{-1}\) in
\cite[Lemma~3.17]{LTXZ2}, combined with \cref{prop:det-mahler},
give each summand valuation at least
\[
 \deg g_n+\frac12
 \bigl(\deg(\underline\zeta)-\deg(\underline\xi)\bigr)
\]
the factorial valuations and the degrees of the two intermediate
index sets cancel. The convergence argument in the cited proof
uses the same change-of-basis estimates and the weak Hodge bound,
which is \cref{lem:weak-hodge-block} here. This proves
\eqref{eq:power-minor-bound}.
\end{proof}

\subsection{Lagrange Interpolation and Control of Taylor Coefficients}

We now combine the matrix estimate and the classical corank to
bound the Taylor coefficients of every finite minor. The interpolation
and cofactor arguments follow \cite[Sections~5--6]{LTXZ2}.
We give the numerical estimates that change for the present
dimension formulae and include the case of complementary minors
whose ghost multiplicity is zero.

Following \cite[Notation~5.1 and equation~(5.5.2)]{LTXZ2}, we have the following Lagrange interpolation identity:
\begin{equation}\label{eq:lagrange-minor}
\det U^{\dagger,(\varepsilon)}
(\underline\zeta\times\underline\xi)=\sum_{m_n(k)>0}
 A_k^{(\underline\zeta\times\underline\xi)}(w)
 g_{n,\widehat k}^{(\varepsilon)}(w)
 +h_{\underline\zeta\times\underline\xi}(w)g_n^{(\varepsilon)}(w)
\end{equation}
with \(h_{\underline\zeta\times\underline\xi}(w)\in E\langle w/p\rangle\) and each $A_{k,i}^{(\underline\zeta\times\underline\xi)}$ the truncation of the Taylor expansion of $\det U^{\dagger,(\varepsilon)}(\underline\zeta\times\underline\xi)/g_{n,\widehat k}^{(\varepsilon)}(w)$ at \(w=w_k\) at degree \(m_n(k)\). More concretely, if we write
\[
 A_k^{(\underline\zeta\times\underline\xi)}(w)
 =\sum_{i=0}^{m_n(k)-1}
 A_{k,i}^{(\underline\zeta\times\underline\xi)}(w-w_k)^i
\]
and expand formally 
\[
 \det U^{\dagger,(\varepsilon)}(\underline\zeta\times\underline\xi)/g_{n,\widehat k}^{(\varepsilon)}(w)
=\sum_{i\geq0}a_{k,i}^{(\underline\zeta\times\underline\xi)}(w-w_k)^i
\]
then we have $A_{k,i}^{(\underline\zeta\times\underline\xi)}
=a_{k,i}^{(\underline\zeta\times\underline\xi)}$ for $0\leq i<m_n(k)$. 

The bound \eqref{eq:power-minor-bound} gives the estimate for the quotient
\(h_{\underline\zeta\times\underline\xi}\)
in \cref{prop:taylor-coefficients}(1).
It remains to estimate the finitely many Taylor coefficients.

We first prove the analogue of \cite[Proposition~5.5]{LTXZ2}. The two-weight estimate uses \cref{cor:multi-gap} and the centre-distance formula for the
present dimensions. We also state the expansion at a non-ghost
centre, which will be needed for smaller minors.

All Taylor coefficient indices are integers, including in inequalities
such as \(i<d_k^{\new}/2\) when \(d_k^{\new}\) is odd.

\begin{proposition}
\label{prop:taylor-coefficients}
Assume \(p\geq11\). Fix row and column sets
\(\underline\zeta,\underline\xi\) of cardinality \(n\). Suppose that
\begin{equation}\label{eq:minor-estimate}
 v_p\bigl(A_{k,i}^{(\underline\zeta\times\underline\xi)}\bigr)
 \geq
 \Delta_{k,\frac12d_k^{\new}-i}
 -\Delta'_{k,\frac12d_k^{\new}-m_n(k)}
 +\frac12(\deg\underline\zeta-\deg\underline\xi)
\end{equation}
for every ghost zero \(w_k\) of \(g_n\) and every integer
\(0\leq i<m_n(k)\). Then:
\begin{enumerate}[label=(\arabic*)]
\item the quotient in \eqref{eq:lagrange-minor} satisfies
\[
 h_{\underline\zeta\times\underline\xi}(w)
 \in p^{\frac12(\deg\underline\zeta-\deg\underline\xi)}
       \cO\langle w/p\rangle;
\]
\item at a ghost zero \(w_{k_0}\),
\begin{equation}\label{eq:overcoeff-ghost}
\begin{aligned}
 v_p(a_{k_0,i})\geq{}&\frac12(\deg\underline\zeta-\deg\underline\xi)\\
 &+\begin{cases}
  \Delta_{k_0,\frac12d_{k_0}^{\new}-m_n(k_0)}-\Delta'_{k_0,\frac12d_{k_0}^{\new}-m_n(k_0)},&i=m_n(k_0),\\[2mm]
  \frac12\bigl((\frac12d_{k_0}^{\new}-i)^2-(\frac12d_{k_0}^{\new}-m_n(k_0))^2\bigr),&m_n(k_0)<i<\frac12d_{k_0}^{\new}
 \end{cases}
\end{aligned}
\end{equation}
\item if \(d_{k_0}^{\ur}\geq n\), expand instead
\[
 \det U^\dagger(\underline\zeta\times\underline\xi)/g_n(w)
 =\sum_{i\geq0}a_{k_0,i}(w-w_{k_0})^i
\]
Then
\begin{equation}\label{eq:overcoeff-nonghost}
\begin{aligned}
 v_p(a_{k_0,i})\geq{}&\frac12(\deg\underline\zeta-\deg\underline\xi)\\
 &+\begin{cases}
  \NP(G(w_{k_0},-))_n-v_p(g_n(w_{k_0})),&i=0,\\[2mm]
  \frac12\bigl((\frac12d_{k_0}^{\new}-i)^2-(\tfrac12d_{k_0}^{\new})^2\bigr),&0<i<\frac12d_{k_0}^{\new}
 \end{cases}
\end{aligned}
\end{equation}
\end{enumerate}
Here and below the row and column sets are suppressed on Taylor
coefficients when they are fixed.
\end{proposition}

\begin{proof}
We use the interpolation argument of
\cite[Proposition~5.5]{LTXZ2}. The changes are the numerical
comparison between distinct weights and the expansion at a
non-ghost centre.

First, \eqref{eq:power-minor-bound} implies (1). Indeed, after multiplying
\eqref{eq:lagrange-minor} by
\(p^{-\deg g_n+\frac12(\deg\underline\xi-\deg\underline\zeta)}\),
its left side belongs to \(\cO\langle w/p\rangle\). Since \(p^{-\deg g_n}g_n(w)\in\cO\langle w/p\rangle\) is monic, division by this
polynomial gives a quotient in \(\cO\langle w/p\rangle\) and a polynomial remainder of degree less than \(\deg g_n\). Note that the right side of the scaled \eqref{eq:lagrange-minor} gives the exact division: the quotient is \(p^{\frac12(\deg\underline\xi-\deg\underline\zeta)}
h_{\underline\zeta\times\underline\xi}\) and the remainder is 
\[
 p^{-\deg g_n+\frac12(\deg\underline\xi-\deg\underline\zeta)}
 \sum_{m_n(k)>0}
 A_k^{(\underline\zeta\times\underline\xi)}(w)g_{n,\widehat k}(w)
\]
since each summand has degree less than \(\deg g_n\).

We prove (2) and (3) together. In (3), \(m_n(k_0)=0\) and
\(g_{n,\widehat{k_0}}=g_n\), \(A_{k_0}=0\) in this case.
Dividing \eqref{eq:lagrange-minor} by \(g_{n,\widehat{k_0}}(w)\) gives
\[
\frac{\det U^\dagger(\underline\zeta\times\underline\xi)}{g_{n,\widehat{k_0}}(w)}
 =A_{k_0}(w)+\sum_{\substack{m_n(k)>0\\k\ne k_0}}
 A_k(w)\frac{(w-w_{k_0})^{m_n(k_0)}}{(w-w_k)^{m_n(k)}}+h(w)(w-w_{k_0})^{m_n(k_0)}
\]
For each \(k\ne k_0\) with \(m_n(k)>0\) and each
\(0\leq j<m_n(k)\), define \(a_{k_0,k,i}^{(j)}\) by the formal expansion
\[
 A_{k,j}(w-w_k)^j\frac{(w-w_{k_0})^{m_n(k_0)}}{(w-w_k)^{m_n(k)}}
 =\sum_{i\geq m_n(k_0)}a_{k_0,k,i}^{(j)}(w-w_{k_0})^i.
\]
The binomial expansion at \(w_{k_0}\) gives
\begin{equation}\label{eq:two-centre-taylor}
 a_{k_0,k,i}^{(j)}
 =\binom{j-m_n(k)}{i-m_n(k_0)}A_{k,j}
  (w_{k_0}-w_k)^{j-m_n(k)-i+m_n(k_0)}
\end{equation}
for \(i\geq m_n(k_0)\). By comparing the coefficients of \((w-w_{k_0})^i\), we see that \(a_{k_0,i}\) is the sum of \(a_{k_0,k,i}^{(j)}\) over \(k\ne k_0\) and \(0\leq j<m_n(k)\), together with the coefficient
of \((w-w_{k_0})^{i-m_n(k_0)}\) in \(h(w)\) for \(i\geq m_n(k_0)\).
When \(i=m_n(k_0)\), apply \cref{prop:lagrange-term-bound} to
\(A=p^{\frac12(\deg\underline\xi-\deg\underline\zeta)}A_{k,j}\), together with the assumed bound \eqref{eq:minor-estimate}, show that each \(a_{k_0,k,i}^{(j)}\) and hence \(a_{k_0,i}\) satisfies the equality in the first line of \eqref{eq:overcoeff-ghost}, or of
\eqref{eq:overcoeff-nonghost}.

For higher orders, put within this calculation
\begin{gather*}
 b=m_n(k)-j\geq1,\qquad b_0=i-m_n(k_0)\geq1,\\
 y=\frac12d_k^{\new}-m_n(k),\qquad y_0=\frac12d_{k_0}^{\new}-m_n(k_0),\qquad
 \gamma=v_p(k_\bullet-k_{0\bullet}).
\end{gather*}
Thus \(y\geq0\) and \(y_0-b_0=\frac12d_{k_0}^{\new}-i\geq1/2\).
The assumed bound \eqref{eq:minor-estimate} gives
\[
 v_p(A_{k,j})-\frac12(\deg\underline\zeta-\deg\underline\xi)
 \geq\Delta_{k,y+b}-\Delta'_{k,y}
\]
The proof of \cref{cor:multi-gap} establishes the consecutive-difference
estimate \(\Delta_{k,r}-\Delta'_{k,r-1}\geq4r+1\).
Since \(\Delta'_{k,r}\geq\Delta_{k,r}\), summing over
\(r=y+1,\ldots,y+b\), with step one also for half-integral indices, gives
\[
\begin{aligned}
 \Delta_{k,y+b}-\Delta'_{k,y}
 &\geq\sum_{t=1}^{b}
       \bigl(\Delta_{k,y+t}-\Delta'_{k,y+t-1}\bigr)\\
 &\geq\sum_{t=1}^{b}\bigl(4(y+t)+1\bigr)
 =b(4y+2b+3)
\end{aligned}
\]
Since the binomial coefficient in \eqref{eq:two-centre-taylor}
is integral and \(v_p(w_k-w_{k_0})=1+\gamma\), we have
\[
\begin{aligned}
 v_p\bigl(a_{k_0,k,i}^{(j)}\bigr)
 -\frac12(\deg\underline\zeta-\deg\underline\xi)
 &\geq b(4y+2b+3)-(b+b_0)(1+\gamma)\\
 &=b(4y+2b+2-\gamma)-b_0(1+\gamma).
\end{aligned}
\]
Thus the required bound
\[
 v_p\bigl(a_{k_0,k,i}^{(j)}\bigr)
 \geq\frac12(\deg\underline\zeta-\deg\underline\xi)-b_0(y_0-b_0/2)
\]
follows once
\begin{equation}\label{eq:two-weight-numerical}
 b(4y+2b+2-\gamma)+b_0(y_0-b_0/2-1-\gamma)\geq0.
\end{equation}

The estimate on \(y+y_0\) differs from the generic case: If \(k_0\) is a ghost zero of \(g_n\), then
\[
 y+y_0
 =|n-d_k^{\Iw}/2|+|n-d_{k_0}^{\Iw}/2|
 \geq\tfrac12|k_\bullet-k_{0\bullet}|
 \geq\tfrac12p^\gamma
\]
If \(d_{k_0}^{\ur}\geq n\), then \(k_{0\bullet}>k_\bullet\) and we have
\[
 \frac1 2 d_{k_0}^{\new}-\frac12d_k^{\new}
 =\frac {k_{0\bullet}-k_{\bullet}}{2}-(d_{k_0}^{\ur}-d_k^{\ur})
 \geq\frac{p-1}{2(p+1)}(k_{0\bullet}-k_{\bullet})-1
\]
As \(\frac12d_k^{\new}\geq m_n(k)\geq1\), this implies
\[
 y+y_0=(\frac12d_{k_0}^{\new}-\frac12d_k^{\new})+\frac12d_k^{\new}+y
 \geq\frac{p-1}{2(p+1)}(k_{0\bullet}-k_{\bullet})\geq\frac13p^\gamma
\]
We may therefore use \(y+y_0\geq p^\gamma/3\) in both cases.

For \(\gamma=0\), both terms in \eqref{eq:two-weight-numerical}
are nonnegative. Suppose \(\gamma\geq1\),
if its first term is negative, then \(\gamma\geq5\) and
\(y<\gamma/4\). Hence the first term is bounded below by \(-\gamma^2/8\), while the second is at least
\(y_0/2-1-\gamma>0\). The sum is consequently at least
\[
 \frac{11^\gamma}{6}-\frac{9\gamma}{8}-1-\frac{\gamma^2}{8}>0
\]
If instead the second term is negative, then
\(y_0<2\gamma+3/2\), whence
\(y>11^\gamma/3-2\gamma-3/2\).
The first term is at least \(4\cdot11^\gamma/3-9\gamma-2>0\),
and the second is at least \(-(\gamma+1/2)^2/2\).
Their sum is strictly greater than
\[
 \frac43\,11^\gamma-9\gamma-2-\frac12(\gamma+1/2)^2>0
\]
This proves \eqref{eq:two-weight-numerical}.

Finally, the coefficient of order \(i\) in
\(h(w)(w-w_{k_0})^{m_n(k_0)}\) has valuation at least
\(\frac12(\deg\underline\zeta-\deg\underline\xi)+m_n(k_0)-i
=\frac12(\deg\underline\zeta-\deg\underline\xi)-b_0\).
For \(i>m_n(k_0)\), the inequality
\(y_0\geq b_0+1/2\) gives \(y_0-b_0/2\geq1\), so
\(-b_0\geq-b_0(y_0-b_0/2)\).
Thus the remainder satisfies the same bound. This completes the proof of both (2) and (3).
\end{proof}

We now isolate the finite determinant identity.  Fix a ghost zero \(w_k\)
of \(g_n\), and abbreviate
\[
 r=r_{\underline\zeta\times\underline\xi}(k),\qquad
 s=s_{\underline\xi}(k)
\]
Let \(L_k\) be the constant infinite matrix whose upper-left
\(d_k^{\Iw}\times d_k^{\Iw}\) block is
\(-\operatorname{AL}_{(k,\widetilde\varepsilon_1)}\), and whose other
entries agree with \(U^\dagger(w_k)\).  Thus
\(U^\dagger(w_k)-L_k\) has rank at most \(d_k^{\ur}\).  Let
\[
 J=\{c\in\underline\xi:c>d_k^{\Iw}\ \text{or}\
 d_k^{\Iw}+1-c\in\underline\zeta\}
\]
then \(\#J=r+s\).  For \(0\leq j\leq r+s\), set
\begin{equation}\label{eq:cofactor-Dj}
D_j=\sum_{\substack{I\subseteq\underline\zeta,\ J'
                  \subseteq J\\ \#I=\#J'=j}}
 \operatorname{sgn}(I,\underline\zeta)
 \operatorname{sgn}(J',\underline\xi)\det L_k(I\times J')
 \det U^\dagger((\underline\zeta-I)\times(\underline\xi-J'))
\end{equation}

In particular \(D_0=\det U^\dagger(\underline\zeta\times\underline\xi)\).

\begin{lemma}\label{lem:cofactor-congruences}
For \(0\leq\ell\leq j_0\leq r+s-1\),
\begin{equation}\label{eq:cofactor-congruence}
 D_\ell\equiv
 \sum_{j=j_0+1}^{r+s}(-1)^{j-j_0-1}
 \binom{j-\ell-1}{j_0-\ell}\binom{j}{\ell}D_j
 \pmod{(w-w_k)^{\max\{0,n-d_k^{\ur}-j_0\}}}
\end{equation}
If \(\eta(w)\in1+(w-w_k)E[\![w-w_k]\!]\), the same congruence holds
after replacing every \(D_j\) by \(D_j\eta(w)^{-j}\).
\end{lemma}

\begin{proof}
This is \cite[Lemma~6.4 and Remark~6.5]{LTXZ2}, with
\(D_j\) denoting the cofactor sum in \eqref{eq:cofactor-Dj}.
The proof uses only the cofactor identity of
\cite[Lemma~6.2]{LTXZ2} and
\[
 \operatorname{rank}\bigl(U^\dagger(w_k)-L_k\bigr)\leq d_k^{\ur},
\]
which follows from \cref{prop:classical-matrix}.
It is independent of the parity of \(d_k^{\Iw}\).
Replacing \(L_k\) by \(\eta(w)^{-1}L_k\) preserves its specialization
at \(w_k\) and replaces \(D_j\) by \(D_j\eta(w)^{-j}\), as required.
\end{proof}

As in \cite[Notation~6.7]{LTXZ2}, define
\[
 p^{\frac12(\deg\underline\xi-\deg\underline\zeta)}
 \frac{D_\ell}{g_{n-\ell,\widehat k}(w)/g_{n-\ell,\widehat k}(w_k)}
 =\sum_{i\geq0}B_{k,i}^{(\ell)}(w-w_k)^i
\]
with \(g_{0,\widehat k}=1\).
For \(\ell=0\) and \(i<m_n(k)\),
\begin{equation}\label{eq:B-versus-A}
 B_{k,i}^{(0)}=p^{\frac12(\deg\underline\xi-\deg\underline\zeta)}
                  A_{k,i}^{(\underline\zeta\times\underline\xi)}
                  g_{n,\widehat k}(w_k)
\end{equation}
The target estimate \eqref{eq:minor-estimate} is therefore equivalent to
\begin{equation}\label{eq:B-target}
 v_p(B_{k,i}^{(0)})\geq
 \Delta_{k,\frac12d_k^{\new}-i}-\frac{k-2}{2}
             \left(\frac{d_k^{\Iw}}2-n\right)
\end{equation}

The following proposition is the analogue of
\cite[Proposition~6.8]{LTXZ2}.  It applies the preceding estimates to the
complementary minors in \eqref{eq:cofactor-Dj}.  The case
\(m_{n-\ell}(k)=0\) is treated separately in the proof because it uses the expansion at a
weight which is not a zero of \(g_{n-\ell}\).

\begin{proposition}\label{prop:smaller-minors}
Assume \eqref{eq:minor-estimate} for minors of size smaller than \(n\). Let
\(1\leq\ell\leq r+s\) and \(m_{n-\ell}(k)<m_n(k)\). Then:
\begin{enumerate}[label=(\arabic*)]
\item for \(0\leq i<m_{n-\ell}(k)\), one has
\begin{equation}\label{eq:smaller-below-multiplicity}
 v_p(B_{k,i}^{(\ell)})\geq
 \Delta_{k,\frac1 2 d_k^{\new}-i}
 -\frac{k-2}{2}\left(\frac{d_k^{\Iw}}2-n\right)
\end{equation}
\item for \(m_{n-\ell}(k)\leq i<m_n(k)\), one has
\begin{align}\label{eq:smaller-positive}
 v_p(B_{k,i}^{(\ell)})\geq{}&
 \Delta_{k,\frac12d_k^{\new}-m_{n-\ell}(k)}-\frac{k-2}{2}
       \left(\frac{d_k^{\Iw}}2-n\right)\nonumber\\
 &-\frac12\bigl((\frac12d_k^{\new}-m_{n-\ell}(k))^2-(\frac12d_k^{\new}-i)^2\bigr)
 \nonumber\\
 \geq{}&\Delta_{k,\frac12d_k^{\new}-i}
 -\frac{k-2}{2}\left(\frac{d_k^{\Iw}}2-n\right)
\end{align}
including when \(m_{n-\ell}(k)=0\).
\end{enumerate}
\end{proposition}

\begin{proof}
We follow \cite[Proposition~6.8]{LTXZ2}, treating separately
the case \(m_{n-\ell}(k)=0\).
For every nonzero summand of \(D_\ell\),
\begin{equation}\label{eq:L-normalization}
 v_p\!\left(
 p^{\frac12(\deg J'-\deg I)}\det L_k(I\times J')\right)
 \geq\frac{k-2}{2}\ell
\end{equation}
The verification in \cite[equations~(6.10.1)--(6.10.2)]{LTXZ2}
uses only the identity that reflected degrees add to \(k-2\)
and the weak Hodge bound for nonclassical columns. These are
\cref{prop:classical-matrix,lem:weak-hodge-block}; a fixed
reflection index satisfies the same identity.

For \(i<m_{n-\ell}(k)\), the induction hypothesis applying to indices \((\underline{\zeta}-I)\times(\underline{\xi}-J')\), together with the formula
\[
B_{k,i}^{(\ell)}=p^{\frac 1 2 (\deg \underline{\xi}-\deg \underline{\zeta})}g_{n-\ell,\widehat{k}}(w_k)\sum_{I,J'}(\pm1)\det L_k(I\times J')A_{k,i}^{((\underline{\zeta}-I)\times(\underline{\xi}-J'))}
\]
give
\[
v_p(B_{k,i}^{(\ell)})\geq v_p(g_{n-\ell,\widehat{k}}(w_k))+\frac{k-2}{2}\ell+\Delta_{k,\frac1 2 d_k^{\new}-i}-\Delta'_{k,\frac1 2 d_k^{\new}-m_{n-\ell}(k)}
\]
Since \(m_{n-\ell}(k)<m_n(k)\), one has \(n-\ell<d_k^{\Iw}/2\), and
\[
 v_p(g_{n-\ell,\widehat k}(w_k))
 =\Delta'_{k,\frac12d_k^{\new}-m_{n-\ell}(k)}
 -\frac{k-2}{2}\left(\frac{d_k^{\Iw}}2-(n-\ell)\right)
\]
this proves (1).
\eqref{eq:smaller-below-multiplicity}. For \(m_{n-\ell}(k)>0\) and \(i\geq m_{n-\ell}(k)\), use
\cref{prop:taylor-coefficients}(2) instead. The same calculation
gives the first inequality in \eqref{eq:smaller-positive}.

When \(m_{n-\ell}(k)=0\), one has \(n-\ell\leq d_k^{\ur}\), since
\(n<d_k^{\Iw}-d_k^{\ur}\). By \cref{prop:old-bound,prop:compatibility}, we have
\[
\begin{split}
 \NP(G(w_k,-))_{n-\ell}
 &\geq v_p(g_{d_k^{\ur}}(w_k))-(d_k^{\ur}-n+\ell)\frac{k-2}{2}\\
 &=\Delta_{k,\frac12d_k^{\new}}
 -\frac{k-2}{2}\left(\frac{d_k^{\Iw}}2-(n-\ell)\right)
\end{split}
\]
Recall that
\[
 \frac{\det U^\dagger((\underline\zeta-I)\times(\underline\xi-J'))}
 {g_{n-\ell}(w)}
 =\sum_{i\geq0}a_{k,i}^{((\underline\zeta-I)\times(\underline\xi-J'))}
 (w-w_k)^i.
\]
By \cref{prop:taylor-coefficients}(3), for \(0\leq i<m_n(k)\),
\[
\begin{aligned}
&v_p\!\left(a_{k,i}^{((\underline\zeta-I)\times(\underline\xi-J'))}\right)
-\frac12\bigl(\deg(\underline\zeta-I)-\deg(\underline\xi-J')\bigr)\\
&\quad\geq
\begin{cases}
 \NP(G(w_k,-))_{n-\ell}-v_p(g_{n-\ell}(w_k)),&i=0,\\[2mm]
 -\frac12\bigl((\frac12d_k^{\new})^2
              -(\frac12d_k^{\new}-i)^2\bigr),&i>0.
\end{cases}
\end{aligned}
\]
The formula for \(B_{k,i}^{(\ell)}\) above now holds with these
\(a_{k,i}\) in place of \(A_{k,i}\):
\[
\begin{aligned}
v_p(B_{k,i}^{(\ell)})\geq\frac{k-2}{2}\ell
+\begin{cases}
 \NP(G(w_k,-))_{n-\ell},&i=0,\\[2mm]
 v_p(g_{n-\ell}(w_k))
 -\frac12\bigl((\frac12d_k^{\new})^2
              -(\frac12d_k^{\new}-i)^2\bigr),&i>0.
\end{cases}
\end{aligned}
\]
Using \(v_p(g_{n-\ell}(w_k))\geq\NP(G(w_k,-))_{n-\ell}\)
and the preceding lower bound for that polygon ordinate, we obtain
\[
v_p(B_{k,i}^{(\ell)})\geq \Delta_{k,\frac12d_k^{\new}}-\frac{k-2}{2}\left(\frac{d_k^{\Iw}}2-n\right)-\frac12\bigl((\tfrac12d_k^{\new})^2-(\tfrac12d_k^{\new}-i)^2\bigr)
\]
This is the first inequality in \eqref{eq:smaller-positive} when
\(m_{n-\ell}(k)=0\).
The second inequality in \eqref{eq:smaller-positive} follows from
\eqref{eq:quadratic-gap}.
\end{proof}

The coefficients associated with cofactors of different sizes are divided
by different ghost polynomials.  Following \cite[Notation~6.18]{LTXZ2}, we
compare these normalizations by setting
\[
 \eta_j(w)=\frac{g_{n-j,\widehat k}(w)/g_{n-j,\widehat k}(w_k)}{g_{n,\widehat k}(w)/g_{n,\widehat k}(w_k)},\qquad
 \frac{\eta_j(w)}{\eta_1(w)^j}
 =1+\sum_{t\geq1}\eta_{(j),t}(w-w_k)^t
\]

The next lemma bounds the distance between \(w_k\) and a weight \(w_{k'}\)
when either the midpoint \(d_{k'}^{\Iw}/2\), an endpoint
\(d_{k'}^{\ur}\), or the reflected endpoint
\(d_{k'}^{\Iw}-d_{k'}^{\ur}\) occurs in an interval relevant to a
cofactor of size \(n-j\).  Since \(d_k^{\Iw}\) changes by one as
\(k_\bullet\) changes by one, the bound contains a factor \(2\) which is
absent from the strongly generic calculation.

\begin{lemma}
\label{lem:weight-distance}
Assume \(d_k^{\new}>0\), and set
\[
 \gamma_k=\left\lfloor\log_p((p+1)d_k^{\new})+1\right\rfloor
\]
For a distinct weight \(k'\), one has
\(v_p(w_k-w_{k'})\leq\gamma_k\) if at least one of the following holds:
\[
 {\frac12d_{k'}^{\Iw}}\in[{d_k^{\ur}},d_k^{\Iw}-{d_k^{\ur}}],\qquad k'_\bullet<k_\bullet,\qquad
 {d_{k'}^{\ur}}\in[{d_k^{\ur}},{\frac12d_k^{\Iw}})
\]
\end{lemma}

\begin{proof}
Write
\[
 k_\bullet=t_\varepsilon+(p+1)h+r,\qquad 0\leq r\leq p
\]
Then \({d_k^{\ur}}=h+1\) and \(d_k^{\new}=s_\varepsilon+r+(p-1)h\).

These identities remain valid for the possible initial values with \(h=-1\);
the condition \(d_k^{\new}>0\) then forces
\(s_\varepsilon+r\geq p\).  Consequently
\begin{equation}\label{eq:tres-crude-distance}
 k_\bullet<(p+1)d_k^{\new}
\end{equation}
Indeed, after subtraction the right side minus the left side is
\[
 p(s_\varepsilon+r)+(p+1)(p-2)h-\delta_\varepsilon-1
\]
which is positive for \(h\geq0\); for \(h=-1\) it is at least
\(p+1-\delta_\varepsilon\).

Since \(d_k^{\Iw}=k_\bullet+1-\delta_\varepsilon\),
\[
 |k_\bullet-k'_\bullet|=|d_k^{\Iw}-d_{k'}^{\Iw}|
\]
Thus the first condition gives
\(|k_\bullet-k'_\bullet|\leq d_k^{\new}\), and the second gives
\(|k_\bullet-k'_\bullet|<(p+1)d_k^{\new}\) by
\eqref{eq:tres-crude-distance}.  For the third condition, if \({d_{k'}^{\ur}}={d_k^{\ur}}\)
then their bullet indices lie in the same interval of length \(p+1\)
on which \(d^{\ur}\) is constant, so
\(|k_\bullet-k'_\bullet|\leq p\).  If \({d_{k'}^{\ur}}>{d_k^{\ur}}\), the same quotient and remainder
calculation gives
\[
 k'_\bullet-k_\bullet\leq(p+1)({d_{k'}^{\ur}}-{d_k^{\ur}})+p
\]
When \(d_k^{\new}\) is even, \({d_{k'}^{\ur}}-{d_k^{\ur}}\leq d_k^{\new}/2-1\); when it is
odd, \({d_{k'}^{\ur}}-{d_k^{\ur}}\leq d_k^{\new}/2-\tfrac12\).  In both cases the last
display is \(<(p+1)d_k^{\new}\).  Finally
the identity \(v_p(w_k-w_{k'})=1+v_p(k_\bullet-k'_\bullet)\) and
\[
 v_p(k_\bullet-k'_\bullet)
 \leq\left\lfloor\log_p|k_\bullet-k'_\bullet|\right\rfloor
\]
give the assertion.
\end{proof}

For fixed \(j\), the function \(i\mapsto m_i(k')\) is affine on
\([n-j,n]\) unless one of
\(d_{k'}^{\ur}\), \(d_{k'}^{\Iw}-d_{k'}^{\ur}\), or
\(d_{k'}^{\Iw}/2\) lies in \((n-j,n)\).  Only such weights can contribute a nonzero exponent in the product
for \(\eta_j/\eta_1^j\), and hence a denominator in its expansion at
\(w_k\).  The following lemma bounds
the sum of their distances from \(w_k\), including repetitions coming from
the exponent with which a factor occurs.  Its second assertion treats
\(m_{n-j}(k)=0\), for which the estimate in the first assertion cannot be
used.

\begin{lemma}\label{lem:weight-multiset-estimate}
Assume \(p\geq11\).  Let
\[
 e_j(k')=m_{n-j}(k')-m_n(k')
       -j\bigl(m_{n-1}(k')-m_n(k')\bigr)
\]
and suppose that \(k'\ne k\) satisfies the condition that one of
\(d_{k'}^{\ur}\), \(d_{k'}^{\Iw}-d_{k'}^{\ur}\), or
\(d_{k'}^{\Iw}/2\) lies in \((n-j,n)\).  Let
\(\mathcal S=(k'_1,\ldots,k'_t)\), with \(1\leq t<m_n(k)\), be a
multiset of weights satisfying this condition such
that the multiplicity of \(k'\) is at most \(e_j(k')\) whenever
\(e_j(k')>0\).

If \(1\leq m_{n-j}(k)<m_n(k)\) and
\(q_t=\min\{m_n(k)-t,m_{n-j}(k)\}\), then
\begin{align}\label{eq:weight-multiset-positive}
 \sum_{\alpha=1}^t v_p(w_k-w_{k'_\alpha})\leq{}&
 \Delta_{k,\frac12d_k^{\new}-q_t}-\Delta_{k,\frac12d_k^{\new}-q_t-t}\nonumber\\
 &-\frac12\bigl((\frac12d_k^{\new}-q_t)^2-(\frac12d_k^{\new}-q_t-t)^2\bigr)
\end{align}
If \(m_{n-j}(k)=0\), then
\begin{align}\label{eq:weight-multiset-zero}
 \sum_{\alpha=1}^t v_p(w_k-w_{k'_\alpha})\leq{}&
 \Delta_{k,\frac12d_k^{\new}}-\Delta_{k,\frac12d_k^{\new}-t}
 -\frac12\bigl((\tfrac12d_k^{\new})^2-(\frac12d_k^{\new}-t)^2\bigr)
\end{align}
\end{lemma}

\begin{proof}
We follow the proof of \cite[Lemma~6.29(1), Section~6.31]{LTXZ2},
using \cref{cor:multi-gap} in place of its Proposition~2.22.
Put \(n^*=\min\{n,d_k^{\Iw}-n\}\).
We first make the endpoint reduction in that proof, with a closed
interval to include equality cases.

If an occurrence in \(\mathcal S\) has an endpoint
\(d_{k'}^{\ur}\) or \(d_{k'}^{\Iw}-d_{k'}^{\ur}\) in
\([n^*,d_k^{\Iw}-n^*]\), remove that occurrence. In the positive-multiplicity
case, set \(s=\frac12d_k^{\new}-q_t-t+1\) if \(q_t=q_{t-1}\), and
\(s=\frac12d_k^{\new}-q_t\) if \(q_{t-1}=q_t+1\).
In the zero-multiplicity case set \(s=\frac12d_k^{\new}-t+1\).
The right side of the desired inequality decreases by
\[
 \Delta_{k,s}-\Delta_{k,s-1}
 -\frac12\bigl(s^2-(s-1)^2\bigr)
\]
In each case \({\frac12d_k^{\Iw}}-n^*=\frac12d_k^{\new}-m_n(k)\leq s-1\).
Applying \eqref{eq:multi-gap} with
\((\ell,\ell',\ell'')=(s-1,s-1,s)\) bounds
\(v_p(w_k-w_{k'})\) by this difference, since
\(\Delta'_{k,s-1}\geq\Delta_{k,s-1}\).
Induction on the number of removed occurrences therefore reduces
the problem to multisets with no such endpoint.

Use \(q_t=\min\{m_n(k)-t,m_{n-j}(k)\}\) also in the
zero-multiplicity case, where \(q_t=0\).  Write
\[
 \gamma=\left\lfloor\log_p(2(p+1){(\frac12d_k^{\new}-q_t)})+2\right\rfloor
\]
The factor \(2\) accounts for
\(|k_\bullet-k'_\bullet|=2|{\frac12d_k^{\Iw}}-d_{k'}^{\Iw}/2|\).
The consecutive-difference estimates in the proof of
\cref{cor:multi-gap}, summed from \({(\frac12d_k^{\new}-q_t)}-t\) to \({(\frac12d_k^{\new}-q_t)}\), give
\[
 t\gamma\leq
 \Delta_{k,\frac12d_k^{\new}-q_t}-\Delta_{k,\frac12d_k^{\new}-q_t-t}
 -\frac12\bigl({(\frac12d_k^{\new}-q_t)}^2-({(\frac12d_k^{\new}-q_t)}-t)^2\bigr)
\]
Equivalently, this is the specialization
\[
 (\ell,\ell',\ell'')=(\tfrac12d_k^{\new}-q_t-t,
 \tfrac12d_k^{\new}-q_t,\tfrac12d_k^{\new}-q_t)
\]
of that proof, where the
distance term has coefficient zero and no endpoint hypothesis
is needed. Hence it suffices to treat any occurrences with
valuation greater than \(\gamma\).

Suppose first that \(0<m_{n-j}(k)<m_n(k)\). Then
\[
 n-j<n^*\leq {\frac12d_k^{\Iw}},\qquad
 m_{n-j}(k)=n-j-d_k^{\ur},\qquad
 {(\frac12d_k^{\new}-q_t)}\geq {\frac12d_k^{\Iw}}-n+j\geq n-{\frac12d_k^{\Iw}}
\]
After the endpoint reduction, every remaining endpoint is in
\((n-j,n^*)\), and every midpoint in \((n-j,n)\).
Both intervals lie in \(({\frac12d_k^{\Iw}}-{(\frac12d_k^{\new}-q_t)},{\frac12d_k^{\Iw}}+{(\frac12d_k^{\new}-q_t)})\).
For a midpoint, the displayed centre-distance identity gives
\(|k_\bullet-k'_\bullet|<2{(\frac12d_k^{\new}-q_t)}\), and thus
\(v_p(w_k-w_{k'})\leq\gamma\).
For endpoints, the last assertion of \cref{cor:multi-gap}
shows that at most one distinct remaining weight has valuation
greater than \(\gamma\): two such endpoints would sum to \(d_k^{\Iw}\),
whereas both are strictly less than \(n^*\leq {\frac12d_k^{\Iw}}\).

Let \(k'\) be this weight, occurring \(t'\) times, and let \(a\)
denote its endpoint in \((n-j,n^*)\), only for this calculation.
Its midpoint is outside \((n-j,n)\), so no second endpoint lies
there. At either the lower or upper endpoint, the multiplicity
formula gives
\[
\begin{aligned}
 e_j(k')&=a-(n-j)>0,\\
 \frac12d_k^{\Iw}-a&=\frac12d_k^{\Iw}-n+j-e_j(k')\\
 &\leq\frac12d_k^{\new}-q_t-e_j(k')
 \leq\frac12d_k^{\new}-q_t-t'.
\end{aligned}
\]
Thus the endpoint lies in the closed interval of radius \({(\frac12d_k^{\new}-q_t)}-t'\).
Applying \eqref{eq:multi-gap} with
\[
 (\ell,\ell',\ell'')=(\tfrac12d_k^{\new}-q_t-t,
 \tfrac12d_k^{\new}-q_t-t',\tfrac12d_k^{\new}-q_t)
\]
yields
\[
\begin{aligned}
 t'v_p(w_k-w_{k'})+(t-t')\gamma
 &\leq\Delta_{k,\frac12d_k^{\new}-q_t}-\Delta'_{k,\frac12d_k^{\new}-q_t-t}\\
 &\quad-\frac12\bigl((\tfrac12d_k^{\new}-q_t)^2
 -(\tfrac12d_k^{\new}-q_t-t)^2\bigr).
\end{aligned}
\]
Since \(\Delta'_{k,\frac12d_k^{\new}-q_t-t}\geq\Delta_{k,\frac12d_k^{\new}-q_t-t}\), this proves
\eqref{eq:weight-multiset-positive}.

It remains to check \(m_{n-j}(k)=0\), which allows
\(n-j\leq d_k^{\ur}\).
A remaining midpoint \(d_{k'}^{\Iw}/2\) in \((n-j,n)\)
either lies in \([{d_k^{\ur}},d_k^{\Iw}-{d_k^{\ur}}]\), or satisfies \({\frac12d_{k'}^{\Iw}}<{d_k^{\ur}}<{\frac12d_k^{\Iw}}\), in which
case \(k'_\bullet<k_\bullet\).
A remaining lower endpoint \(d_{k'}^{\ur}\) in
\((n-j,n^*)\) either has \({d_{k'}^{\ur}}<{d_k^{\ur}}\), again implying
\(k'_\bullet<k_\bullet\), or belongs to \([{d_k^{\ur}},{\frac12d_k^{\Iw}})\).
For a remaining upper endpoint \(d_{k'}^{\Iw}-{d_{k'}^{\ur}}\in(n-j,n^*)\), one has
\({\frac12d_{k'}^{\Iw}}\leq d_{k'}^{\Iw}-{d_{k'}^{\ur}}<{\frac12d_k^{\Iw}}\), so \(k'_\bullet<k_\bullet\).
Every remaining weight therefore satisfies one of the three
conditions of \cref{lem:weight-distance}.
In this case \(q_t=0\), and that lemma bounds all its occurrences by
\(\gamma\); the preceding bound for \(t\gamma\) proves
\eqref{eq:weight-multiset-zero}.

All radii used above are admissible and
\({(\frac12d_k^{\new}-q_t)}-t\geq \frac12d_k^{\new}-m_n(k)\geq0\).
If \({(\frac12d_k^{\new}-q_t)}<2\), the only nonempty cases are \(t=1\) with
\({(\frac12d_k^{\new}-q_t)}=1\) or \(3/2\), covered by the corresponding initial
estimates in the proof of \cref{cor:multi-gap}.
\end{proof}

We now apply the preceding lemma to the coefficients of
\(\eta_j/\eta_1^j\).  These are the estimates needed in the comparison with
the coefficients of the normalized cofactors in
\cite[Proposition~6.21]{LTXZ2}.

\begin{proposition}
\label{prop:eta-estimate}
Assume \(p\geq11\).  Let \(1\leq j\leq r+s\) and
\(1\leq t<m_n(k)\).
If \(1\leq m_{n-j}(k)<m_n(k)\), put
\(q_t=\min\{m_n(k)-t,m_{n-j}(k)\}\).  Then
\begin{align}\label{eq:eta-positive}
 v_p(\eta_{(j),t})\geq{}&
 \Delta_{k,\frac12d_k^{\new}-q_t-t}-\Delta_{k,\frac12d_k^{\new}-q_t}\nonumber\\
 &+\frac12\bigl((\frac12d_k^{\new}-q_t)^2-(\frac12d_k^{\new}-q_t-t)^2\bigr)
\end{align}
If \(m_{n-j}(k)=0\), then
\begin{align}\label{eq:eta-zero}
 v_p(\eta_{(j),t})\geq{}&
 \Delta_{k,\frac12d_k^{\new}-t}-\Delta_{k,\frac12d_k^{\new}}
 +\frac12\bigl((\tfrac12d_k^{\new})^2-(\frac12d_k^{\new}-t)^2\bigr)
\end{align}
\end{proposition}

\begin{proof}
As in \cite[equation~(6.28.1)]{LTXZ2}, the definitions give
\begin{equation}\label{eq:eta-factorization}
 \frac{\eta_j(w)}{\eta_1(w)^j}
 =\prod_{k'\ne k}
 \left(1+\frac{w-w_k}{w_k-w_{k'}}\right)^{e_j(k')}
\end{equation}
If \(e_j(k')\ne0\), the multiplicity function is not affine on
\([n-j,n]\), so one of its three break points lies in \((n-j,n)\).
The coefficient of \((w-w_k)^t\) is a sum of terms
\[
 \left(\prod_{k'\ne k}\binom{e_j(k')}{r_{k'}}\right)
 \prod_{k'\ne k}(w_k-w_{k'})^{-r_{k'}},
 \qquad r_{k'}\geq0,\quad \sum_{k'\ne k}r_{k'}=t
\]
The binomial coefficients are integral. A nonzero term has
\(r_{k'}\leq e_j(k')\) when \(e_j(k')>0\); negative exponents impose
no such restriction. Thus the denominator multiset satisfies the
hypotheses of \cref{lem:weight-multiset-estimate}. Applying its
respective assertions proves \eqref{eq:eta-positive} and
\eqref{eq:eta-zero}.
\end{proof}

Define \(C_{k,i}^{(j)}\) by
\begin{equation}\label{eq:def-C-cofactor}
 \left(\sum_{i\geq0}B_{k,i}^{(j)}(w-w_k)^i\right)
 \frac{\eta_j(w)}{\eta_1(w)^j}
 =\sum_{i\geq0}C_{k,i}^{(j)}(w-w_k)^i
\end{equation}

The preceding proposition shows that, under the inductive bounds for the
coefficients of order less than \(i_0\), multiplication by
\(\eta_j/\eta_1^j\) does not affect whether the required estimate holds at
order \(i_0\).  This is the comparison used in
\cite[Proposition~6.21]{LTXZ2}.

\begin{corollary}\label{cor:cofactor-normalization}
Let \(1\leq j\leq r+s\) and \(0\leq i_0<m_n(k)\), and assume
\eqref{eq:minor-estimate} for minors of size smaller than \(n\), as in
\cref{prop:smaller-minors}.  In particular, retain the first inequality in
\eqref{eq:smaller-positive}, not only the last inequality. Suppose
\[
 v_p(B_{k,i}^{(j)})\geq
 \Delta_{k,\frac12d_k^{\new}-i}
 -\frac{k-2}{2}\left(\frac{d_k^{\Iw}}2-n\right)
 \qquad(0\leq i<i_0)
\]
If \(m_{n-j}(k)<m_n(k)\), then
\[
\begin{gathered}
 v_p(B_{k,i_0}^{(j)})\geq
 \Delta_{k,\frac12d_k^{\new}-i_0}
 -\frac{k-2}{2}\left(\frac{d_k^{\Iw}}2-n\right)\\
 \Longleftrightarrow\\
 v_p(C_{k,i_0}^{(j)})\geq
 \Delta_{k,\frac12d_k^{\new}-i_0}
 -\frac{k-2}{2}\left(\frac{d_k^{\Iw}}2-n\right)
\end{gathered}
\]
\end{corollary}

\begin{proof}
We use the coefficient comparison in
\cite[Proposition~6.21 and its proof in Sections~6.25--6.31]{LTXZ2}, replacing its
estimates by \cref{prop:smaller-minors,prop:eta-estimate}.
Indeed,
\[
 C_{k,i_0}^{(j)}-B_{k,i_0}^{(j)}
 =\sum_{i<i_0}B_{k,i}^{(j)}\eta_{(j),i_0-i}
\]
For \(t=i_0-i\) and \(i<m_{n-j}(k)\), one has \(q_t+t>i_0\).
Convexity gives
\[
 \Delta_{k,\frac12d_k^{\new}-q_t-t}-\Delta_{k,\frac12d_k^{\new}-q_t}
 \geq\Delta_{k,\frac12d_k^{\new}-i_0}-\Delta_{k,\frac12d_k^{\new}-i}
\]
so the assumed bound for \(B_{k,i}^{(j)}\) and
\eqref{eq:eta-positive} bound the corresponding summand below by
\[
 \Delta_{k,\frac12d_k^{\new}-i_0}
 -\frac{k-2}{2}\left(\frac{d_k^{\Iw}}2-n\right)
\]

If \(i\geq m_{n-j}(k)\), then \(q_t=m_{n-j}(k)\) and
\(q_t+t\leq i_0\), taking \(q_t=0\) in the zero-multiplicity case.
Adding the stronger bound in \eqref{eq:smaller-positive} to the corresponding bound in
\cref{prop:eta-estimate} gives
\[
\begin{split}
 v_p\bigl(B_{k,i}^{(j)}\eta_{(j),t}\bigr)\geq{}
 &\Delta_{k,\frac12d_k^{\new}-q_t-t}
 -\frac{k-2}{2}\left(\frac{d_k^{\Iw}}2-n\right)\\
 &+\frac12\bigl((\frac12d_k^{\new}-i)^2-(\frac12d_k^{\new}-q_t-t)^2\bigr)
\end{split}
\]
Replacing \(i\) by \(i_0\) weakens this bound; then
\eqref{eq:quadratic-gap} makes it at least
\[
 \Delta_{k,\frac12d_k^{\new}-i_0}
 -\frac{k-2}{2}\left(\frac{d_k^{\Iw}}2-n\right)
\]
with equality in this comparison if
\(q_t+t=i_0\). Thus every summand of the difference has the required
valuation, proving the equivalence. This also verifies explicitly
the zero-multiplicity case, which uses the additional estimates
proved above.
\end{proof}

\begin{theorem}\label{thm:minors}
Assume \(p\geq11\). Inequality \eqref{eq:minor-estimate} holds for all
row and column sets $\underline\zeta,\underline\xi$ of cardinality $n$,
every ghost zero $w_k$ of $g_n$, and every integer $0\leq i<m_n(k)$.
\end{theorem}

\begin{proof}
We use induction on \(n\). For \(n=1\), we have \(m_1(k)=1\), \(d_k^{\ur}=0\),
\(d_k^{\new}=d_k^{\Iw}\geq2\), so only
\(i=0\) occurs. For \(\underline\zeta=\{\zeta\}\) and \(\underline\xi=\{\xi\}\), \eqref{eq:minor-estimate} is equivalent to
\[
 v_p\bigl(U^\dagger_{\zeta,\xi}(w_k)\bigr)
 \geq\frac{k-2}{2}+\frac{\deg\bf{e}_\zeta-\deg\bf{e}_\xi}{2}
\]
If \(\zeta,\xi\leq d_k^{\Iw}\), then \cref{prop:classical-matrix}
gives \(U_p=-\operatorname{AL}\) on the classical subspace.
The entry is zero unless \(\zeta+\xi=d_k^{\Iw}+1\), in which case
\(v_p((U^\dagger_{\zeta,\xi}(w_k))={\deg \bf{e}_{\zeta}}\), and \(\deg\mathbf{e}_\zeta+\deg\mathbf{e}_\xi=k-2\) gives equality
in the displayed bound. The remaining cases are proved using \cref{lem:weak-hodge-block}.

Suppose the result is known for all smaller minors.
We prove \eqref{eq:B-target} for the fixed \(n\times n\) minor using induction on \(i\).
By \cref{prop:corank},
\(B_{k,i}^{(0)}=0\) for \(i<\max\{0,n-d_k^{\ur}-r-s\}\).
For
\(\max\{0,n-d_k^{\ur}-r-s\}\leq i_0<m_n(k)\), set
\[
 j_0=n-d_k^{\ur}-i_0-1
\]
Then \(0\leq j_0\leq r+s-1\). Apply
\eqref{eq:cofactor-congruence} with \(\ell=0\) and
\(\eta=\eta_1\), divide by the common normalization, and take
the coefficient of order \(i_0\) give
\begin{equation}\label{eq:C-induction}
 B_{k,i_0}^{(0)}
 =\sum_{j=j_0+1}^{r+s}(-1)^{j-j_0-1}
       \binom{j-1}{j_0}C_{k,i_0}^{(j)}
\end{equation}
For every \(j>j_0\),
\[
 m_{n-j}(k)\leq\max\{0,n-j-d_k^{\ur}\}
 \leq i_0<m_n(k)
\]
Thus \cref{prop:smaller-minors} bounds \(B_{k,i}^{(j)}\)
for all \(i\leq i_0\), including when \(m_{n-j}(k)=0\).
Its stronger estimates verify the hypotheses of
\cref{cor:cofactor-normalization}, which gives the target bound
for \(C_{k,i_0}^{(j)}\).
Since the coefficients in \eqref{eq:C-induction} are integers,
the same bound holds for \(B_{k,i_0}^{(0)}\).
This proves every required order using only smaller minors and
closes the induction.
Finally, \eqref{eq:B-versus-A} and the definition of
\(\Delta'_k\) identify this bound with \eqref{eq:minor-estimate}.
\end{proof}

\begin{remark}\label{rem:p-bound}
The assumption $p\geq11$ is used in the determinant estimate
\cref{prop:det-mahler} and in the \(\Delta_{k,l}-\Delta'_{k,l}\) estimate \cref{lem:convex-hull-difference}.
\end{remark}

\subsection{Proof of the main theorem}

We now pass from estimates for finite minors to the Fredholm coefficients.
There is one additional issue: integrality of the interpolation remainder
only gives one inequality between Newton polygons.  Equality at the vertices
requires that remainder to be a unit.  We prove this separately; the argument
used in the proof of \cite[Proposition~4.4]{LTXZ2} does not give the required
strict inequalities for the two endpoint characters in the present case.

For each \(n\), put
\[
 A_{k,i}^{(n)}=(-1)^n
 \sum_{\#\underline\xi=n}
 A_{k,i}^{(\underline\xi\times\underline\xi)}
\]
The weak Hodge bound makes this sum converge.  Summing
\eqref{eq:lagrange-minor} over all principal minors gives
\begin{equation}\label{eq:lagrange-cn}
 c_n(w)=\sum_{m_n(k)>0}
 \left(\sum_{i=0}^{m_n(k)-1}A_{k,i}^{(n)}(w-w_k)^i\right)
 g_{n,\widehat k}(w)+h_n(w)g_n(w)
\end{equation}
For principal minors, the degree correction in \eqref{eq:minor-estimate}
is zero. Thus \cref{thm:minors} gives the same bound without that correction
for the coefficients \(A_{k,i}^{(n)}\), in the notation of
\cite[Notation~4.3 and equation~(4.3.1)]{LTXZ2}.  Here convergence
also holds after taking each of the finitely many Taylor coefficients:
on a sufficiently small disc around its center, division by the fixed
nonvanishing polynomial \(g_{n,\widehat k}\) and each fixed Taylor
coefficient functional are continuous.  The row bound makes the
principal-minor sum converge in the corresponding analytic norm.

The characteristic-series construction in
\cite[Section~2.10 and Lemma~2.14]{LTXZ1} gives
\(c_n\in\cO[\![w]\!]\); it is independent of the genericity
condition on the Serre weight. Formal Weierstrass division by
the monic polynomial \(g_n\), whose roots lie in \(p\cO\),
has a quotient in \(\cO[\![w]\!]\).
Its polynomial remainder has the same prescribed Taylor
coefficients as the interpolation sum in \eqref{eq:lagrange-cn},
so uniqueness of division identifies its quotient with \(h_n\).
Consequently
\[
 h_n\in\cO[\![w]\!]\cap\cO\langle w/p\rangle
\]
The formal integrality is needed below: a unit value at a point
of \(p\cO\) then implies a unit constant term.

We prove that the remainder is a unit by the method of
\cite[Proposition~4.4]{LTXZ2}, specifying the auxiliary
weight and checking the strict slope inequalities for the
endpoint characters.

\begin{lemma}\label{lem:unit-remainder}
For every \(n\geq1\), except for the already separated case
\(n=1,s_\varepsilon=0\), the function \(h_n(w)\) in
\eqref{eq:lagrange-cn} is a unit.  In the exceptional case
\(g_1=1\) and \(h_1=c_1\) is a unit by \cref{lem:first-coefficient}.
\end{lemma}

\begin{proof}
We use the comparison with opposite characters in
\cite[Proposition~4.4]{LTXZ2}. We must choose the auxiliary
weight so that the required slope inequalities are strict,
including for \(s_\varepsilon=0,p-2\).

Choose \(r\in\{1,2\}\setminus\{s_\varepsilon\}\), and set
\[
 k=2+s_\varepsilon+(n-1)(p-1)+r
\]
Then \(r\ne0,p-2\), and \eqref{eq:iw-general} gives
\[
 d_k^{\Iw}(\varepsilon(1\times\omega^{2-k}))=n,
 \qquad k\not\equiv k_\varepsilon\pmod{p-1}
\]
For the theta and opposite characters in
\cref{prop:compatibility}, we have
\[
 \varepsilon'=\varepsilon(\omega^{k-1}\times\omega^{1-k}),
 \quad s_{\varepsilon'}=p-2-r>0,\qquad
 \varepsilon''=\omega^{-r}\times\omega^r,
 \quad s_{\varepsilon''}=r>0
\]
For \(\chi=\varepsilon,\varepsilon''\), let
\(P_k^{\Iw,(\chi)}\) be the classical characteristic polynomial.
The theta exact sequence gives
\[
 C^{(\chi)}(w_k,T)
 =P_k^{\Iw,(\chi)}(T)\,
 C^{(\chi(\omega^{k-1}\times\omega^{1-k}))}
      (w_{2-k},p^{k-1}T)
\]
Both classical polynomials have degree \(n\).

The lower bound for Fredholm coefficient ordinates already follows
from \cref{prop:lagrange-term-bound,thm:minors} and the formal
integrality of the remainders; it does not require their being units.
At an integral weight, each ghost root factor has valuation at least
one. By \cref{prop:degree-increment}, a character of parameter
\(s>0\) has \(\deg g_i\geq is\), so all its actual slopes are
positive; for \(s=0\) they are nonnegative.
Both classical spectra are therefore nonnegative, and
Atkin--Lehner duality bounds them above by \(k-1\).

For \(\chi=\varepsilon\), the theta parameter \(p-2-r\) is
positive, so every theta slope is strictly greater than \(k-1\).
For \(\chi=\varepsilon''\), if \(s_\varepsilon>0\), duality
makes every classical slope strictly less than \(k-1\).
If \(s_\varepsilon=0\), its theta parameter is
\(\{r+1-k\}=p-2>0\), so the theta slopes are strictly greater
than \(k-1\). Thus, for both characters, every theta slope
is strictly larger than every classical slope.
It follows that the \(n\)-th Fredholm coefficient has the
valuation of the classical determinant. Explicitly, if
\(\alpha<\beta\) are the largest classical and smallest theta
slopes, then in the coefficient product
\((\sum_{i=0}^n a_iT^i)(\sum_{j\geq0}b_jT^j)\),
\[
 v_p(a_{n-j}b_j)\geq v_p(a_n)+j(\beta-\alpha)>v_p(a_n)
 \qquad(1\leq j\leq n)
\]
Consequently Atkin--Lehner duality gives
\begin{equation}\label{eq:actual-valuation-sum}
 v_p(c_n^{(\varepsilon)}(w_k))
 +v_p(c_n^{(\varepsilon'')}(w_k))=n(k-1)
\end{equation}
The ghost identities in \cref{prop:compatibility}(1)--(2) give
\begin{equation}\label{eq:ghost-valuation-sum}
 v_p(g_n^{(\varepsilon)}(w_k))
 +v_p(g_n^{(\varepsilon'')}(w_k))=n(k-1)
\end{equation}
and put both coefficient points on their ghost polygons.
Comparing \eqref{eq:actual-valuation-sum} and
\eqref{eq:ghost-valuation-sum} with the already established lower
bounds forces
\(v_p(c_n^{(\chi)}(w_k))=v_p(g_n^{(\chi)}(w_k))\) for both
characters, exactly as in the cited proof.

For the original character, the index \(n\) is a strict ghost
vertex: the opposite parameter \(r>0\) makes the last classical
ghost slope strictly less than \(k-1\), while the theta parameter
\(p-2-r>0\) makes the first theta ghost slope strictly greater
than \(k-1\).
By \cref{prop:lagrange-term-bound}, every truncated term in
\eqref{eq:lagrange-cn} is strictly higher at \(w_k\).
The equality of coefficient valuations then gives
\(v_p(h_n(w_k))=0\).
Since \(h_n\in\cO[\![w]\!]\) and \(w_k\in p\cO\), its constant
term is a unit. The case separated in the statement follows
from \cref{lem:first-coefficient}.
\end{proof}

\begin{theorem}\label{thm:local-ghost}
Assume \(p\geq11\).  Let \(\widetilde H\) be any primitive projective-augmented module of type
\(\bar\rho\), with the unramified normalization fixed in \cref{prop:Pi}.
Then, for
every $w_\star\in\mathfrak m_{\Cp}$,
\[
 \NP\bigl(C_{\widetilde H}^{(\varepsilon)}(w_\star,-)\bigr)
 =\NP\bigl(G^{(\varepsilon)}(w_\star,-)\bigr)
\]
\end{theorem}

\begin{proof}
We conclude as in \cite[Proposition~4.4]{LTXZ2}.
By \cref{prop:lagrange-term-bound,thm:minors}, each truncated term in
\eqref{eq:lagrange-cn} lies on or above the ghost Newton polygon.
By \cref{lem:unit-remainder}, the remaining term has ordinate
\(v_p(g_n(w_\star))\). Hence
\[
 \NP\bigl(C_{\widetilde H}^{(\varepsilon)}(w_\star,-)\bigr)
 \geq\NP\bigl(G^{(\varepsilon)}(w_\star,-)\bigr)
\]
At a strict ghost vertex, the truncated terms have strictly larger
valuation, so
\(v_p(c_n(w_\star))=v_p(g_n(w_\star))\).
The ghost polygon has unbounded vertex abscissae: with
\(\rho=\min\{v_p(w_\star),1\}>0\), one has
\(v_p(g_n(w_\star))\geq\rho\deg g_n\) whenever the coefficient is
nonzero, and \(\deg g_n\) grows quadratically by
\cref{prop:degree-increment}. Equality at these vertices proves
equality of the polygons.
For \(s_\varepsilon=0\), \cref{lem:first-coefficient} gives the
common initial slope zero. This proves \cref{thm:intro}.
\end{proof}


\begin{thebibliography}{99}
\bibitem{Col10}
P. Colmez,
\emph{Repr\'esentations de \(\mathrm{GL}_2(\mathbb Q_p)\) et
\((\varphi,\Gamma)\)-modules},
Ast\'erisque 330 (2010), 281--509.

\bibitem{BP}
J. Bergdall and R. Pollack,
\emph{Slopes of modular forms and the ghost conjecture},
Int. Math. Res. Not. IMRN 2017 (2017), no.~4, 1125--1144.

\bibitem{HP}
Y. Hu and V. Pa\v sk\=unas,
\emph{On crystabelline deformation rings of
\(\operatorname{Gal}(\overline{\mathbb Q}_p/\mathbb Q_p)\)},
Math. Ann. 373 (2019), 421--487.

\bibitem{LTXZ1}
R. Liu, N. X. Truong, L. Xiao, and B. Zhao,
\emph{A local analogue of the ghost conjecture of Bergdall and Pollack},
Peking Math. J. 7 (2024), 247--344.

\bibitem{LTXZ2}
R. Liu, N. X. Truong, L. Xiao, and B. Zhao,
\emph{Slopes of modular forms and geometry of eigencurves},
arXiv preprint arXiv:2302.07697.

\bibitem{PasBM}
V. Pa\v sk\=unas,
\emph{On the Breuil--M\'ezard conjecture},
Duke Math. J. 164 (2015), 297--359.

\end{thebibliography}
\end{document}